\documentclass[11pt]{amsart}
\usepackage{enumerate,amsmath,amssymb, mathrsfs,stmaryrd}

\usepackage{latexsym, amssymb,enumerate}
\usepackage{appendix}

\newcommand{\ga}{\alpha}
\newcommand{\gb}{\beta}
\newcommand{\gc}{\gamma}
\newcommand{\gd}{\delta}

\newcommand{\gra}{\nabla}
\newcommand{\de}{\partial}
\newcommand{\bpf}{\begin{proof}}
\newcommand{\epf}{\end{proof}}
\newcommand{\beq}{\begin{equation}}
\newcommand{\eeq}{\end{equation}}
\newcommand{\beqn}{\begin{eqnarray*}}
\newcommand{\eeqn}{\end{eqnarray*}}

\newcommand\R{\mathbb{R}}

\def\circledwedge{\setbox0=\hbox{$\bigcirc$}\relax \mathbin {\hbox
to0pt{\raise.5pt\hbox to\wd0{\hfil $\wedge$\hfil}\hss}\box0 }}

\def\R{{\mathfrak R}}

\newtheorem{prop}{Proposition}[section]
\newtheorem{theo}[prop]{Theorem}
\newtheorem{lemm}[prop]{Lemma}
\newtheorem{coro}[prop]{Corollary}
\newtheorem{rema}[prop]{Remark}

\def\begeq{\begin{equation}}
\def\endeq{\end{equation}}
\def\p{\partial}

\def\R{\mathbb R}

\def \ds{\displaystyle}
\def \vs{\vspace*{0.1cm}}

\def\div{{\rm div\,}}

\def\odot{\setbox0=\hbox{$\bigcirc$}\relax \mathbin {\hbox
to0pt{\raise.5pt\hbox to\wd0{\hfil $\wedge$\hfil}\hss}\box0 }}

\numberwithin{equation} {section}

\def\tilde{\widetilde}

\begin{document}

\title[Conformally compact Einstein manifolds] {Compactness of conformally compact Einstein manifolds in dimension 4}

\author{Sun-Yung A. Chang}
\address{Princeton University, Princeton, NJ}
\email{chang@math.princeton.edu}

\author{Yuxin Ge}
\address{Universit\'e Toulouse\\
IMT,
Universit\'e Toulouse 3, \\118, route de Narbonne
31062 Toulouse, France}
\email{yge@math.univ-toulouse.fr}

\begin{abstract} In this paper, we establish some compactness results of conformally compact Einstein metrics  on $4$-dimensional manifolds. Our results were proved under assumptions on the behavior of some local and non-local conformal invariants, on the compactness of the boundary metrics at the conformal infinity, and  on the topology of the manifolds. 
\end{abstract}

\thanks{Research of Chang is supported in part by the NSF grant MPS-1509505.}

\subjclass[2000]{}

\keywords{}

\maketitle




\section{Introduction}
In this paper we study the compactness of a set of conformally compact Einstein metrics on some manifold $X$
of dimension four with three dimensional boundary $\p X$. 
We introduce a class of conformally invariant quantities on $X$ and on its boundary. We aim to establish a compactness result that under 
suitable conditions on the size of these invariants, the compactness of a class of metrics on the boundary would imply the compactness of the corresponding conformal structures in the interior.
To be more precise, we consider on $X = X^4$ a set of conformally compact Einstein metrics $g = \rho^2g^+$ where $g^+$ is an asymptotically hyperbolic Einstein metric on $X$ 
and $\rho$ is a smooth defining function of the boundary such that $g$ extends to a smooth metric on the closure of $X$. 
To state our results, we first introduce a class of $2$-tensor $S$ on the boundary which is pointwisely conformally invariant. The definition of $S$ is motivated by the 
Gauss-Bonnet formula on $4$-manifolds with boundary $(X,\p X,g)$, with $g$ defined on $X$ and extended smoothly to $\p X$. 
On such a manifold $(X, \p X, g)$, we consider the functional on $(X,\p X,g)$
$$
g\to \int_X |W|^2_gdv_g+8\oint_{\p X}W_{\alpha n \beta n}L^{\alpha\beta}d\sigma_g,
$$
where $L$ is the second fundamental form and $W$ the Weyl tensor, $n$ is the outwards unit normal vector on the boundary,  
the greek indices $\alpha,\beta,\cdots$ represent the tangential indices letter and $i,j,k\cdots$ are the full indices. 
Both terms of the functional are conformally invariant, i.e., under conformal change of metric $\hat{g}= e^{2w}g$ for a smooth function $w$ on the closure of $X$,  the value of the above functional stays the same. Critical metrics of the functional under variation of the metric $g$ satisfy in the interior the well-known condition $B_{ij }= 0$ where $B_{ij}$ is the Bach tensor, and on the boundary
(more details in Section 2 below) 
$$
 S_{\alpha\beta} := \nabla^i W_{i\alpha n\beta} + \nabla^i W_{i \beta n\alpha} - \nabla^n W_{n\alpha n\beta}  + \frac43H W_{n\alpha n\beta}=0,
$$
where $H$ denotes the mean curvature. When the boundary is totally geodesic (i.e. $L = 0$), as in the case of a conformally compact Einstein metric, it turns out
$$
S_{\alpha\beta} = \frac12 \nabla^nR_{\alpha\beta}  - \frac{1}{12} \nabla^nR g_{\alpha\beta},
$$
where $R_{\alpha\beta}$ (or $Ric$) is the Ricci tensor and $R$ is the scalar curvature. $S$  is a non-local tensor for conformally compact Einstein manifolds and is a conformal 
invariant in the sense of Lemma \ref{lemma2.0} below, that is $S(e^{2w}g)=e^{-w}S(g)$.\\

For any four-dimensional Riemannian manifold $(X^4 , g)$ with or without boundary, the $Q$-curvature $Q_4$ is defined as:
\beq
\label{q4}
Q_4:=-\frac16 \triangle R-\frac12|Ric|^2+\frac16 R^2.
\eeq
In the study of conformal geometry, $Q$ is naturally related to a $4$th-order differential operator, called the Paneitz (\cite{Pa}) 
operator (which is a special case of some general class of GJMS{\cite{GJMS} operators) and is a 4-th order
generalization of the conformal Laplacian operator defined as:
\beq
\label{defp4}
P:=( \triangle)^2-\div[(\frac23 R g-2 Ric)\nabla]
\eeq
Throughout this paper, we denote by $[g]=\{e^{2w}g|\; w:X\to\R \mbox{ is a regular function}\}$ the class of metrics conformal to
 $g$. Under the conformal change $g_w:=e^{2w}g$, the associate  $Q$-curvature for $g_w$ metric, denote
by $Q_4(g_w)$, is related to $Q_4(g)$ by the PDE:
\beq
\label{p4}
P(g) w+Q_4(g)=Q_4(g_w)e^{4w}.
\eeq
The 4th order operator Paneitz operator $P$ and its corresponding $Q$ curvature have been extensively studied in the recent literature, here we will just cite a few of them (\cite {BCY}, \cite {CY}, 
\cite {GZ}, \cite{FG}, \cite{DM}, \cite{GM}). 

There are two non-local curvature tensors of order three defined on the boundary $\p X$; one is the $T$ curvature defined on the boundary of any compact four manifolds (see \cite{CQ1}, 
\cite{ChangQingYang2006bis}), the other is 
the conformally invariant $Q_3$ curvature defined on the boundary of confomally compact asymptotically hyperbolic manifolds
in (\cite {FG}). Without going into details of their respective definitions, here we will just cite that the result that in the special case when the boundary is totally 
geodesic, it turns out the curvatures $T$  and $Q_3$ agree (see \cite{ChangQingYang2006bis}, Lemma 2.2) and in this case $$ T:=\frac{1}{12}\frac{\p R}{\p n}$$
On a four manifold $(X,\p X,g)$ with boundary, the
 $Q$-curvature on $X$ and the
$T $-curvature on the boundary $\p X$ are related by the Chern-Gauss-Bonnet formula \cite{CQ1} 
$$
\chi(X)=\frac1{32\pi^2}\int_X( |W|^2+4Q) dvol +\frac1{4\pi^2}\oint_{\p X} (\mathscr{L}  + T) d\sigma
$$
where $\chi(X)$ is the Euler characteristic number of $X$ and $\mathscr{L} d \sigma$ is a pointwise conformal
invariant on $\p X$.

As a consequence,  $\int_X Q+2 \oint_{\p X}T$ is independent of the choice of metrics in the conformal class $[g]$ 
since $|W|^2 dvol$ is also a pointwise conformal invariant term on X. \\

Let us now recall briefly the Yamabe invariants on compact $4$-manifolds $(X,g)$ with boundary $\p X$. We consider the Yamabe energy functional
$$
Y(g):=\frac 16\int_X R_g +\int_{\p X}H_g, 
$$
where $R_g$ is the scalar curvature of the metric $g$ and $H_g$ is the mean curvature on the boundary $\p X$. Note that when $(X, g)$ is with totally geodesic boundary $\p X$ (or more generally, 
the mean curvature vanishes on the boundary)
 for a conformal metric $\tilde g=U^2 g\in [g]$, we can rewrite
$$
Y(\tilde g)=\int_X |\nabla U|^2 +\frac16 R_{g} U^2.
$$
We now denote the first Yamabe constant as
$$
Y(X,M,[g]):= \inf_{\tilde g\in [g]}\frac{Y(\tilde g)}{vol(X,\tilde g)^{1/2}}=\inf_{U\in C^1\setminus\{0\}}\frac{\ds\int_X |\nabla U|^2 +\frac16 R_{g} U^2}{\ds\left(\int_X U^4\right)
^{\frac{1}{2}}},  
$$
and the second Yamabe constant as 
$$
Y_b(X,M,[g]) := \inf_{\tilde g\in [g]}\frac{Y(\tilde g)}{vol(\p X,\tilde g)^{2/3}}=\inf_{U\in C^1\setminus\{0\}}\frac{\ds\int_X |\nabla U|^2 +\frac16 R_{g} U^2}
{\ds\left(\oint_M U^3\right)^{\frac{2}{3}}}.
$$
Here $\ds\oint$ denotes the integral on the boundary, and $vol(X,\tilde g)$ (resp. $vol(\p X,\tilde g)$) is the volume of $X$ (resp. $\p X$) under the metric $\tilde g$.\\

On ${(X, \p X, g^+)}$ a four-dimensional oriented manifold, we say the manifold is conformally compact
if there exists some defining function $\rho >0$ on $X$ so that $\rho^2 g^{+}$ is a compact metric defined
on $ \bar X =: X \cup \p  X.$ In the case when
$g^{+} $ is a Poincare Einstein metric which we normalized
so that $Ricci_{g^+} = - 3 {g^+} $, we say that ${(X, \p X, g^+)}$
  is  a conformally compact Einstein manifold (abbreviated as CCE) and we say $\p X$
the conformal infinity of $X$. Note that since the choice of the defining functions
are by no means unique but a multiple of each other,  their corresponding compactified metrics are conformal to each other and so  are their restriction to $\p X$. Thus the boundary metric on $\p X$ is unique up to a conformal class.

Throughout this paper,  we will choose a special compactification of $g^+$.
This special compactification was first introduced in the paper by 
Fefferman-Graham \cite{FG} 
(
to study
the renormalized volume of CCE manifolds with odd dimensional boundary). Here
we will restrict our attention to the special case when $\p X $ is of dimension 
3. To define this special compactification, given any boundary metric
$h \in [\rho^2 g^{+}|_{\p X}]$, one solves the partial differential equation
\beq 
\label{eq1.1}
-\triangle_{g^+}w=3
\eeq
We denote the metric $g = e^{2w} g^{+}$ with $g|_{\p X}=h$ and we name it as the Fefferman-Graham (abbreviated FG) compactification with the boundary $h$. Later in this paper, we will further derive other relevant properties (e.g. Lemma \ref{lemma2.3}, Lemma \ref{lem4.1}) of this
compactification; but here we will point out one key property which leads us to think the metric is the most suitable representative metric among the
conformal compactification metrics of $g^{+}$. The property, which was
pointed out and applied to derive a formula of the renormalized volume in the earlier paper by Chang-Qing-Yang \cite{ChangQingYang2006bis}; 
is that, for this choice of compactification, the $Q$-curvature $Q_4(g)$ on $X$ vanishes identically.
To see this, we notice that for the
Einstein metric $g^+$, by our normalization of it being Poincare Einstein, the Paneitz operator can be written as
\beq
\label{p4}
P_{g^+}=\triangle_{g^+}^2+2\triangle_{g^+}, \,\,\,\,\,\,  Q_4(g^+) = 6,  
\eeq
so that applying equation
(\ref{p4}), 
we find
\beq
\label{q4zero}
Q_4(g)= e^{-4w}(P_{g^+} w +6)\,= \, 0.
\eeq

In this paper, we will
always choose the Yamabe metric on the boundary as representative in the conformal infinity $[g|_{TM}]$ and take the corresponding FG compactification.\\

Through the whole paper, we assume $X$ is $4$-dimensional oriented CCE, and the boundary $M=\p X=\mathbb{S}^3$ is $3$-sphere and the boundary Yamabe metric 
$\hat{g}$ in the conformal infinity is non-negative type, that is, the scalar curvature  of $\hat{g}$ is an non-negative constant; and we denote
denote the corresponding FG compactification. \\

Our main compactness results are as follows.

\begin{theo} 
\label{maintheorem}
Let  $\{X, \p X=\mathbb{S}^3, g_i^+\}$ be a family of $4$-dimensional oriented CCE on $X$ with boundary $\p X$.  We assume
the boundary Yamabe metric $\hat{g_i}$ in conformal infinity is
of non-negative type and denote
 $g_i$ be the corresponding FG compactification.
Assume  
\begin{enumerate}
\item The boundary Yamabe metrics  $\hat{g_i}$ form a compact family in $C^{k+3}$ norm with $k \ge 2$;
\item There is no concentration of $S$-tensor in $L^1$ norm for the $g_i$ metric on $\p X$ 
in the following sense, 
$$
\lim_{r\to 0}\sup_i\sup_x\oint_{B(x,r)} |S_i|=0
$$
\item $H_1(X,\mathbb{Z})=H_2(X,\mathbb{Z})=0$.
\item there exists some positive constant $C_1>0$ such that the first Yamabe constant for the compactified metric $g_i$ is bounded uniformly from below by $C_1$
$$
Y(X,M,[g_i])   \ge C_1;
$$
\item there exists some positive constant $C_2>0$ such that the second Yamabe constant for the metric $g_i$ is bounded uniformly from below by $C_2$
$$
Y_b(X,M,[g_i]) \ge C_2;
$$
\end{enumerate}
Then, the family of the Fefferman-Graham compactified metrics $(X,g_i)$  is compact in $C^{k+2,\alpha}$ norm for any $\alpha\in (0,1)$ up to a diffeomorphism fixing the boundary.
\end{theo}

As consequences of the main theorem, we can establish the following corollaries.

\begin{coro}
\label{maincorollary0}
Under the assumptions (1) and (3)-(5) as in Theorem \ref{maintheorem}, suppose  $\{S_i\}$ is a relatively weakly compact family in $L^1$, that is, the closure of $\{S_i\}$ is compact in the weak topology generated by all linear continuous maps on  $L^1$. 
Then the family of the Fefferman-Graham compactified metrics $(X,g_i)$ is compact in $C^{k+2,\alpha}$ norm for any $\alpha\in (0,1)$, up to a diffeomorphism fixing the boundary, provided
$ k \ge 2$.
\end{coro}

\begin{coro}
\label{maincorollary0bis}
Under the assumptions (1) and (3)-(5) as in Theorem \ref{maintheorem}, suppose  there is some constant $C_4>0$ such that for some $+\infty\ge q>1$ one has
$$
\oint |S_i|^q\le C_4.
$$
Then the family of the Fefferman-Graham compactified metrics $(X,g_i)$ is compact in $C^{k+2,\alpha}$ norm for any $\alpha\in (0,1)$, up to a diffeomorphism fixing the boundary, provided
$ k \ge 2$.
\end{coro}

\begin{coro}
\label{maincorollary01}
Under the assumptions  (1) and (3)-(5) as in Theorem \ref{maintheorem},  \\
\begin{itemize}  
\item[](2') \; \;there exists some small constant $\varepsilon_1>0$ (depending on $C_1$,$C_2$ and $C^{k+3}$ norm bound of the boundary metric and also on the topology of $X$) such that for all $i$ one has 
$$
\oint |S_i|\le \varepsilon_1,
$$
\end{itemize} 
then the family of the Fefferman-Graham compactified metrics $(X,g_i)$ is compact in $C^{k+2,\alpha}$ norm for any $\alpha\in (0,1)$ up to a diffeomorphism fixing the boundary, provided
$ k \ge 2$.
\end{coro}

\begin{rema}

In the statement of Theorem \ref{maintheorem}, the conditions (4) and (5) are conformally invariant conditions but condition (2) is not. A more natural conformally invariant condition would be the uniform boundedness of the $L^1$ norm of the $S$ tensor for the family of metrics, but the authors are so far not able to establish Theorem \ref{maintheorem} under this more natural assumption. Instead we can establish the compactness result under the stronger assumption (2), which implies the uniform bound of  $L^1$ norm of the $S$ tensor;
or we can establish the compactness result under the conformally invariant condition $(2')$. 
We remak that, by Dunford-Pettis Theorem, the condition (2) in Theorem \ref{maintheorem} is equivalent to the compactness of $S$-tensor under the weak topology in $L^1$. 
\end{rema}

\begin{rema}
For the unit ball $X=B^4$ (more generally, when $X$ is a homology sphere removed a 4-ball), the topological 
conditions (3) in Theorem \ref{maintheorem} are satisfied.

\end{rema}

Another version of our main theorem is to replace condition on $S$ tensor by the curvature tensor $T$.

\begin{theo} 
\label{maintheorem1}
Under the assumptions  (1) and (3)-(5) as in Theorem \ref{maintheorem}, suppose 
the $T $ curvature on the boundary $T_i=\frac{1}{12}\frac{\p R_i}{\p n}$  satisfies the following condition
$$
\liminf_{r\to 0}\inf_i\inf_{x\in M}\oint_{B(x,r)}T_i\ge 0
$$
Then the family of the Fefferman-Graham compactified metrics $(X,g_i)$ is compact in $C^{k+2,\alpha}$ norm for any $\alpha\in (0,1)$ up to a diffeomorphism fixing the boundary, provided $k\ge 5$.
\end{theo}

\begin{rema}
In Theorem  \ref{maintheorem} and Theorem  \ref{maintheorem1}, if  the boundary Yamabe metric $\hat{g}$ in conformal infinity is of positive type, we can drop the condition $H_1(X,\mathbb{Z})=0$ in the condition (3). To see this, by a result due to Wittten and Yau \cite{WY}, we know, under the assumptions that the conformal infinity is of positive type and that  the conformal infinity is simply connected, then $H_1(X,\mathbb{Z})=0$.
\end{rema}

We remark we have assumed the stronger regularity on Theorem \ref{maintheorem1} for $k \ge 5$ than
$k\ge 2$ due to  a technical reason that in the proof of the  theorem we have taken a power series
expansion of the metric $g_i$ for up to order 7  (see (\ref{Laplaceinfty}) in the proof of Lemma \ref{T=S} below).

Some parallel direct consequences of Theorem \ref{maintheorem1} can be stated as follows.

\begin{coro}
\label{maincorollary1.0}
Under the assumptions  (1) and (3)-(5) as in Theorem \ref{maintheorem}, suppose  $\{\max(-T_i,0)\}$ is a relatively weakly compact family in $L^1$. Then, the family of the Fefferman-Graham compactified metrics $(X,g_i)$ is compact in $C^{k+2,\alpha}$ norm for any $\alpha\in (0,1)$ up to a diffeomorphism fixing the boundary, provided $k\ge 5$.
\end{coro}

\begin{coro} 
\label{maincorollary1}
Under the assumptions  (1) and (3)-(5) as in Theorem \ref{maintheorem}, suppose  there is some constant $C_5>0$ and some $+\infty\ge q>1$
independent of $i$ such that for all $i$ one has
$$
\oint_M (\max(-T_i,0))^q\le C_5
$$
Then the family of the Fefferman-Graham compactified metrics $(X,g_i)$ is compact in $C^{k+2,\alpha}$ norm for any $\alpha\in (0,1)$, up to a diffeomorphism fixing the boundary, 
provided $k\ge 5$.
\end{coro}

\begin{coro} 
\label{maincorollary1.1}
Under the assumptions (1) and (3)-(5) as in Theorem \ref{maintheorem},  there exists some small constant $\varepsilon_2>0$ (possibly depending on $C_1$,$C_2$ and $C^{k+3}$ norm bound of the boundary metric and also on the topology of $X$)  such that  if  for all $i$ one has
$$
\oint_M (\max(-T_i,0))\le \varepsilon_2,
$$
then the family of the Fefferman-Graham compactified metrics $(X,g_i)$ is compact in $C^{k+2,\alpha}$ norm for any $\alpha\in (0,1)$, up to a diffeomorphism fixing the boundary, provided $k\ge 5$.
\end{coro}

\begin{rema}
 Although the non-local terms $S$ and $T$ appears to be independent from each other in their
 definitions,  it turns out their behavior are coupled in the setting of conformally compact
Einstein manifolds. 
As we will show in section \ref{sb4.2} of the paper, for the limiting metric of a class of conformal compact Einstein manifolds, when the Yamabe 
invariant on the boundary is non-negative, the limiting metric of the blow-up metrics $T\equiv 0$ is equivalent to $S\equiv 0$.
\end{rema}

The paper is organized as follows: in section 2, we provide some background and some basic calculations; in section 3, we prove a $\varepsilon$-regularity result for  our 
$Q_4$ flat metrics.
In
section 4, which is the main part of the paper, we do the blow-up analysis. First we rule out the boundary blow up by our boundedness assumptions on the boundary metrics and 
the condition on the $S$-tensors or $T$ curvature in section \ref{sb4.2}, 
we then rule out the interior blow up based on our assumption that the $\p X $ is topologically $S^3$ and
the condition (3) in the statement  of the theorems by some topological arguments in section \ref{sb4.3}.
 This permits us to establish the uniform boundedness of the
$L^2$ norm of
the curvature tensor of the sequence of
Fefferman-Graham's compactified metrics.
From there, we apply the $\varepsilon$-regularity argument to jerk
 up the order of the regularity in section \ref{sb4.4}.
Finally in section 5, 
We estimate some geometric quantity including the diameter of the metrics and show they are uniformly bounded 
and establish the desired compactness results claimed in section 1.

\subsection*{Acknowledgement}
The authors were aware that in the paper \cite{anderson1} by M. Anderson,  he had asserted similar compactness results
in the CCE setting 
under no assumptions on the (analogue of the) nonlocal tensor $S$. We have difficulty understanding some key estimates in his arguments.

In both Theorem \ref{maintheorem} and \ref{maintheorem1}, the topological assumption conditions (3) are only used to establish 
that there is no interior blow up. In the earlier version of this paper, the authors have stated both
these two theorems without the additional assumption that the boundary of the four manifold $X$ is $S^3$ as is 
in the current version of the paper; in the proof we had quoted a result of M. Anderson (the claim after
Proposition 3.10 in \cite{anderson1}, see also the
result of M. Anderson in another paper \cite{anderson2} Lemma 6.3) 
to establish the argument of no interior blow up.  It was pointed out to us
by the referee that this
result of Anderson was questioned in the
 recent work of Akutagawa-Endo-Seshadri \cite{A}. Inspired by the
proof in the paper of Akutagawa-Endo-Seshadri, we have in this version applied a result of Chrisp-Hillman  \cite{CH} to 
establish the argument of no interior blow up;
 under the additional assumption that the boundary is topologically $S^3$.  

The authors have worked on the paper over a long period of time, and over the period, they have consulted with a number of colleagues
on different parts of the paper. They are grateful to all of them, in particular to the consultations with Olivier Biquard, Robin Graham, Jie Qing and Paul Yang. 

The authors are also grateful to the referee for pointing out the question raised in the paper of Akutagawa-Endo-Seshadri \cite{A}; 
for suggesting the elliptic iteration argument (used
in section \ref{sb4.4} of this paper) to improve the
order of the regularity and also for making many other useful comments concerning
 the presentation of the paper. 


\section{Some basic calculus on the boundary}
We use  the greek indices $\alpha,\beta,\gamma\cdots$ to represent the tangential indices, $n$ is the unit normal vector on the boundary and letter $i,j,k\cdots$ are full indices. $A=\frac12(Ric -\frac{R}{6}g)$ is the Schouten tensor in $X$ and $W$ is the Weyl tensor in $X$. Denote by $\hat{\nabla}$ the connection on the boundary $M$ and by $\nabla$ the connection in  $X$. Similarly, we denote by $\hat{Ric},\hat{R}$ the Ricci curvature and scalar curvature on the boundary $M = \p X$, and $\hat{A}=\hat{Ric} -\frac{\hat{R}}{4}\hat{g}$ is the Schouten tensor on the boundary $M$. Recall the Cotten tensor in $X$ (resp. on $M$) is defined by $C_{\alpha\beta\gamma}=A_{\alpha\beta,\gamma}-A_{\alpha\gamma,\beta} $ (resp. $\hat{C}_{\alpha\beta\gamma}=\hat{A}_{\alpha\beta,\gamma}-\hat{A}_{\alpha\gamma,\beta} $). Moreover, we denote by $L$ the second fundamental form on $M$ and $H$ the mean curvature of the boundary $M$.

Let $T_{i_1 \cdots i_k}$ be a tensor defined on $X$. Then  the Ricci identity
\beq
\label{Ricci}
T_{i_1 \cdots i_k, j l}= T_{i_1 \cdots i_k, l j} - 
\sum_{s=1}^k  R_{m i_s l j} T_{i_1 \cdots i_{s-1} m i_{s+1} \cdots i_k}
\eeq
gives the formula for exchanging derivatives.  The curvature tensor is
  decomposed as
  $$R_{ijkl}= W_{ijkl} + A_{ik} g_{jl}+ A_{il} g_{ik}- A_{il}g_{jk}- A_{jk} g_{il}.$$
Recall that  $C_{ijk} = A_{ij,k} - A_{ik,j}= - W_{lijk, l}.$  The second Bianchi identity can 
be expressed as
\beq \label{e:2Bianchi}
 W_{ijkl,m}+ W_{ijmk,l}+ W_{ijlm, k}+ C_{ikm} g_{jl}+ C_{jlm} g_{ik}+ C_{iml} g_{jk} + 
C_{jmk} g_{il}
 + C_{jkl} g_{im} + C_{ilk} g_{jm}=0.
\eeq
We now recall some facts about the Bach tensor and the $Q$-curvature in dimension 4. It is known the Bach tensor (see \cite{TV05})
$$
\begin{array}{ll}
B_{ij}&\ds=\nabla^k\nabla^l W_{ikjl}+\frac12R^{kl}W_{ikjl}=\triangle A_{ij}-\nabla^k\nabla_i A_{jk}+\frac12R^{kl}W_{ikjl}\\
&\ds=\triangle A_{ij}-\frac16\nabla_i\nabla_j R+R_{ikjp}A^{pk}-R_{ip}{A_j}^p+\frac12R^{kl}W_{ikjl}
\end{array}
$$
Thus, the Bach-flat equation is
\beq
\label{Bachflat1bis}
\nabla^k\nabla^l W_{ikjl}+A^{kl}W_{ikjl}=0
\eeq
or equivalently
\beq
\label{Bachflat1}
\triangle A_{ij}-\frac16\nabla_i\nabla_j R+R_{ikjp}A^{pk}-R_{ip}{A_j}^p+A^{kl}W_{ikjl}=0
\eeq
since ${{W_{ik}}^i}_l=0$. Using (\ref{Ricci}), (\ref{e:2Bianchi}) and (\ref{Bachflat1bis}), we infer
\beq
\label{Bachflat2}
\triangle W_{ijkl}+\nabla_l C_{kji}+\nabla_k C_{lij} +\nabla_iC_{jkl}+\nabla_jC_{ilk}=W*Rm+ g*W*A
\eeq
since
\beqn
{W_{ijmk,l}}^m={{W_{ijmk,}}^m}_l+W*Rm=C_{kji,l}+W*Rm\\
{W_{ijlm,k}}^m={{W_{ijlm,}}^m}_k+W*Rm=C_{lij,k}+W*Rm\\
{C_{ikm,}}^m g_{jl}=\nabla^m\nabla^hW_{himk}g_{jl}=g*W*A,\;\;\; {C_{jlm,}}^m g_{ik}=g*W*A\\
{C_{iml,}}^m g_{jk}=g*W*A,\;\;\; {C_{jmk,}}^m g_{il}=g*W*A.
\eeqn
Now we recall the $Q$-curvature $Q=\frac16(-\triangle R+R^2-3|Ric|^2)$ so that $Q$-flat metric can be interpreted as 
\beq
\label{Qflat}
\triangle R=R^2-3|Ric|^2.
\eeq

The following two lemmas regard basic properties of the tensor $S$ and the relation between $S$ 
and the behavior of the Weyl tensor on the boundary.\\

\begin{lemm} 
\label{lemma2.0}
\begin{enumerate}
\item $S$ tensor $S_{\alpha\beta}=\nabla^i W_{i\alpha n\beta}+\nabla^i W_{i\beta n\alpha}-\nabla^n W_{n\alpha n\beta}+\frac{4}{3}H{W_{\alpha n \beta}}^n $ is a symmetric 2-tensor;
\item $Tr(S)=0$;
\item $S$ is a conformally invariant tensor in the sense that $S(\varphi^2 g)=\varphi^{-1}S( g)$;
\item We have
$$
\begin{array}{ll}
S_{\alpha\beta}=&-A_{\alpha n, \beta}+A_{\beta n, \alpha}+A_{\alpha\beta,n}+\hat{\triangle }L_{\alpha\beta}-L_{\gamma\alpha,\hat{\beta}\hat{\gamma}}
+L_{\beta\gamma}R_{\gamma n \alpha n}\\
&-HR_{\alpha n \beta n}+L_{\gamma\delta}R_{\beta \gamma \alpha\delta}-A_{\gamma n,\gamma}g_{\alpha\beta}+\frac{4}{3}HW_{\alpha n \beta n}
\end{array}
$$
\item if the boundary is totally geodesic, then 
$$
S_{\alpha\beta}=A_{\alpha\beta,n}=\frac 12 R_{\alpha\beta,n}-\frac{R_{,n}}{12}g_{\alpha\beta};
$$
\end{enumerate}
\end{lemm}
\begin{proof}(1) and (2) follow by definition: $Tr (S)=  2 \gra^iW_{i \ga n
\ga} = 2 \gra^iW_{i l n l}= 0.$ \\
3) Let $\tilde g=e^{2f}g$ be a conformal change. Denote $\nabla$ (resp. $\tilde \nabla$) the Levi-Civita connection with respect to the metric $g$ (resp. $\tilde g$). We write $T$  (resp. $\tilde T$) a tensor with respect to the metric $g$ (resp. $\tilde g$).  Let $\Gamma_{ij}^k$ (resp. $\tilde \Gamma_{ij}^k$) be the Christoffel symbols with respect to the metric $g$ (resp. $\tilde g$). We write 
$$
\tilde \Gamma_{ij}^k=\Gamma_{ij}^k+U_{ij}^k
$$
where $U_{ij}^k={\delta^k}_j\nabla_i f+{\delta^k}_i\nabla_j f-g_{ij}\nabla^k f$ is a $(2,1)$ tensor. Under the conformal change, we know $\tilde W=W$ as a $(3,1)$ tensor and the Cotton tensor (recall $n=4$)
$$
\tilde C_{ijk}=C_{ijk}+\nabla_l f {W_{jki}}^{l}.
$$
Moreover, the mean curvature can be changed as follows
$$
\tilde H=e^{-f}(H-3\nabla_n f)
$$
On the other hand, we know $\nabla^i W_{ijkl}=-C_{jkl}$ so that
$$
S_{\alpha\beta}=-C_{\alpha n\beta}-C_{\beta n\alpha}-\nabla^n W_{n\alpha n\beta}+\frac{4}{3}H{W_{\alpha n \beta}}^n
$$
Gathering these relations, we deduce (recall the unit normal $\tilde n$ (resp. $n$)  for $\tilde g$ (resp. $g$) satisfying $\tilde n=e^{-f}n$)
$$
\begin{array}{lll}
\tilde S_{\alpha\beta}&=&e^{-f}(-\tilde C_{\alpha n\beta}-\tilde C_{\beta n\alpha}-e^{-2f}g^{ni}\tilde \nabla_i(e^{2f} W_{n\alpha n\beta})+\frac{4}{3}H{W_{\alpha n \beta}}^n-4\nabla_n f{W_{\alpha n \beta}}^n)\\
&=&e^{-f}(- C_{\alpha n\beta}- C_{\beta n\alpha}-\nabla_i f(W_{\alpha n \beta}^i+W_{\beta n \alpha}^i)-2\nabla^n f W_{n\alpha n\beta}-\nabla^n W_{n\alpha n\beta})\\
&&+e^{-f}(U^i_{nn}W_{i\alpha n\beta}+U^i_{n\alpha}W_{ni n\beta}+U^i_{n\beta}W_{n\alpha ni}+U^i_{nn}W_{n\alpha i\beta}+\frac{4}{3}H{W_{\alpha n \beta}}^n-4\nabla_n f{W_{\alpha n \beta}}^n)
\end{array}
$$
Together with
\beqn
U^i_{nn}W_{i\alpha n\beta}=2\nabla^n f W_{n\alpha n\beta}-\nabla^i f W_{i\alpha n\beta}\\
U^i_{nn}W_{n\alpha i\beta}=2\nabla^n f W_{ n\beta n\alpha }-\nabla^i f W_{i\beta n\alpha}\\
U^i_{n\alpha}W_{n i n\beta}=\nabla_n f W_{  n\alpha n\beta }\\
U^i_{n\beta}W_{n \alpha ni}=\nabla_n f W_{  n\alpha n\beta }
\eeqn
the desired result follows, that is, $S$ is a pointwise conformal invariant.\\
(4) and (5). For the rest of lemma, we use the Fermi coordinates.  The Christoffel symbols satisfy
$\Gamma^n_{\ga \gb}= L_{\ga \gb}, \, \Gamma^{\gb}_{\ga n}= - L_{\ga
\gb}, \, \Gamma^n_{\ga n}= 0$ on the boundary (\cite{Chen05b}, P.8).
Therefore,
 \begin{eqnarray*}
 R_{\ga \gb \gc n, \ga}&=&  \de_{\ga}   R_{\ga \gb \gc n}- \Gamma^l_{\ga \ga} R_{l \gb \gc n} - 
\Gamma^l_{\ga \gb} R_{\ga l \gc n}
   -\Gamma^l_{\ga \gc} R_{\ga  \gb l n} -  \Gamma^l_{\ga n} R_{\ga \gb \gc l}\\
   &=&  R_{\ga \gb \gc n, \hat \ga}- \Gamma^n_{\ga \ga} R_{n \gb \gc n} - 
\Gamma^n_{\ga \gb} R_{\ga n \gc n}
       -  \Gamma^{\gd}_{\ga n} R_{\ga \gb \gc \gd}\\
   &=& R_{\ga \gb \gc n, \hat \ga} + H R_{\gb n \gc n}- L_{\gb \ga} R_{\ga n \gc n}+ 
L_{\ga \gd} R_{\ga \gb \gc \gd}
\end{eqnarray*}
By the Codazzi equation,
   \begin{eqnarray*}
R_{\ga \gb \gc n, \ga}
      &=& L_{\ga \gc, \hat \gb \hat \ga}- \hat \Delta L_{\gb \gc}+ H R_{\gb n \gc n}- 
L_{\gb \ga} R_{\ga n \gc n}+ L_{\ga \gd} R_{\ga \gb \gc \gd}.
\end{eqnarray*}
 Hence, by curvature decomposition and the above formula
 \begin{eqnarray*}
 S_{\ga \gb} &=& \gra^iW_{i \ga n \gb} + \gra^{\gc} W_{\gc \gb n \ga}   + 
\frac{4}{3} H W_{\ga n \gb n}\\
                  &=& (- A_{\ga n, \gb}+ A_{\ga \gb, n}) + (\gra^{\gc} R_{\gc \gb n \ga}-
 A_{\gc n, \gc} g_{\ga \gb}+ A_{\gb n, \ga}) + \frac{4}{3} H W_{\ga n \gb n}\\
                   &=& -A_{\ga n, \gb}+ A_{\gb n, \ga}+ A_{\ga \gb, n} +  
\hat \Delta L_{\ga \gb}- 
L_{\gc \ga, \hat \gb \hat \gc}+ L_{\gb \gc} R_{\gc n \ga n}
                         - H R_{\ga n \gb n}+ L_{\gc \gd} R_{\gb \gc \ga \gd}    \\
                    & &            - A_{\gc n, \gc} g_{\ga \gb}+ \frac{4}{3} 
H W_{\ga n \gb n}.
\end{eqnarray*}
  When the boundary is totally geodesic, then $L= 0$ and $R_{\ga n}= 0.$ It follows  that
   $S_{\ga \gb} = A_{\ga \gb, n}.$\\
 This proves the lemma.
\end{proof}

\begin{rema}
\label{remark2.2}
   We remark that in the conformal compact Einstein setting $(X, M, g^+)$,
with a Poincare Einstein metric $g^{+}$ with $ Ric (g^{+}) = - 3 g^{+}$ and 
$(M, \hat g)$ as conformal infinity; if we choose a special defining function 
$r$ associated with $\hat g$ (c.f. \cite{G00}), with $|\nabla_{g^{+}} r| =1$ near
the boundary M and
\begin{eqnarray*}
r^2 g^{+} =:  g = ds^2 + g_r\mbox{ and }
g_r = \hat g + g^{(2)} r^2 + g^{(3)} r^3 + O (r^4)
g^{(4)}_{\alpha \beta} 
\end{eqnarray*}
Recall $tr(g^{(3)} )=0$. Thus a straight forward computation gives 
$$ \nabla^n{R}_{\alpha \beta} = -3 g^{(3)}_{\alpha \beta}, \,\,\,\,\,\,\,\, \text{and} \,\,\, 
\nabla^n R = 0. $$
Thus in this case, applying properties of the $S$ tensor as above,
we get for any compactified metric g on $X$ with totally geodesic boundary,
\beq
\label{s-to-g3}
S_{\alpha \beta} = - \frac {3}{2} g^{(3)}_{\alpha \beta}. 
\eeq   
\end{rema}

\begin{lemm} 
\label{lemma2.1}
Suppose the boundary is totally geodesic and $W|_M=0$. Then on the boundary, we have
\begin{enumerate}
\item $\nabla_n W_{ \alpha\beta \gamma \delta}=-S_{\alpha\gamma}g_{\beta\delta}-S_{\beta\delta}g_{\alpha\gamma}+S_{\alpha\delta}g_{\beta\gamma}+S_{\beta\gamma}g_{\alpha\delta}$;
\item $\nabla_n  W_{ \alpha\beta \gamma n}=-\hat{C}_{\gamma\beta\alpha}=-C_{\gamma\beta\alpha}$;
\item $\nabla_n  W_{ n \alpha n\beta}=S_{\alpha\beta}$;
\item $\nabla_n \nabla_n W_{ \alpha\beta \gamma \delta}=\hat{\nabla}_\delta\hat{C}_{\gamma\alpha\beta}-\hat{\nabla}_\gamma\hat{C}_{\delta\alpha\beta}-\hat{\nabla}_\lambda\hat{C}_{\alpha\lambda\gamma}g_{\beta\delta}-\hat{\nabla}_\lambda\hat{C}_{\beta\lambda\delta}g_{\alpha\gamma}+\hat{\nabla}_\lambda\hat{C}_{\alpha\lambda\delta}g_{\beta\gamma}+\hat{\nabla}_\lambda\hat{C}_{\beta\lambda\gamma}g_{\alpha\delta}$;
\item $\nabla_n \nabla_n W_{ \alpha\beta \gamma n}=2\hat{\nabla}_\beta S_{\alpha\gamma}-2\hat{\nabla}_\alpha S_{\beta\gamma}$;
\item $\nabla_n \nabla_n W_{ n \alpha n\beta}=-\hat{\nabla}_\gamma \hat{C}_{\alpha\beta\gamma}-\hat{\nabla}_\gamma \hat{C}_{\beta\alpha\gamma}$;
\end{enumerate}
\end{lemm}

\bpf (1) Note that
        \beq \label{e:Cgagbn}
         C_{\ga \gb n}= A_{\ga \gb, n}- A_{\ga n, \gb}= A_{\ga \gb, n} =  S_{\ga \gb}.
         \eeq
             Now by (\ref{e:2Bianchi}),
             \begin{eqnarray*}
             \gra_n W_{\ga \gb \gc \gd} &=& -W_{\ga \gb n \gc, \gd}- W_{\ga \gb \gd n, \gc}- 
(C_{\ga \gc n} g_{\gb \gd}+ C_{\gb \gd n} g_{\ga \gb}
             + C_{\ga n \gd} g_{\gb \gc} + C_{\gb n \gc} g_{\ga \gd}) \\
             &=&  - C_{\ga \gc n} g_{\gb \gd}- C_{\gb \gd n} g_{\ga \gb}  - 
C_{\ga n \gd} g_{\gb \gc} - C_{\gb n \gc} g_{\ga \gd}\\
             &=& -S_{\ga \gc} g_{\gb \gd} - S_{\gb \gd} g_{\ga \gc} + 
S_{\ga \gd} g_{\gb \gc}+ S_{\gb \gc} g_{\ga \gd},
               \end{eqnarray*}
         where in the second equality we use $W|_{\de X}= 0$ and hence $\gra_{\ga} W|_{\de X}= 0,$ 
and the last equality is by (\ref{e:Cgagbn}).

 (2) We first prove that $\hat C_{\ga \gb \gc}= C_{\ga \gb \gc}.$ By Gauss equation, 
$R = \hat R + 2 R_{nn}.$ Therefore, $R_{, \gc}= \hat R_{,\hat \gc}+ 2 R_{nn, \gc}.$
      On the other hand, by curvature decomposition and the fact that $\gra_{\ga} W|_{\de X} = 0,$ 
we have
      $R_{\ga n \gb n. \gc}= W_{\ga n \gb n, \gc}+ A_{\ga \gb, \gc}+ A_{nn, \gc} g_{\ga \gb}= 
A_{\ga \gb, \gc}+ A_{nn, \gc} g_{\ga \gb}.$
      Using above information, we obtain
      \begin{eqnarray*}
    A_{\ga \gb, \gc} &=& \frac{1}{2} R_{\ga \gb, \gc}- \frac{R_{, \gc}}{12} g_{\ga \gb} =   
\frac{1}{2} \hat R_{\ga \gb, \hat \gc}
      + \frac{1}{2} R_{\ga n \gb n. \gc}- \frac{R_{, \gc}}{12} g_{\ga \gb}\\
      &=&  \frac{1}{2} \hat R_{\ga \gb, \hat \gc}
      + \frac{1}{2}A_{\ga \gb, \gc}+ \frac{1}{2}A_{nn, \gc} g_{\ga \gb}- 
\frac{R_{, \gc}}{12} g_{\ga \gb}.
     \end{eqnarray*}
    Hence,
       \begin{eqnarray*}
    \frac{1}{2} A_{\ga \gb, \gc}  &=&  \frac{1}{2} \hat R_{\ga \gb, \hat \gc} + 
\frac{1}{2}A_{nn, \gc} g_{\ga \gb}- \frac{R_{, \gc}}{12} g_{\ga \gb}\\
          &=& \frac{1}{2} \hat R_{\ga \gb, \hat \gc} + \frac{1}{8}(2 R_{nn, \gc}- 
R_{, \gc}) g_{\ga \gb}
          =  \frac{1}{2} \hat R_{\ga \gb, \hat \gc} - \frac{\hat R_{, \hat \gc}}{8} g_{\ga \gb}=  
\frac{1}{2} \hat A_{\ga \gb, \hat \gc}.
     \end{eqnarray*}
      Therefore,  $\hat C_{\ga \gb \gc}= \hat A_{\ga \gb, \hat \gc} -\hat A_{\ga \gc, \hat \gb}= 
A_{\ga \gb, \gc} -A_{\ga \gc,  \gb} = C_{\ga \gb \gc}.$

     Now, $\gra_n W_{\ga \gb \gc n}= -C_{\gc \gb \ga} - \gra_{\gd} W_{ \ga \gb \gc \gd}= 
-C_{\gc \gb \ga}.$

 (3) $\gra_n W_{n \ga n \gb}= -C_{\gb n \ga} - \gra_{\gc} W_{n \ga \gc \gb}= -C_{\gb n \ga}.$
     Using (\ref{e:Cgagbn}), we  get $\gra_n W_{n \ga n \gb}= S_{\ga \gb}.$
     
(4) By (\ref{e:2Bianchi}),
          \begin{align}
            \gra_n \gra_n W_{\ga \gb \gc \gd} &= -(W_{\ga \gb n \gc, \gd}+ 
W_{\ga \gb \gd n, \gc})_n- (C_{\ga \gc n} g_{\gb \gd}+ C_{\gb \gd n} g_{\ga \gc}
           + C_{\ga n \gd} g_{\gb \gc} + C_{\gb n \gc} g_{\ga \gd})_n \notag\\
            &= -W_{\ga \gb n \gc, n \gd}- W_{\ga \gb \gd n, n \gc}  - 
C_{\ga \gc n, n} g_{\gb \gd}- C_{\gb \gd n, n} g_{\ga \gc}   - 
C_{\ga n \gd, n} g_{\gb \gc} - C_{\gb n \gc, n} g_{\ga \gd} \notag\\
            &= \hat \gra_{\gd} \hat C_{\gc \ga \gb}- 
\hat \gra_{\gc} \hat C_{\gd \ga \gb}-C_{\ga \gc n, n} g_{\gb \gd}- 
C_{\gb \gd n, n} g_{\ga \gc}
              - C_{\ga n \gd, n} g_{\gb \gc} - C_{\gb n \gc, n} g_{\ga \gd} \label{e:gragraWa},
               \end{align}
            where in the second quality, we use the Ricci identity (\ref{Ricci})
$W_{\ga \gb n \gc, \gd n}= W_{\ga \gb  n \gc,  n \gd}$
       by noting that $W|_{\de X}= 0,$ and the last equality is by Lemma~\ref{lemma2.0} (2).

    Now by the Bach-flat equation $\gra_l \gra_k W_{kilj}= 0$ and Lemma~\ref{lemma2.0} (2),
    $$C_{\ga n  \gb, n} = C_{\ga  l \gb, l} - C_{\ga \gc \gb, \gc}=   
- \gra_l \gra_k W_{k  \ga l \gb}- C_{\ga \gc \gb, \gc}= - C_{\ga \gc \gb, \gc}
       = - \hat C_{\ga \gc \gb, \hat \gc}.$$  Inserting the above formula into 
(\ref{e:gragraWa}) gives (4).

(5) By (\ref{e:2Bianchi}),
     \begin{eqnarray*}
     \gra_n \gra_n W_{\ga \gb \gc n} &=& -(W_{\gc n n \ga, \gb}+ 
W_{\gc n \gb n, \ga})_n- (C_{n \gb n} g_{\gc \ga}+ C_{n n \ga} g_{\gc \gb}
          + C_{\gc \gb \ga} g_{n n})_n\\
          & =& -W_{\gc n n \ga,  n \gb}- W_{\gc n \gb n, n \ga }- 
C_{n \gb n, n} g_{\gc \ga}- C_{n n \ga, n} g_{\gc \gb}- C_{\gc \gb \ga, n}\\
           &=& \hat \gra_{\gb} S_{\ga \gc} - \hat \gra_{\ga} S_{\gb \gc}- 
C_{n \gb n, n} g_{\gc \ga}- C_{n n \ga, n} g_{\gc \gb}- C_{\gc \gb \ga, n},
       \end{eqnarray*}
       where in the second quality, we use Ricci identity (\ref{Ricci}) $W_{\gc n n \ga, \gb n}= 
W_{\gc n n \ga,  n \gb}$
        and the last equality is by Lemma~\ref{lemma2.0} (3).

       Using the Bach-flat equation gives $C_{n n \ga, n} = C_{n l \ga, l} - C_{n \gb \ga, \gb}=  
 -\gra_l \gra_k W_{k  n l \ga}- C_{n \gb \ga, \gb}
       = - C_{n \gb \ga, \gb}.$ Since the boundary is totally geodesic, $ C_{n \gb \ga, \gb}= 
(A_{n \gb, \ga} - A_{n \ga , \gb})_{\gb}= 0.$
       Therefore,  $C_{n n \ga, n} =0$ and
       $$ \gra_n \gra_n W_{\ga \gb \gc n}  = \hat \gra_{\gb} S_{\ga \gc} - 
\hat \gra_{\ga} S_{\gb \gc}- C_{\gc \gb \ga, n}.$$
       To compute $C_{\gc \gb \ga, n},$ by Ricci identity (\ref{Ricci})
       \begin{eqnarray*}
      C_{\gc \gb \ga, n}&=& A_{\gc \gb, \ga n} - A_{\gc \ga, \gb n}\\
          &=&  A_{\gc \gb, n \ga} - A_{\gc \ga, n \gb} - R_{l \gc n \ga} A_{l \gb} -
                 R_{l \gb n \ga} A_{l \gc} + R_{l \gc n \gb} A_{l \ga} +R_{l \ga n \gb} A_{l \gc}\\
          &=&  A_{\gc \gb, n \ga} - A_{\gc \ga, n \gb} = (A_{\gc n, \gb}+ 
C_{\gc \gb n})_{\ga} - (A_{\gc n, \ga} + C_{\gc \ga n})_{\gb}
          =  C_{\gc \gb n, \ga} -  C_{\gc \ga n, \gb},
      \end{eqnarray*}
       where we use $A_{n \ga}= R_{n \ga \gb \gc} =0$ because the boundary is totally geodesic. 
Finally, by (\ref{e:Cgagbn})
        \beq \label{e: Cgcgbgan}
          C_{\gc \gb \ga, n} =  C_{\gc \gb n, \ga} -  C_{\gc \ga n, \gb}=   
\hat \gra_{\ga} S_{\gb \gc}- \hat \gra_{\gb} S_{\ga \gc},
          \eeq
        which finishes the proof of (5).

(6) By the Bach-flat equation,
   \begin{eqnarray*}
  \gra_n \gra_n W_{\ga n \gb n}&=& \gra_l \gra_k W_{k \ga l \gb}-
\gra_{\gc} \gra_{\gd} W_{\ga \gc \gb \gd}  -
\gra_n \gra_{\gc} W_{\ga n \gb \gc} -\gra_{\gc} \gra_n W_{\ga \gc \gb n} \\
        &=&  - W_{\ga n \gb \gc, n \gc} - W_{\ga \gc \gb n, n \gc} =  
-\hat \gra_{\gc} \hat C_{\ga \gb \gc} - \hat \gra_{\gc} \hat C_{\gb \ga \gc},
  \end{eqnarray*} where in the second equality we use the Ricci identity (\ref{Ricci})   and
  the last equality is by Lemma~\ref{lemma2.0} (2).
 \epf

We now recall the following result.
\begin{lemm}
\label{lemma2.3}
Under the Fefferman-Graham's compactification with the Yamabe metric on the boundary, we have 
$R=3\hat{R}$ on $M$.
\end{lemm}

\begin{proof}
From Corollary 6.6 in \cite{casechang}, we have $J_g=2J_{\hat {g}}$ where $J_g=\frac{R}6 $ (resp. $J_{\hat {g}}=\frac{\hat{R}}4$) is the trace of the Schouten tensor of the metric $g$ (resp. $\hat {g}$). Therefore 
we get the desired result.
\end{proof}

The following result is well known (see \cite{FG0}).\\

\begin{lemm}
\label{lemma2.3bis}
Under any compactification, we have $W|_M=0$  on $M$.
\end{lemm}

We split the tangent bunlde on the boundary $T_xX=\mathbb{R}\vec{\nu}\oplus T_xM$ for all $x\in M$, where $\vec{\nu}$ is unit normal vector on the boundary. Given a tensor $T$, we decompose tensor $\nabla^{(k)}  T$ on $X$ along $M$ that are related to the splitting $T_xX=\mathbb{R}\vec{\nu}\oplus T_xM$:
let us denote by $\nabla_{odd}^{(k)}  T$ (resp. $\nabla_{even}^{(k)}  T$) the normal component $\vec{\nu}$ appeared odd time (resp. even time) in the tensor $\nabla^{(k)}  T$.

\begin{lemm}
\label{lemma2.4}
Suppose the boundary is totally geodesic and $W|_M=0$. We have on the boundary $M$ for any $k\le 1$
$$
\nabla_n^{(k+1)}  W= L(\hat{ \nabla}^{(k)} S, \hat{ \nabla}^{(k)} \hat C)
$$ 
where $L$ is some linear function. As a consequence, we have on $M$ for any $k\le 1$
$$
\nabla^{(k+1)}  W= L(\hat{ \nabla}^{(k)} S, \hat{ \nabla}^{(k)} \hat C).
$$
More precisely, we have
$$
\nabla_{odd}^{(k+1)}  W= L(\hat{ \nabla}^{(k)} S)
$$
and
$$
\nabla_{even}^{(k+1)}  W= L(\hat{ \nabla}^{(k)} \hat C)
$$
\end{lemm}
\begin{proof}
The first part of Lemma comes from Lemma \ref{lemma2.1}. Recall $W=0$ on $M$. Thus 
$$
\nabla_\alpha W=0
$$
Also it follows from the Ricci identity (\ref{Ricci}),
$$
\nabla_i \nabla_j W=\nabla_j \nabla_i W
$$
Thus the desired result follows from Lemma \ref{lemma2.1} again and the lemma is proved.
\end{proof}

\begin{lemm}
\label{lemma2.5}
Suppose the boundary is totally geodesic and $W|_M=0$ for some $Q$-flat and Bach-flat metric.  Then on the boundary $M$ 
\begin{enumerate}
\item We have
\beqn
A_{\alpha\beta,\gamma}=\hat A_{\alpha\beta, \gamma}, A_{\alpha\beta,n}=S_{\alpha\beta}, A_{n\alpha, \beta}=A_{\alpha n, \beta}=0, 
A_{n n, n}=\frac16 R_{,n},\\ A_{nn,\alpha}=\frac16 \hat{\nabla}_\alpha R-\hat{A}_{\beta\beta,\alpha}, A_{n\alpha,n}=A_{\alpha n,n}=\frac16 \hat{\nabla}_\alpha R-\hat{A}_{\alpha\beta,\beta}
\eeqn
\item We have
\beqn
A_{\alpha \beta,\gamma\lambda}=\hat{A}_{\alpha \beta,\gamma\lambda},\;A_{\alpha \beta, \gamma n}=A_{\alpha \beta, n \gamma }=\hat \nabla_\gamma S_{\alpha\beta},\; \\ A_{\alpha \beta,nn}=\frac16 \hat{\nabla}_{\alpha} \hat{\nabla}_{\beta}R - \hat{A}_{\alpha \beta,\gamma\gamma}-|A|^2g_{\alpha \beta}+4 A_{\alpha p}{A_\beta}^p\\
A_{n\alpha ,\beta\gamma}=A_{\alpha n,\beta\gamma}=0,\;
A_{n\alpha , nn }=A_{\alpha n, nn }=A_{n n, n \alpha }=A_{n n,  \alpha n}=\frac16 R_{,n\alpha},\\
A_{n\alpha , n\beta }=A_{\alpha n, n\beta }=\frac16 \hat{\nabla}_{\beta}\hat{\nabla}_{\alpha} R- \hat{A}_{\alpha \gamma,\gamma\beta}, \\
A_{n\alpha , \beta n}=A_{\alpha n, \beta n}=\frac16 \hat{\nabla}_{\beta}\hat{\nabla}_{\alpha} R- \hat{A}_{\alpha \gamma,\gamma\beta}-(A_{nn})^2g_{\alpha\beta}+A_{\alpha\gamma}{A_{\alpha}}^\gamma\\
A_{nn,\alpha\gamma}=\frac16 \hat{\nabla}_{\alpha} \hat{\nabla}_{\gamma}R-\hat{A}_{\beta\beta,\alpha\gamma},\; 
R_{,nn}=R^2-3|Ric|^2-\hat{\nabla}_\alpha\hat{\nabla}_\alpha R
\eeqn
Moreover, we have
$$
A_{nn,nn}=-\frac{1}{3}R_{,\alpha\alpha}+\frac{1}{6}( R^2-3|Ric|^2) -A^{\alpha\beta}A_{\alpha\beta}+3 (A_{nn})^2+\frac14 \hat{\triangle} \hat{R}.
$$
In particular,  when $\hat{R}$ is constant on $M$, then 
$$
A_{nn,nn}=-\frac{1}{3}R_{,\alpha\alpha}+\frac{1}{6}( R^2-3|Ric|^2) -A^{\alpha\beta}A_{\alpha\beta}+3 (A_{nn})^2
$$
\item We have
\beqn
\nabla^{(2)}_{odd} R=L(\hat\nabla \nabla_n R),\; \nabla^{(2)}_{even} R=L(\hat\nabla^{(2)} R)+A*A,\\
\nabla^{(1)}_{odd} A=L(S,\nabla_n R),\; \nabla^{(1)}_{even} A=L(\hat\nabla \hat A,\hat\nabla R),\\
\nabla^{(2)}_{odd} A=L(\hat\nabla S, \hat\nabla \nabla_n R),\; \nabla^{(2)}_{even} A=L(\hat\nabla^{(2)} \hat A,\hat\nabla^{(2)} R)+A*A.
\eeqn
\item There holds
\beqn
A_{\alpha\beta}=\hat{A}_{\alpha\beta}+(A_{nn}+\frac{\hat R}{4}-\frac{R}{6})g_{\alpha\beta}
\eeqn
\item Under the Fefferman-Graham's compactification with the Yamabe metric on the boundary, we have 
\beqn
A_{nn}=\frac{\hat{R}}4\\
A_{\alpha\beta}=\hat{A}_{\alpha\beta}
\eeqn
\end{enumerate}
\end{lemm}
\begin{proof}
(1) The first two equalities are proved in the proof of Lemmas \ref{lemma2.0} and \ref{lemma2.1}. Recall $A_{\alpha n}=0$ on $M$ from the Codazzi equations so that the third one comes.  From the relation $tr(A)=\frac R6$,  we have on the boundary
$$
A_{\alpha\alpha,n}+A_{nn, n}=\frac{R_{,n}}{6}
$$
On the other hand, by Lemma \ref{lemma2.0}
$$
A_{\alpha\alpha,n}=tr(S)=0
$$
so that 
$$
A_{nn, n}=\frac{R_{,n}}{6}
$$
Similarly, 
$$
A_{\beta\beta,\alpha}+A_{nn, \alpha}=\frac{R_{,\alpha}}{6}
$$
so that
$$
A_{nn,\alpha}=\frac16 \hat{\nabla}_\alpha R-\hat{A}_{\beta\beta,\alpha}.
$$
since ${A}_{\beta\beta,\alpha}=\hat{A}_{\beta\beta,\alpha}$. By the second Bianchi identity, we obtain 
$$
A_{\alpha\beta,\beta}+A_{\alpha n,n}=\frac {R_{,\alpha}}{6}
$$
which implies
$$
A_{n\alpha,n}=A_{\alpha n,n}=\frac16 \hat{\nabla}_\alpha R-\hat{A}_{\alpha\beta,\beta}
$$
since $A$ is symmetric and ${A}_{\alpha\beta,\beta}=\hat{A}_{\alpha\beta,\beta}$.
\\
(2) It follows from $A_{\alpha \beta,\gamma}=\hat{A}_{\alpha \beta,\gamma} $ that $ A_{\alpha \beta,\gamma\lambda}=\hat{A}_{\alpha \beta,\gamma\lambda}$. \\
From the Ricci identity (\ref{Ricci}) and the codazzi equations $R_{\alpha\beta\gamma n}=0$ on $M$, we get $A_{\alpha \beta, \gamma n}=A_{\alpha \beta, n \gamma }$ so that $A_{\alpha \beta, \gamma n}=A_{\alpha \beta, n \gamma }=\hat \nabla_\gamma S_{\alpha\beta}$.\\
By the Bach flat  equation (\ref{Bachflat1}) for $A_{\alpha \beta}$ and the decomposition of curvature tensor $ Rm=W+A\circledwedge g$, we infer
\beqn
A_{\alpha \beta,nn}=\triangle A_{\alpha \beta}-\hat{A}_{\alpha \beta,\gamma\gamma} =\frac16 \hat{\nabla}_{\alpha} \hat{\nabla}_{\beta}R - \hat{A}_{\alpha \beta,\gamma\gamma}-R_{\alpha k\beta p}A^{pk}+R_{\alpha p}{A_\beta}^p\\
=\frac16 \hat{\nabla}_{\alpha} \hat{\nabla}_{\beta}R - \hat{A}_{\alpha \beta,\gamma\gamma}-|A|^2g_{\alpha \beta}+4 A_{\alpha p}{A_\beta}^p
\eeqn
From the fact $A_{\alpha n}=A_{n\alpha }=0$, we have $A_{n \alpha,\beta\gamma}=A_{\alpha n,\beta\gamma}=0$.  There holds
$$
A_{n\alpha ,nn}=A_{\alpha n,nn}=\triangle A_{\alpha n}-A_{\alpha n, \beta\beta }=\triangle A_{\alpha n}=\frac16 R_{,n\alpha}
$$
since $A_{\alpha n}=R_{\alpha \beta \gamma n}=0$ and Bach flat  equation (\ref{Bachflat1}) for $A_{\alpha n}$. We have
$$
A_{n n, \alpha n}=\frac16 R_{,  \alpha n}-A_{\beta\beta, \alpha n}=\frac16 R_{,  \alpha n}-A_{\beta\beta, n \alpha }=\frac16 R_{,  \alpha n}-\hat\nabla_\alpha S_{\beta\beta}=\frac16 R_{,  \alpha n}
$$
since $tr(S)=0$. It follows from (2) $A_{nn,n}=\frac16 R_{,n}$ there holds
$$
A_{n n, n \alpha }=\frac16 R_{,  \alpha n}
$$
Using $A_{n\alpha,n}=A_{\alpha n,n}=\frac16 \hat{\nabla}_\alpha R-\hat{A}_{\alpha\gamma,\gamma}$, there holds
$$
A_{n\alpha , n\beta }=A_{\alpha n, n\beta }=\frac16 \hat{\nabla}_{\beta}\hat{\nabla}_{\alpha} R- \hat{A}_{\alpha \gamma,\gamma\beta}
$$
With the Ricci identity (\ref{Ricci}), we deduce
$$
\begin{array}{lll}
A_{n\alpha , \beta n}&=A_{\alpha n, \beta n}=A_{\alpha n, n\beta }-R_{m\alpha n\beta}{A^m}_n-R_{mn n\beta}{A_\alpha}^m\\
&=\frac16 \hat{\nabla}_{\beta}\hat{\nabla}_{\alpha} R- \hat{A}_{\alpha \gamma,\gamma\beta}-(A_{nn})^2g_{\alpha\beta}+A_{\alpha\gamma}{A_{\alpha}}^\gamma
\end{array}
$$
By $A_{nn,\alpha}=\frac16 \hat{\nabla}_\alpha R-\hat{A}_{\gamma\gamma,\alpha}$, we infer
$$
A_{nn,\alpha\beta}=\frac16  \hat{\nabla}_\beta\hat{\nabla}_\alpha R-\hat{A}_{\gamma\gamma,\alpha\beta}
$$
By $Q$-flat condition (\ref{Qflat}), we have
$$
R_{,nn}=\triangle R -\hat{\nabla}_\alpha\hat{\nabla}_\alpha R=R^2-3|Ric|^2-\hat{\nabla}_\alpha\hat{\nabla}_\alpha R
$$
From Bach flat condition (\ref{Bachflat2}), $Q$-flat condition (\ref{Qflat}) and from the decomposition of curvature tensor $ Rm=W+A\circledwedge g$, 
we calculate
$$
\begin{array}{lll}
A_{nn, nn}&=\triangle A_{nn}-A_{nn, \alpha\alpha}=\triangle A_{nn}-(\frac{R}{6}-A_{\beta\beta})_{, \alpha\alpha}\\
&\ds=\frac16 R_{,nn}-A^{\alpha\beta}A_{\alpha\beta}+3(A_{nn})^2-
\frac{1}{6}R_{, \alpha\alpha}+A_{\beta\beta, \alpha\alpha}\\
&\ds=\frac16(\triangle R- R_{,\alpha\alpha})-A^{\alpha\beta}A_{\alpha\beta}+3(A_{nn})^2-
\frac{1}{6}R_{, \alpha\alpha}+A_{\beta\beta, \alpha\alpha}\\
&\ds=\frac16(R^2-3|Ric|^2)-A^{\alpha\beta}A_{\alpha\beta}+3(A_{nn})^2-
\frac{1}{3}R_{, \alpha\alpha}+A_{\beta\beta, \alpha\alpha}
\end{array}
$$
On the other hand
$$
A_{\beta\beta, \alpha\alpha}=\hat{A}_{\beta\beta, \alpha\alpha}=\frac14 \hat{R}_{, \alpha\alpha}=\frac14 \hat{\triangle}\hat{R}
$$
since $tr(\hat{A})=\frac14 \hat{R}$.  This yields the desired result. \\
(3) It is just the result from (1) and (2). \\
(4) $$
2A_{\alpha\beta}=Ric_{\alpha\beta}-\frac{R}{6}g_{\alpha\beta}=R_{\alpha\gamma\beta\gamma}+R_{\alpha n\beta n}-\frac{R}{6}g_{\alpha\beta}
$$
By Gauss equation $R_{\alpha\gamma\beta\gamma}=\hat R_{\alpha\gamma\beta\gamma}$ and the decomposition of curvature tensor $ Rm=W+A\circledwedge g$,
$$
2A_{\alpha\beta}=\hat A_{\alpha\beta}+A_{\alpha \beta }+ A_{nn}g_{\alpha \beta }-\frac{R}{6}g_{\alpha\beta}+\frac{\hat R}{4}g_{\alpha\beta},
$$
which implies the desired result. \\
(5) From Gauss-Codazzi equation, we have $R=\hat{R}+2R_{nn}$. Together with Lemma \ref{lemma2.3}, we infer $R_{nn}=\hat{R}$ and $A_{nn}=\frac{\hat{R}}{4}$. Combining this with the result in (4), we infer $A_{\alpha\beta}=\hat{A}_{\alpha\beta}$.  Finally, we prove the result.
\end{proof}

\begin{lemm}
\label{lemma2.6}
Suppose the boundary is totally geodesic and $W|_M=0$ for some $Q$-flat and Bach-flat metric.  We have on the boundary $M$ for any $k\ge 2$
$$
\nabla^{(k)}_{odd} R =  L(\hat \nabla^{(k-1)}\nabla_n R)+ \sum_{l=0}^{k-2}\nabla^{(l)}Rm*\nabla^{(k-2-l)}A
$$
$$
\nabla^{(k)}_{odd} A = L(\hat \nabla^{(k-1)}S,\hat \nabla^{(k-1)}\nabla_n R)+\sum_{l=0}^{k-2}\nabla^{(l)}Rm*\nabla^{(k-2-l)}A
$$
$$
\nabla^{(k)}_{odd} W = L(\hat\nabla^{(k-1)}S,\hat \nabla^{(k-1)}\nabla_n R)+\sum_{l=0}^{k-2}\nabla^{(l)}Rm*\nabla^{(k-2-l)}Rm
 $$
where $L$ is some linear function.
\end{lemm}
\begin{proof}
We prove it by the induction. When $k=1,2$, it is clear for  $A$ and $W$ from  lemmas \ref{lemma2.4} and \ref{lemma2.5}. 
Suppose it is true for $k$. For $k+1$, we consider first $R$. 
The terms  $\nabla^{(k+1)}_{odd} R$ are in three cases.\\
a) $\nabla^{(k+1)}_{odd} R=\nabla_n\nabla_\alpha \nabla^{(k-1)}_{even} R$.\\
In such case, by the Ricci identity (\ref{Ricci})
$$
\nabla^{(k+1)}_{odd} R=\nabla_\alpha\nabla_n \nabla^{(k-1)}_{even} R+Rm*\nabla^{(k-1)} R=\nabla_\alpha \nabla^{(k)}_{odd} R+Rm*\nabla^{(k-1)} A
$$
By the assumptions of the induction, we get the result.\\
b) $\nabla^{(k+1)}_{odd} R=\nabla_n\nabla_n \nabla^{(k-1)}_{odd} R$.\\
We write by the Ricci identity (\ref{Ricci})
$$
\nabla^{(k+1)}_{odd} R=\triangle  \nabla^{(k-1)}_{odd} R-\nabla_\alpha\nabla_\alpha \nabla^{(k-1)}_{odd} R
=  \nabla^{(k-1)}_{odd} \triangle R-\nabla_\alpha\nabla_\alpha \nabla^{(k-1)}_{odd} R+\sum_{l=0}^{k-1}\nabla^{(l)}Rm*\nabla^{(k-2-l)}A
$$
By flat $Q_4$ curvature condition (\ref{Qflat}) and using the assumption of the induction, we get the result.\\
c)  $\nabla^{(k+1)}_{odd} R=\nabla_\alpha\nabla^{(k)}_{odd} R$.\\
This is an easier case. It follows from the assumptions of the induction.\\

Now we consider  the terms $\nabla^{(k+1)}_{odd} A$. Similarly, we consider them in three cases.\\
a) $\nabla^{(k+1)}_{odd} A=\nabla_n\nabla_\alpha \nabla^{(k-1)}_{even} A$.\\
In such case, by the Ricci identity (\ref{Ricci})
$$
\nabla^{(k+1)}_{odd} A=\nabla_\alpha\nabla_n \nabla^{(k-1)}_{even} A+Rm*\nabla^{(k-1)} A=\nabla_\alpha \nabla^{(k)}_{odd} A+Rm*\nabla^{(k-1)} A
$$
Thus, we get the result.\\
b) $\nabla^{(k+1)}_{odd} A=\nabla_n\nabla_n \nabla^{(k-1)}_{odd} A$.\\
We write
$$
\nabla^{(k+1)}_{odd} A=\triangle  \nabla^{(k-1)}_{odd} A-\nabla_\alpha\nabla_\alpha \nabla^{(k-1)}_{odd} A
$$
From the Bach flat equation (\ref{Bachflat1}), we get
$$
\triangle \nabla^{(k-1)} A-\nabla^{(k+1)} \frac R 6+\sum_{l=0}^{k-1}\nabla^{(l)}Rm*\nabla^{(k-1-l)}A=0
$$
which implies
$$
\nabla^{(k+1)}_{odd} A=  \nabla^{(k+1)}_{odd} \frac R6-\nabla_\alpha\nabla_\alpha \nabla^{(k-1)}_{odd} A+\sum_{l=0}^{k-1}\nabla^{(l)}Rm*\nabla^{(k-1-l)}A
$$
By the assumption of the induction and the above equation, we could write
$$
\nabla^{(k+1)}_{odd} A= L(\hat\nabla^{(k)}S,
\hat\nabla^{(k)} \nabla_n R)+\nabla^{(k+1)}_{odd} \frac R 6+\sum_{l=0}^{k-1}\nabla^{(l)}Rm*\nabla^{(k-1-l)}A
$$
which implies by the result for the scalar curvature
$$
\nabla^{(k+1)}_{odd} A= L(\hat\nabla^{(k)}S,
\hat\nabla^{(k)} \nabla_n R)+\sum_{l=0}^{k-1}\nabla^{(l)}Rm*\nabla^{(k-1-l)}A
$$
c) $\nabla^{(k+1)}_{odd} A=\nabla_\alpha  \nabla^{(k)}_{odd} A$.\\
It follows from the assumptions of the induction.\\
The proof for the Weyl tensor is quite similar as the schouten tensor $A$. We divide into 3 cases as above\\
a) $\nabla^{(k+1)}_{odd} W=\nabla_n\nabla_\alpha \nabla^{(k-1)}_{even} W=\nabla_\alpha\nabla_n \nabla^{(k-1)}_{even} W+Rm*\nabla^{(k-1)} A=\nabla_\alpha \nabla^{(k)}_{odd} W+Rm*\nabla^{(k-1)} W$\\
b) $\nabla^{(k+1)}_{odd} W=\nabla_n\nabla_n \nabla^{(k-1)}_{odd} W =\triangle  \nabla^{(k-1)}_{odd} W-\nabla_\alpha\nabla_\alpha \nabla^{(k-1)}_{odd} W$\\
From the Bach flat equation (\ref{Bachflat2}) and the Ricci identity (\ref{Ricci}), we get
$$
\nabla^{(k+1)}_{odd} W= L( \nabla^{(k+1)}_{odd} A)-\nabla_\alpha\nabla_\alpha \nabla^{(k-1)}_{odd} W+\sum_{l=0}^{k-1}\nabla^{(l)}Rm*\nabla^{(k-1-l)}Rm.
$$
c) $\nabla^{(k+1)}_{odd} W=\nabla_\alpha  \nabla^{(k)}_{odd} W$.\\
In the above 3 cases, we can prove the result by the assumptions of the induction and the results for $A$. 
We therefore have established the proof of Lemma \ref{lemma2.6}.
\end{proof}

\begin{lemm}
\label{lemma2.7}
Suppose the boundary is totally geodesic and $W|_M=0$ for some $Q$-flat and Bach-flat metric. For any $k\ge 2$, we have 
$$
\nabla^{(k)}_{even} R =  L( \hat{\nabla}^{(k)} R) +\sum_{l=0}^{k-2}\nabla^{(l)}Rm*\nabla^{(k-2-l)}A
$$
$$
\nabla^{(k)}_{even} A = L(\hat \nabla^{(k)} \hat A,\hat \nabla^{(k)} R )+\sum_{l=0}^{k-2}\nabla^{(l)}Rm*\nabla^{(k-2-l)}A
$$
$$
\nabla^{(k)}_{even} W =L(\hat\nabla^{(k)} \hat A,\hat \nabla^{(k)} R)+\sum_{l=0}^{k-2}\nabla^{(l)}Rm*\nabla^{(k-2-l)}Rm
$$
where $L$ is some linear function. In particular, when the restriction of $R$ on $M$ is constant, we have
$$
\nabla^{(k)}_{even} R = \sum_{l=0}^{k-2}\nabla^{(l)}Rm*\nabla^{(k-2-l)}A
$$
$$
\nabla^{(k)}_{even} A = L(\hat \nabla^{(k)} \hat A)+\sum_{l=0}^{k-2}\nabla^{(l)}Rm*\nabla^{(k-2-l)}A
$$
$$
\nabla^{(k)}_{even} W =L(\hat\nabla^{(k)} \hat A)+\sum_{l=0}^{k-2}\nabla^{(l)}Rm*\nabla^{(k-2-l)}Rm
$$
\end{lemm}
\begin{proof} We prove  the result by induction. For $k=1,2$, the results follow from Lemmas \ref{lemma2.4} and  \ref{lemma2.5}. 
As  before, we treat the three cases. First we consider the scalar curvature.\\
a) $\nabla^{(k+1)}_{even} R=\nabla_n\nabla_\alpha \nabla^{(k-1)}_{odd} R$.\\
In such case, by the Ricci identity (\ref{Ricci})
$$
\nabla^{(k+1)}_{even} R=\nabla_\alpha\nabla_n \nabla^{(k-1)}_{odd} R+Rm*\nabla^{(k-1)} R=\nabla_\alpha \nabla^{(k)}_{even} R+Rm*\nabla^{(k-1)} A
$$
Thus, we get the result by the assumptions of the induction.\\
b) $\nabla^{(k+1)}_{even} R=\nabla_n\nabla_n \nabla^{(k-1)}_{even} R$.\\
We write by the Ricci identity (\ref{Ricci})
$$
\nabla^{(k+1)}_{even} R=\triangle  \nabla^{(k-1)}_{even} R-\nabla_\alpha\nabla_\alpha \nabla^{(k-1)}_{even} R
=  \nabla^{(k-1)}_{even} \triangle R-\nabla_\alpha\nabla_\alpha \nabla^{(k-1)}_{even} R+\sum_{l=0}^{k-1}\nabla^{(l)}Rm*\nabla^{(k-2-l)}A
$$
By flat $Q_4$ curvature condition (\ref{Qflat}), we get the result by the induction.\\
c)  $\nabla^{(k+1)}_{even} R=\nabla_\alpha\nabla^{(k)}_{even} R$.\\
It is clear by the induction argument.\\

Now we consider  the terms $\nabla^{(k+1)}_{even} A$. Similarly, we consider them in three cases.\\
a) $\nabla^{(k+1)}_{even} A=\nabla_n\nabla_\alpha \nabla^{(k-1)}_{old} A$.\\
In such case, by the Ricci identity (\ref{Ricci})
$$
\nabla^{(k+1)}_{even} A=\nabla_\alpha\nabla_n \nabla^{(k-1)}_{odd} A+Rm*\nabla^{(k-1)} A=\nabla_\alpha \nabla^{(k)}_{even} A+Rm*\nabla^{(k-1)} A
$$
We get the result by the induction.\\
b) $\nabla^{(k+1)}_{even} A=\nabla_n\nabla_n \nabla^{(k-1)}_{even} A$.\\
We write
$$
\nabla^{(k+1)}_{even} A=\triangle  \nabla^{(k-1)}_{even} A-\nabla_\alpha\nabla_\alpha \nabla^{(k-1)}_{even} A
$$
We get
$$
\nabla^{(k+1)}_{even} A=\triangle  \nabla^{(k-1)}_{even} A-\nabla_\alpha\nabla_\alpha \nabla^{(k-1)}_{even} A=  \nabla^{(k-1)}_{even}\triangle A-\nabla_\alpha\nabla_\alpha \nabla^{(k-1)}_{even} A+\sum_{l=0}^{k-1}\nabla^{(l)}Rm*\nabla^{(k-1-l)}A
$$
From the Bach flat equation (\ref{Bachflat1}), we could write
$$
\nabla^{(k+1)}_{even} A=\frac 16 \nabla^{(k+1)}_{even} R -\nabla_\alpha\nabla_\alpha \nabla^{(k-1)}_{even} A+\sum_{l=0}^{k-1}\nabla^{(l)}Rm*\nabla^{(k-1-l)}A
$$
It follows from the result for $R$ and from the assumptions of  the induction.\\
c) $\nabla^{(k+1)}_{even} A=\nabla_\alpha  \nabla^{(k)}_{even} A$.\\
It is clear in this case by the induction.\\
The proof for the Weyl tensor is quite similar as the schouten tensor $A$ and the scalar curvature $R$ as in the proof of Lemma \ref{lemma2.6}. We omit the details. Thus we have established the lemma.\\
 \end{proof}

\section{$\varepsilon$-regularity}

\begin{theo}
\label{epsilonregularity}
Suppose the assumptions (4) and (5) in Theorem \ref{maintheorem} are satisfied and assume  $\|\hat Rm\|_{C^{k+1}(M)}$ and  $\|S\|_{L^1_{loc}(M)}$ are bounded, the metric is $Q$-flat and Bach-flat, the boundary $M$ is totally geodesic, $W|_M=0$ and the restriction of the scalar curvature $R|_M$ is some positive bounded constant. Assume further there exists some positive constant $C_3>0$ such that for any $r<1$ and for any $p$, we have
$$
vol(B(p,r))\le C_4r^4.
$$
Then There exists constants  $\varepsilon>0$ (independent of $k$) and $C>0$ (depending on $k$, $\|\hat Rm\|_{C^{k+1}(M)}$, $\|S\|_{L^1_{loc}(M)}$ and $C_1,C_2$ in  the assumptions (4) and (5) in Theorem \ref{maintheorem}) such that if 
  $$
   \|Rm\|_{L^2(B(p,r))}\le \varepsilon
   $$
   then for all $r<1$
    $$
\left\{\int_{B(p,r/2)} |\nabla^{k}A|^4 dV_g\right\}^{1/2}\le  \frac{C}{r^{2k+2}}\left(\int_{B(p,r)} |A|^2 dV_g+\oint_{B(p,r)\cap M} |S|  +r^4\right)
   $$
    $$
\int_{B(p,r/2)} |\nabla^{k+1} A|^2 dV_g\le  \frac{C}{r^{2k+2}}\left(\int_{B(p,r)} |A|^2 dV_g+\oint_{B(p,r)\cap M} |S|  +r^4\right)
   $$
    $$
\left\{\int_{B(p,r/2)} |\nabla^k Rm|^4 dV_g\right\}^{1/2}\le  \frac{C}{r^{2k+2}}\left(\int_{B(p,r)} |A|^2 dV_g+\oint_{B(p,r)\cap M} |S|  +r^4\right)
   $$
    $$
\int_{B(p,r/2)} |\nabla^{k+1} Rm|^2 dV_g\le  \frac{C}{r^{2k+2}}\left(\int_{B(p,r)} |A|^2 dV_g+\oint_{B(p,r)\cap M} |S|  +r^4\right)
   $$
   
\end{theo}

\begin{proof}

We now begin the proof of the theorem by considering the case $k=0$ first. Let $\eta$ be some cut-off function such that $\eta=1$ on $B(p,\frac {3r}{4})$ and $\eta=0$ outsides $B(p,r)$ and $|\nabla \eta|\le C/r$. Taking the test tensor $\eta^2 A$ in (\ref{Bachflat1}), we have
\beq
\label{eq3.1}
\begin{array}{lllll}
\ds\int\eta^2|\nabla A|^2&=&\ds -\int \eta^2 \langle \triangle A,  A\rangle -2\int \eta \langle\nabla A,  \nabla\eta \otimes A\rangle+\oint \eta^2 \langle\nabla_n A, A \rangle
\\
&=&\ds -\int \eta^2 \langle \triangle A-\frac16  \nabla^2 R,  A\rangle -2\int \eta \langle\nabla A,  \nabla\eta \otimes A\rangle+\oint \eta^2 \langle\nabla_n A, A \vs\rangle\\
&&\ds \vs-\frac16\int \eta^2 \langle  \nabla^2 R,  A\rangle\\
&\le& \ds C\int |\nabla A||A|\eta|\nabla \eta|+\left|\frac{1}{36}\int \eta^2 R\triangle R\right|+ \left|\int \eta^2 (\triangle A_{ij}-\frac16  \nabla_i\nabla_j R)  A_{ij}\right|\\
&&\ds+\left|\oint \eta^2 (\langle\nabla_n A, A \rangle-\frac16 \langle \nabla R, A(n,\cdot)\rangle\right|+\left|\frac{1}{36}\oint \eta^2 R \nabla_n R\right|\\
&\le& \ds C\int |\nabla A||A|\eta|\nabla \eta|+C\int \eta^2 |Rm| |A|^2\\
&&\ds+\left|\oint \eta^2 (\langle\nabla_n A, A \rangle-\frac16 \langle \nabla R, A(n,\cdot)\rangle\right|+\left|\frac{1}{36}\oint \eta^2 R \nabla_n R\right|
\end{array}
\eeq
Here we use the $Q$ flat condition (\ref{Qflat}) and second Bianchi idendity $A_{ij,j}=\frac16 R_{,i}$ and also Bach flat equation (\ref{Bachflat1}).\\
Now we want to estimate the boundary terms. From Lemmas \ref{lemma2.0} and \ref{lemma2.5}, we have
$A_{n\alpha }=0$, $A_{nn,n }-\frac 16 R_{,n}=0$ and  $A_{\alpha\beta}-\hat A_{\alpha\beta}=\lambda g_{\alpha\beta}$, since we also have $tr S = 0$, we obtain       
\beq
\label{eq3.2}
\left|\oint \eta^2 (\langle\nabla_n A, A \rangle-\frac16 \langle \nabla R, A(n,\cdot)\rangle\right|=\left| \oint \eta^2 \langle S,\hat A\rangle|\right|\le \oint \eta^2 |S||\hat A|
\eeq
On the other hand, since by our assumption on $(M, \hat g)$, $\hat A$ is bounded in $C^0$ norm, we obtain 
$$
\oint \eta^2 |S||\hat A|\le C \oint \eta^2 |S|.
$$
Recall $R = 3 \hat R $ is a bounded constant on the boundary by our assumption, $R$ is uniformly bounded
on the boundary. Therefore, we could bound from $Q$-flat condition (\ref{Qflat})
\beq
\label{eq3.4}
\begin{array}{lllll}
\ds \left|\frac{1}{36}\oint \eta^2 R \nabla_n R\right| &\ds=\left|\frac{R}{36}\int \div(\eta^2\nabla R)\right| \\
&\ds\le    \frac 18 \int_X |\eta \nabla  A|^2 + C\int_X \eta^2|  A|^2+C\int  | \nabla  \eta|^2 \\
&\ds\le    \frac 18 \int_X |\eta \nabla  A|^2 + C\int_X  \eta^2|  A|^2+\frac{C}{r^2}vol(B(p,r))
\end{array}
\eeq

{\bf Claim.} For any Lipschitz function $f\in C^{1,0}(X)$ and for any regular function $\eta$ vanishing on $X\setminus B(p,r)$, we have
\beq
\label{Sobolev}
\|\eta f\|_{L^4(X)}^2 \le C'\int_X | \nabla  (\eta f)|^2,
\eeq
and
\beq
\label{Trace}
\|\eta f\|_{L^3(M)}^2 \le C'\int_X | \nabla  (\eta f)|^2,
\eeq
provided $\varepsilon$ is small. \\

To see the claim, we have from the assumption condition (4)   on the Yamabe constants in Theorem \ref{maintheorem}, we get 
$$
\begin{array}{lllll}
\ds\|\eta f\|_{L^4}^2&\ds\le C_1 \int_X |\nabla (\eta f)|^2 +\frac{R}{6} |\eta f|^2\le  C_1( \int_X |\nabla (\eta f)|^2+\|\eta f\|_{L^4}^2 \|Rm\|_{L^2(B(p,r))})\\
&\ds \le C_1( \int_X |\nabla (\eta f)|^2+\|\eta f\|_{L^4}^2 \varepsilon)
\end{array}
$$
so that
$$
\|\eta f\|_{L^4}^2\le 2C_1 \int_X |\nabla (\eta f)|^2
$$
provided $2C_1\varepsilon \le 1$. Similarly, by the assumption condiiton (5) on the boundary type Yamabe constants in Theorem \ref{maintheorem} and (\ref{Sobolev}), we infer
$$
\|\eta f\|_{L^3(\p X)}^2\le C_2 ( \int_X |\nabla (\eta f)|^2+\|\eta f\|_{L^4}^2 \varepsilon)\le 2C_2(1+C'\varepsilon)\int_X |\nabla (\eta f)|^2\le 4C_2\int_X |\nabla (\eta f)|^2,
$$
provided $C'\varepsilon \le 1$. This proves the claim.\\

Now we apply the Cauchy-Schwarz inequality and (\ref{Sobolev})
$$
\begin{array}{lllll}
 &\ds C\int |\nabla A||A|\eta|\nabla \eta|+C\int \eta^2 |Rm| |A|^2\\
 \le & \ds\frac 18 \int_X |\eta \nabla  A|^2 +C'(\int_X  |\nabla \eta|^2 |  A|^2+\|\eta A\|_{L^4}^2 \|Rm\|_{L^2(B(p,r))})\\
 \le & \ds\frac 18 \int_X |\eta \nabla  A|^2 +C'(\int_X  |\nabla \eta|^2 |  A|^2+\|\nabla(\eta |A|)\|_{L^4}^2 \|Rm\|_{L^2(B(p,r))})
\end{array} 
$$
Together with (\ref{eq3.1})-(\ref{eq3.4}), we infer
$$
\begin{array}{lllll}
&\ds \int\eta^2|\nabla A|^2\\
\le&\ds \frac 12 \int\eta^2|\nabla A|^2+C(r^2+ \int  (\eta^2+|\nabla \eta|^2)| A|^2+  \|\nabla (\eta |A|)\|_{L^2}^2 \|Rm\|_{L^2(B(p,r))})+ C \oint \eta^2 |S|\\
\le&\ds (\frac 12+2C\|Rm\|_{L^2(B(p,r))}) \int\eta^2|\nabla A|^2+ C \oint \eta^2 |S|\\
&\ds+C(r^2+ \int  (\eta^2+|\nabla \eta|^2)| A|^2+ 2 \||\nabla \eta| |A|\|_{L^2}^2 \|Rm\|_{L^2(B(p,r))})
\end{array} 
$$
Therefore, when $2C\varepsilon<\frac14$,  we get
$$
\int\eta^2|\nabla A|^2 dV_g\le  \frac{C}{r^2}(\int_{B(p,r)} |A|^2 dV_g+\oint_{B(p,r)\cap M} |S|+r^4)
$$
Here we use $|\nabla |A||\le |\nabla A|$. Again from the Sobolev inequality (\ref{Sobolev}), we deduce
$$
\left\{\int\eta^2| A|^4 dV_g\right\}^{1/2}\le  \frac{C}{r^2}(\int_{B(p,r)} |A|^2 dV_g+\oint_{B(p,r)\cap M} |S|+r^4)
$$
Recall from the Bach flat equation for the Weyl tensor (\ref{Bachflat2}) and the Ricci identity  (\ref{Ricci}) 
$$
\triangle W_{ijkl} =2(\nabla_i\nabla_k A_{jl}-\nabla_i\nabla_l A_{jk}-\nabla_j\nabla_k A_{il}+\nabla_j\nabla_l A_{ik})  + A*A+A*W+W*W
$$
As before, we take $\eta^2 W$ as the test tensor to the above equation. We remark on the boundary $W=0$. Thus
$$
\begin{array}{lllll}
&&\ds\int\eta^2|\nabla W|^2\\
&\le& \ds C\int (|\nabla W|+|\nabla A|)|W|\eta|\nabla \eta|+C\int \eta^2 |Rm| (|W|^2+|A|^2)\\
&&\ds+2|\int \eta^2( \nabla_k A_{jl} \nabla_i W_{ijkl}  -\nabla_l A_{jk}  \nabla_i W_{ijkl}-\nabla_k A_{il}\nabla_j W_{ijkl}+\nabla_l A_{ik} \nabla_jW_{ijkl})|\\
&\le&\ds  C\int (|\nabla W|+|\nabla A|)|W|\eta|\nabla \eta|+\int \eta^2 |Rm| (|W|^2+|A|^2)+\int \eta^2 |\nabla W| |\nabla A| \\
&\le&\ds  \frac12 \int\eta^2|\nabla W|^2+C\int | W|^2|\nabla \eta|^2+(\|\eta W\|_{L^4}^2+\|\eta A\|_{L^4}^2) \|Rm\|_{L^2(B(p,r))}+\int \eta^2 |\nabla A|^2
\end{array}
$$
With the similar arguments as above, we infer
$$
\int\eta^2|\nabla W|^2 \le  \frac{C}{r^2}(\int_{B(p,r)}|Rm|^2 dV_g+r^4)
$$
Again from the Sobolev inequality  (\ref{Sobolev}), we get the desired inequalitiy
$$
\left \{\int(\eta|W|)^4\right\}^{\frac12}\le  \frac{C}{r^2}(\int_{B(p,r)}|Rm|^2 dV_g+r^4)
$$
Now use the relation
$$
 |\nabla^{(l)}Rm|^2= |\nabla^{(l)}W|^2+|\nabla^{(l)}A\circledwedge g|^2
$$
Therefore, we obtain the corresponding inequalities for $Rm$. Thus we have finished the part for $k=0$
of the theorem.\\ 

We now prove for the high $k\ge 1$ by induction. For each k, let $\eta_k$ be some cut-off function such that $\eta_k=1$ on $B(p,\frac {r}{2}+\frac {r}{2^{k+2}})$ and $\eta_k=0$ outsides $B(p,\frac {r}{2}+\frac {r}{2^{k+1}})$ and $|\nabla \eta_k|\le C/r$. But for simplicity of the notation, we denote all such cut function as $\eta$ and skip the index k. 
First we treat the estimates for the Schouten tensor $A$. From the Bach flat equation (\ref{Bachflat1}), we obtain
$$
\triangle \nabla^{(k)} A-\nabla^{(2)}\nabla^{(k)} \frac R 6+\sum_{l=0}^{k}\nabla^{(l)}Rm*\nabla^{(k-l)}A=0
$$
As above, we take $\eta^2\nabla^{(k)} A$ as the test tensor, integrate the equality
of the Bach equation, we obtain
$$
\begin{array}{lllll}
&&\ds\int\eta^2| \nabla^{(k+1)} A|^2\\
&=&\ds -\int \eta^2 \langle \triangle  \nabla^{(k)} A,  \nabla^{(k)} A\rangle -2\int \eta \langle \nabla^{(k+1)} A,  \nabla\eta \otimes  \nabla^{(k)}A\rangle+\oint \eta^2 \langle\nabla_n  \nabla^{(k)}A,  \nabla^{(k)}A \rangle
\\&\le& \ds C\int | \nabla^{(k+1)} A|| \nabla^{(k)}A|\eta|\nabla \eta|+\left|\frac{1}{36}\int \eta^2 \langle \nabla^{(k)} R, \nabla^{(k)} \triangle R\rangle\right|\\
&&\ds+ C \sum_{l=0}^{k}\left|\int \eta^2\nabla^{(l)}Rm*\nabla^{(k-l)}A*\nabla^{(k)}A\right|\\
&&\ds+C\left|\oint \eta^2 \langle\nabla_n \nabla^{(k)}  A, \nabla^{(k)}  A \rangle\right|+C \left|\oint \eta^2 \langle \nabla^{(k+1)}   R, \nabla^{(k)}  A(n,\cdot)\rangle\right|\\
&&\ds+\left|\frac{1}{36}\oint \eta^2 \langle \nabla^{(k)} R, \nabla_n \nabla^{(k)} R\rangle\right|+ C \sum_{l=0}^{k-1}\oint \eta^2 |\nabla^{(l)}Rm||\nabla^{(k-1-l)}A|| \nabla^{(k)}R|\\
&\le& \ds C\int | \nabla^{(k+1)} A|| \nabla^{(k)}A|\eta|\nabla \eta|+ C \sum_{l=0}^{k}\left|\int \eta^2\nabla^{(l)}Rm*\nabla^{(k-l)}A*\nabla^{(k)}A\right|\\
&&\ds+C\left|\oint \eta^2 \langle\nabla_n \nabla^{(k)}  A, \nabla^{(k)}  A \rangle\right|+C \left|\oint \eta^2 \langle \nabla^{(k+1)}   R, \nabla^{(k)}  A(n,\cdot)\rangle\right|\\
&&\ds+\left|\frac{1}{36}\oint \eta^2 \langle \nabla^{(k)} R, \nabla_n \nabla^{(k)} R\rangle \right|+ C \sum_{l=0}^{k-1}\oint \eta^2 |\nabla^{(l)}Rm||\nabla^{(k-1-l)}A|| \nabla^{(k)}R|
\end{array}
$$
Here we use the $Q$ flat condition (\ref{Qflat}) and second Bianchi idendity $A_{ij,j}=\frac16 R_{,i}$. We need just to consider the boundary term
$$
\langle \nabla_n\nabla^{(k)}A, \nabla^{(k)}A\rangle\mbox{ (resp. }  \langle \nabla^{(k+1)}   R, \nabla^{(k)}  A(n,\cdot)\rangle,  \mbox{ or } \langle\nabla^{(k)} R,\nabla_n \nabla^{(k)} R\rangle)
$$
Our basic observation is that in all these products, one is an odd term and another one an even term, where odd and even is defined as in the  proof of the Lemmas \ref{lemma2.6} and \ref{lemma2.7}; we also deduce
from these lemmas that
$$
\langle \nabla_n\nabla^{(k)}_{odd}A, \nabla^{(k)}_{old}A\rangle=O(|\nabla^{(k)}A|+\sum_{l=0}^{k-1}|\nabla^{(k)}A||\nabla^{(l)}Rm||\nabla^{(k-1-l)}A|)
$$
and
$$
\begin{array}{lllll}
&\ds\oint \eta^2\langle \nabla_n\nabla^{(k)}_{even}A, \nabla^{(k)}_{even}A\rangle\\
=&\ds\oint \eta^2\langle \sum c_\alpha \nabla_\alpha \nabla^{(k)}_{odd}A, \nabla^{(k)}_{even}A\rangle+O(\sum_l\oint \eta^2|\nabla^{(k)}A||\nabla^{(l)}Rm||\nabla^{(k-1-l)}A|)
\end{array}
$$
where $c_\alpha$ is some  constant. By the integration by parts, we infer
$$
\begin{array}{lllll}
&\ds\oint \sum c_\alpha\eta^2\langle \nabla_\alpha \nabla^{(k)}_{odd}A, \nabla^{(k)}_{even}A\rangle\\
=&\ds-2\oint \eta  \sum c_\alpha \nabla_\alpha \eta \langle  \nabla^{(k)}_{odd}A, \nabla^{(k)}_{even}A\rangle-\oint \eta^2\langle  \nabla^{(k)}_{odd}A,   \sum c_\alpha\nabla_\alpha\nabla^{(k)}_{even}A\rangle
\end{array}
$$
Thus, we could estimate from  Lemma  \ref{lemma2.7}
$$
\begin{array}{lllll}
\ds\oint \sum c_\alpha\eta^2\langle \nabla_\alpha \nabla^{(k)}_{odd}A, \nabla^{(k)}_{even}A\rangle&=&\ds O(\oint\frac { \eta}{r}( | \nabla^{(k)}A|(1+ \sum_l |\nabla^{(l)}Rm||\nabla^{(k-2-l)}A|)))\\
&&\ds +O(\oint \eta^2( | \nabla^{(k)}A|(1+\sum_l |\nabla^{(l)}Rm||\nabla^{(k-1-l)}A|)))
\end{array}
$$
Our basic observation is that
$$
\begin{array}{lllll}
&\ds\oint \eta^2( | \nabla^{(k)}A|(1+ \sum_l|\nabla^{(l)}Rm||\nabla^{(k-1-l)}A|)) \\
\le &\ds Cr^2  \| \eta\nabla^{(k)}A\|_{L^3}+ \sum_l\| \eta\nabla^{(k)}A\|_{L^3} \|\nabla^{(l)}Rm\|_{L^3(B(p,r_l)\cap  M)} \|\eta |\nabla^{(k-1-l)}A|\|_{L^3}
\end{array}
$$
and
$$
\begin{array}{lllll}
&\ds\oint\frac { \eta}{r}( | \nabla^{(k)}A|(1+ \sum_l |\nabla^{(l)}Rm||\nabla^{(k-2-l)}A|))\\
\le&\ds  Cr  \| \eta\nabla^{(k)}A\|_{L^3}+\sum_l \frac1r \| \eta\nabla^{(k)}A\|_{L^3} \|\nabla^{(l)}Rm\|_{L^3(B(p,r_l)\cap  M)} \| \nabla^{(k-2-l)}A\|_{L^3}
\end{array}
$$
where $r_l=r/2+r/2^{l+2}$.  By the Sobolev trace inequality (\ref{Trace}) and from the induction, we get for any $l< k$
$$
 \|\nabla^{(l)}Rm\|_{L^3(B(p,r_l)\cap M)}^2\le   \frac{C}{r^{2l+2}}\left(\int_{B(p,r)} |Rm|^2 dV_g+\oint_{B(p,r)\cap M} |S|+r^4\right)
$$
$$
\begin{array}{lllll}
 \|\eta\nabla^{(k)}A\|_{L^3(M)}^2&\le&  C  \|\eta\nabla^{(k+1)}A\|_{L^2(X)}^2+ C\|\nabla \eta\otimes\nabla^{(k)}A\|_{L^2(X)}^2 \\
 &\le&  \ds C  \|\eta\nabla^{(k+1)}A\|_{L^2(X)}^2+   \frac{C}{r^{2k+2}}(\int_{B(p,r)} |Rm|^2 dV_g+\oint_{B(p,r)\cap M} |S|+r^4)
 \end{array}
 $$
which implies from the Cauchy-Schwarz inequality
$$
\begin{array}{lllll}
&\ds  \| \eta\nabla^{(k)}A\|_{L^3} \|\nabla^{(l)}Rm\|_{L^3} \|\eta |\nabla^{(k-1-l)}A|\|_{L^3}\\
\le&\ds \frac1\gamma  \| \eta\nabla^{(k)}A\|_{L^3}^2+ \frac \gamma 4 ( \|\nabla^{(l)}Rm\|_{L^3} \|\eta |\nabla^{(k-1-l)}A|\|_{L^3})^2\\
\le&\ds \frac1\gamma  \|\eta\nabla^{(k+1)}A\|_{L^2(X)}^2+   \frac{C}{r^{2k+2}}\left(\int_{B(p,r)} |Rm|^2 dV_g+\oint_{B(p,r)\cap M} |S|+r^4\right)
\end{array}
$$
and
$$
r \| \eta\nabla^{(k)}A\|_{L^3}\le Cr^2+ \frac1\gamma  \|\eta\nabla^{(k+1)}A\|_{L^2(X)}^2+   \frac{C}{r^{2k+2}}\left(\int_{B(p,r)} |Rm|^2 dV_g+\oint_{B(p,r)\cap M} |S|+r^4\right)
$$
and also
$$
\begin{array}{lllll}
&\ds \frac 1r \| \eta\nabla^{(k)}A\|_{L^3} \|\nabla^{(l)}Rm\|_{L^3} \|\eta |\nabla^{(k-2-l)}A|\|_{L^3}\\
\le&\ds \frac1\gamma  \|\eta\nabla^{(k+1)}A\|_{L^2(X)}^2+   \frac{C}{r^{2k+2}}\left(\int_{B(p,r)} |Rm|^2 dV_g+\oint_{B(p,r)\cap M} |S|+r^4\right)
\end{array}
$$
On the other hand, we have
$$
\begin{array}{lllll}
\ds 
&&\ds\int | \nabla^{(k+1)} A|| \nabla^{(k)}A|\eta|\nabla \eta|\\
&\le &\ds \frac1\gamma  \|\eta\nabla^{(k+1)}A\|_{L^2(X)}^2 + C\int | \nabla^{(k)}A|^2|\nabla \eta|^2\\
&\le&\ds \frac1\gamma  \|\eta\nabla^{(k+1)}A\|_{L^2(X)}^2+   \frac{C}{r^{2k+2}}\left(\int_{B(p,r)} |Rm|^2 dV_g+\oint_{B(p,r)\cap M} |S|+r^4\right)
\end{array}
$$
and
$$
\begin{array}{lllll}
\ds 
&&\ds\sum_{l=1}^{k-1}\int \eta^2|\nabla^{(l)}Rm||\nabla^{(k-l)}A|| \nabla^{(k)}A|\\
&\le &\ds \sum_{l=1}^{k-1}\|\nabla^{(l)}Rm\|_{L^4(B(p,r_l))} \|\eta \nabla^{(k-l)}A\|_{L^4}\|\eta \nabla^{(k)}A\|_{L^2}\\
&\le&\ds \frac1\gamma  \|\eta\nabla^{(k+1)}A\|_{L^2(X)}^2+   \frac{C}{r^{2k+2}}\left(\int_{B(p,r)} |Rm|^2 dV_g+\oint_{B(p,r)\cap M} |S|+r^4\right)
\end{array}
$$
Here the constant $C$ depends also on the $\gamma$. It remains to treat $\ds\int \eta^2\nabla^{(k)}Rm*A* \nabla^{(k)}A$ and $\ds \int \eta^2Rm*\nabla^{(k)}A* \nabla^{(k)}A$.  For the term $\int \eta^2\nabla^{(k)}Rm*A* \nabla^{(k)}A$,  using the Sobolev inequality (\ref{Sobolev}), H\"older's inequality and Cauchy-Schwarz inequality
$$
\begin{array}{lllll}
&&\ds\int \eta^2|\nabla^{(k)}Rm* A* \nabla^{(k)}A)|\\
&\le &\ds \|\eta \nabla^{(k)}Rm\|_{L^2}\|\eta \nabla^{(k)}A\|_{L^4} \| A\|_{L^4(B(p,r_1))}\\
&\le&\ds \frac1\gamma  \|\eta\nabla^{(k)}A\|_{L^4(X)}^2+C\|\eta  \nabla^{(k)}Rm\|_{L^2}^2 \| A\|_{L^4(B(p,r_1))}^2\ \\
&\le&\ds   \frac C\gamma  \|\eta\nabla^{(k+1)}A\|_{L^2(X)}^2+\frac{C}{r^{2k+2}}\left(\int_{B(p,r)} |Rm|^2 dV_g+\oint_{B(p,r)\cap M} |S|+r^4\right)
\end{array}
$$
Similarly, we have
$$
\begin{array}{lllll}
&&\ds 
\int \eta^2|Rm*\nabla^{(k)}A* \nabla^{(k)}A|\\
&\le&\ds  \|Rm\|_{L^4(B(p,r_1))}\|\eta \nabla^{(k)}A\|_{L^4}\|\eta \nabla^{(k)}A\|_{L^2}\\
&\le &\ds \frac 1\gamma \|\eta \nabla^{(k)}A\|_{L^4}^2+ C   \|Rm\|_{L^4(B(p,r_1))}^2 \|\eta \nabla^{(k)}A\|_{L^2}^2\\
&\le & \ds \frac C\gamma  \|\eta\nabla^{(k+1)}A\|_{L^2(X)}^2+   \frac{C}{r^{2k+2}}\left(\int_{B(p,r)} |Rm|^2 dV_g+\oint_{B(p,r)\cap M} |S|+r^4\right)
\end{array}
$$
Gathering all these estimates together, we deduce
$$
 \|\eta\nabla^{(k+1)}A\|_{L^2(X)}^2 \le \frac12  \|\eta\nabla^{(k+1)}A\|_{L^2(X)}^2+   \frac{C}{r^{2k+2}}\left(\int_{B(p,r)} |Rm|^2 dV_g+\oint_{B(p,r)\cap M} |S|+r^4\right)
$$
provided $\gamma$ is a sufficiently large constant. Therefore
$$
 \|\eta\nabla^{(k+1)}A\|_{L^2(X)}^2 \le  \frac{C}{r^{2k+2}}\left(\int_{B(p,r)} |Rm|^2 dV_g+\oint_{B(p,r)\cap M} |S|+r^4\right)
$$
By the Sobolev inequality (\ref{Sobolev}),  we get 
$$
\|\eta\nabla^{(k)}A\|_{L^4(X)}^2 \le  \frac{C}{r^{2k+2}}\left(\int_{B(p,r)} |Rm|^2 dV_g+\oint_{B(p,r)\cap M} |S|+r^4\right).
$$
It is similar for the Weyl tensor. From the Bach flat equation (\ref{Bachflat2}), we have
\beq
\label{Bachflat3}
\begin{array}{lllll}
\ds 
\triangle \nabla^{(k)} W_{ijml} &=&\ds 2 ( \nabla_i\nabla_m \nabla^{(k)} A_{jl}-\nabla_i\nabla_l \nabla^{(k)} A_{jm}-\nabla_j\nabla_m\nabla^{(k)} A_{il}+\nabla_j\nabla_l \nabla^{(k)}A_{im})  \\
&&\ds + \sum_{l=0}^{k} (\nabla^{(l)}W*\nabla^{(k-l)}W+\nabla^{(l)}W*\nabla^{(k-l)}A+\nabla^{(l)}A*\nabla^{(k-l)}A)
\end{array}
\eeq
As before, we take $\eta^2\nabla^{(k)} W$ as test tensor and integrate the equality. Thus, we have
$$
\begin{array}{lllll}
\ds\int\eta^2| \nabla^{(k+1)} W|^2 &\le& \ds C\int (| \nabla^{(k+1)} A|+| \nabla^{(k+1)} W|)| \nabla^{(k)}W|\eta|\nabla \eta|\\
&&\ds + C \sum_{l=0}^{k}\left[\left|\int \eta^2\nabla^{(l)}W*\nabla^{(k-l)}A*\nabla^{(k)}W\right|+\left|\int \eta^2\nabla^{(l)}W*\nabla^{(k-l)}A*\nabla^{(k)}A\right|\right.\\
&&\ds +\left|\int \eta^2\nabla^{(l)}A*\nabla^{(k-l)}A*\nabla^{(k)}A\right|+\left|\int \eta^2\nabla^{(l)}W*\nabla^{(k-l)}W*\nabla^{(k)}A\right|\\
&&\ds+\left.\left|\int \eta^2\nabla^{(l)}W*\nabla^{(k-l)}W*\nabla^{(k)}W\right|+\left|\int \eta^2\nabla^{(l)}A*\nabla^{(k-l)}A*\nabla^{(k)}W\right| \right]\\
&&\ds+C\left|\oint \eta^2 \langle\nabla_n \nabla^{(k)}  W, \nabla^{(k)}  W \rangle\right|+C\left|\oint \eta^2 \langle\nabla_n \nabla^{(k)}  A, \nabla^{(k)}  W \rangle\right|\\
&&\ds+C\left|\oint \eta^2 \langle \nabla^{(k)}  A, \nabla_n \nabla^{(k)}  W \rangle\right|
\end{array}
$$
Here we use the above two Bach flat equations (\ref{Bachflat1bis}) and (\ref{Bachflat3}) and Ricci identity (\ref{Ricci}). With the similar arguments, we can bound the boundary terms as above
$$
\begin{array}{lllll}
&\ds C\left|\oint \eta^2 \langle\nabla_n \nabla^{(k)}  W, \nabla^{(k)}  W \rangle\right|+C\left|\oint \eta^2 \langle\nabla_n \nabla^{(k)}  A, \nabla^{(k)}  W \rangle\right|+C\left|\oint \eta^2 \langle \nabla^{(k)}  A, \nabla_n \nabla^{(k)}  W \rangle\right|\\
\le &\ds    \frac{C}{r^{2k+2}}\left(\int_{B(p,r)} |Rm|^2 dV_g+\oint_{B(p,r)\cap M} |S|+r^4\right)+\frac 14 \int\eta^2| \nabla^{(k+1)} W|^2
\end{array}
$$
And also from the induction and results for $A$ and H\"older's and Cauchy-Schwarz inequalities
$$
\begin{array}{lllll}
& \ds C\int (| \nabla^{(k+1)} A|+| \nabla^{(k+1)} W|)| \nabla^{(k)}W|\eta|\nabla \eta|\\
\le &\ds\frac 14 \int\eta^2| \nabla^{(k+1)} W|^2 +C \int\eta^2| \nabla^{(k+1)} A|^2+ C \int|\nabla \eta^2|^2 |\nabla^{(k)} W|^2\\
\le &\ds    \frac{C}{r^{2k+2}}\left(\int_{B(p,r)} |Rm|^2 dV_g+\oint_{B(p,r)\cap M} |S|+r^4\right)+\frac 14 \int\eta^2| \nabla^{(k+1)} W|^2
\end{array}
$$
and
$$
\begin{array}{lllll}
&\ds C \sum_{l=0}^{k}\left[\left|\int \eta^2\nabla^{(l)}W*\nabla^{(k-l)}A*\nabla^{(k)}W\right|+\left|\int \eta^2\nabla^{(l)}W*\nabla^{(k-l)}A*\nabla^{(k)}A\right|\right.\\
&\ds +\left|\int \eta^2\nabla^{(l)}A*\nabla^{(k-l)}A*\nabla^{(k)}A\right|+\left|\int \eta^2\nabla^{(l)}W*\nabla^{(k-l)}W*\nabla^{(k)}A\right|\\
&\ds+\left.\left|\int \eta^2\nabla^{(l)}W*\nabla^{(k-l)}W*\nabla^{(k)}W\right|+\left|\int \eta^2\nabla^{(l)}A*\nabla^{(k-l)}A*\nabla^{(k)}W\right| \right]\\
\le &\ds C  \sum_{l=1}^{k-1}[\| \nabla^{(l)}Rm\|_{L^4}\|\eta \nabla^{(k-l)}Rm\|_{L^4}\|\eta \nabla^{(k)}Rm \|_{L^2}]\\
&\ds + C\| Rm\|_{L^4} \|\eta \nabla^{(k)}Rm \|_{L^2}(\|\eta \nabla^{(k)}W \|_{L^4}+\|\eta \nabla^{(k)}A \|_{L^4})\\
\le &\ds    \frac{C}{r^{2k+2}}\left(\int_{B(p,r)} |Rm|^2 dV_g+\oint_{B(p,r)\cap M} |S|+r^4\right)+\frac 14 \int\eta^2| \nabla^{(k+1)} W|^2
\end{array}
$$
Finally, we infer
$$
 \int\eta^2| \nabla^{(k+1)} W|^2\le    \frac{C}{r^{2k+2}}\left(\int_{B(p,r)} |Rm|^2 dV_g+\oint_{B(p,r)\cap M} |S|+r^4\right)
$$
which implies from the Sobolev inequality (\ref{Sobolev})
$$
\|\eta\nabla^{(k)}W\|_{L^4(X)}^2 \le  \frac{C}{r^{2k+2}}\left(\int_{B(p,r)} |Rm|^2 dV_g+\oint_{B(p,r)\cap M} |S|+r^4\right)
$$
We have thus finished the proof of Theorem \ref{epsilonregularity}.
\end{proof}

\begin{theo}
\label{epsilonregularity1}
Under the same assumptions as Theorem \ref{epsilonregularity}, we have the estimates for $L^\infty$ norm, that is, there exists constants $\varepsilon$ (independent of $k$) and $C$ (depending on $k$) such that if 
  $$
   \|Rm\|_{L^2(B(p,r))}\le \varepsilon
   $$
   then for any $r<1$
\beq
\label{Infinitybound}
 \sup_{B(p,r/2)} |\nabla^{k-2} Rm| \le  \frac{C}{r^{k}}\left(\int_{B(p,r)}|Rm|^2 dV_g+\oint_{B(p,r)\cap M} |S|+ r^4\right)^{\frac12}
\eeq  
\end{theo}
 
We recall a technique result, which can be found in \cite{Streets,Bour, LSU}.\\ 
 
\begin{lemm}
\label{lemma3.3}
Under the  assumption (4) as in Theorem \ref{maintheorem}, there exists constants $\varepsilon$  and $C$  such that if 
  $$
   \|Rm\|_{L^2(B(p,r))}\le \varepsilon
   $$
   then for any Lipschitz function $f$ with the compact support in $B(p,r)$
\beq
\label{Sobolev1}
\left(\int_X f^4 dV_g\right)^{1/2}\le  C \int_{X} |\nabla f|^2 dV_g
\eeq
Moreover, we have
\beq
\label{Interpolation}
\|f\|_{L^p}\le C' \|f\|_{L^m}^{1-\gamma} \|\nabla f\|_{L^q}^{\gamma}
\eeq
where  $\gamma= \frac{\frac{1}{m}- \frac{1}{p}}{\frac{1}{4}-\frac{1}{q}+ \frac{1}{m}}$.  Here, when  $q>4$, we can let $+\infty\ge p\ge m\ge 2$ and  $C'$ is some constant depending on $m,p, q,C$; when  $q=4$, we can let $+\infty> p\ge m\ge 2$ and  $C'$ is some constant depending on $m,p,C$.
\end{lemm}

\begin{proof}[Proof of Theorem  \ref{epsilonregularity1}]
Recall a basic fact that for any tensor $T$, we have
$$
|\nabla |T||\le |\nabla T|
$$ 
We will show first $\nabla^{k-1} Rm\in L^{8}(B(p,r/2))$. Let $\eta$ be some cut-off function such that $supp(\eta)\subset B(p,3r/4)$ and $\eta\equiv 1$ on $B(p,r/2)$ and $|\nabla \eta|\le C/r$. From Lemma \ref{lemma3.3}
$$
 \|\eta |\nabla^{k-1} Rm|\|_{L^8}^2\le C \|\eta |\nabla^{k-1} Rm|\|_{L^4}\|\nabla (\eta |\nabla^{k-1} Rm|)\|_{L^4}
$$
Applying Theorem \ref{epsilonregularity}, we have
$$
\|\eta |\nabla^{k} Rm|\|_{L^4}\le \frac{C}{r^{1+k}}\left(\int_{B(p,r)} |Rm|^2 dV_g+\oint_{B(p,r)\cap M} |S|+r^4\right)^{\frac12}
$$
$$
\begin{array}{lll}
\|\nabla (\eta |\nabla^{k-1} Rm|)\|_{L^4}&\le&\ds C(\frac1r \| \nabla^{k-1} Rm\|_{L^4(B(p,3r/4))}+\| \eta \nabla^{k} Rm\|_{L^4})\\
&\le&\ds  \frac{C}{r^{1+k}}\left(\int_{B(p,r)} |Rm|^2 dV_g+\oint_{B(p,r)\cap M} |S|+r^4\right)^{\frac12}
\end{array}
$$
which  yields
$$
\|\eta |\nabla^{k-1} Rm|\|_{L^8}^2\le \frac{C}{r^{1+2k}}\left(\int_{B(p,r)} |Rm|^2 dV_g+\oint_{B(p,r)\cap M} |S|+r^4\right)
$$
By the same argument,
$$
\|\eta | \nabla^{k-2}Rm|\|_{L^8}^2\le \frac{C}{r^{2k-1}}\left(\int_{B(p,r)} |Rm|^2 dV_g+\oint_{B(p,r)\cap M} |S|+r^4\right)
$$
Again from Lemma \ref{lemma3.3}
$$
\begin{array}{lll}
\|\eta | \nabla^{k-2} Rm|\|_{L^\infty}^2&\le&\ds C \|\eta | \nabla^{k-1} Rm|\|_{L^8}\|\eta |  \nabla^{k-2} Rm|\|_{L^8} \\
&\le &\ds\frac{C}{r^{2k}}\left(\int_{B(p,r)} |Rm|^2 dV_g+\oint_{B(p,r)\cap M} |S|+r^4\right)
\end{array}
$$
This gives the desired estimate (\ref{Infinitybound}), which establishes Theorem  \ref{epsilonregularity1}.
\end{proof}

We now derive a better regularity result in the interior of the manifold.

\begin{theo}
\label{epsilonregularity2}
Suppose the assumption (4) in Theorem \ref{maintheorem} are satisfied and the metric is $Q$-flat and Bach-flat. Assume for some $r>0$ and for some $p$ with $B(p,r)\subset \bar X\setminus M$, 
then there exist constants $\varepsilon$ (independent of $k$) and $C$ (depending on $k$) such that if 
  $$
   \|Rm\|_{L^2(B(p,r))}\le \varepsilon
   $$
   then
    $$
\left\{\int_{B(p,r/2)} |\nabla^{k}A|^4 dV_g\right\}^{1/2}\le  \frac{C}{r^{2k+2}}\int_{B(p,r)} |A|^2 dV_g
   $$
    $$
\int_{B(p,r/2)} |\nabla^{k+1} A|^2 dV_g\le  \frac{C}{r^{2k+2}}\int_{B(p,r)}|A|^2 dV_g
   $$
    $$
\left\{\int_{B(p,r/2)} |\nabla^k Rm|^4 dV_g\right\}^{1/2}\le  \frac{C}{r^{2k+2}}\int_{B(p,r)} |Rm|^2 dV_g
   $$
    $$
\int_{B(p,r/2)} |\nabla^{k+1} Rm|^2 dV_g\le  \frac{C}{r^{2k+2}}\int_{B(p,r)}|Rm|^2 dV_g
   $$
 Moreover, we have
 $$
 \sup_{B(p,r/2)} |\nabla^{k-2} Rm| \le  \frac{C}{r^{k}}\left(\int_{B(p,r)}|Rm|^2 dV_g\right)^{\frac12}\le \frac{C\varepsilon}{r^{k}}
 $$  
\end{theo}
\begin{proof} The proof is as same as the one of Theorems \ref{epsilonregularity} and \ref{epsilonregularity1}. We just remark that there is no boundary term in the estimates of the corresponding inequalities now.
\end{proof}

\section{blow-up analysis}

\subsection{Statement of the results}
\label{sb4.1}

\noindent\\
In this section, we will do blow up analysis both on the boundary and in the interior. Recall the Fefferman-Graham's compactification $g_i=e^{2w_i}g_i^+=v_i^{-2}g_i^+$ where $g_i^+$ is a  conformally compact Einstein metric and the defining function $  w_i$ solves the equation (\ref{eq1.1}). We aim to prove the curvature tensor for the FG metric is uniformly bounded, namely,

\begin{theo}
\label{Boundgeometry}
Under the same assumptions as in Theorem \ref{maintheorem} (or Theorem \ref{maintheorem1}), there exists some positive constant $C>0$ such that for all index $i$,  we have
\beq
\label{boundcurvature}
\|Rm_{g_i}\|_{C^{k-2}}\le C
\eeq
Moreover, we have
\beq
\label{boundcurvaturebis}
\|Rm_{g_i}\|_{C^{k+1}}\le C
\eeq
\vs
\end{theo}

Our  strategy to prove Theorem \ref{Boundgeometry} as follows:  to get the uniform $C^1$ bound of the curvature tensor of metrics $g_i$, we will first prevent the boundary blow up in
section  \ref{sb4.2}. We then use this fact to help us to rule out the interior
blow-up in section \ref{sb4.3}.  After that we apply Theorems \ref{epsilonregularity1} and \ref{epsilonregularity2} to get
uniform $C^{k-2}$ of the curvature tensor and in section \ref{sb4.4} applying the
Bach flat and Q-flat equations to improve the regularity of the
curvature tensor and get the uniform boundedness of their 
$C^{k+1}$ norm.

A property of  Q-flat metrics $g_i$ in our setting is that,
under the assumption that the scalar curvature of their boundary metric $\hat{g_i}$ 
is non-negative, the  scalar curvature of $g_i$ is positive. This fact was first proved in Chang-Case \cite{casechang}. For the convenience  of readers, here we present also the proof of the result. \\

\begin{lemm}\cite{casechang}
\label{lem4.1}
Let $(X,\p X, g^+)$ be a $4$-dimensional conformal compact Einstein manifold and $g=e^{2w}g^+$ be the Fefferman-Graham's compactification, that is, $-\triangle_{g^+}w=3$. Assume the representative of conformal infinity $h=g|_{T\p X}$ has non-negative scalar curvature. Then the scalar curvature $R_g>0$ is positive in $X$.
\end{lemm}

\begin{proof}
Let $n=3$ be the dimension of the boundary of $X$. As in \cite{GZ}, for any $s\in \mathbb C$ with $\mbox{Re}(s)>n/2$ and $s\not\in n/2+\mathbb{N}/2$ and for any given $f\in C^\infty(\p X)$, we consider the following Poisson equations
$$
\left\{
\begin{array}{lll}
-\triangle_{g^+}v_s-s(n-s)v_s=0& \mbox{ in }X\\
v_s=F_s r^{n-s}+G_s r^s,&F_s,G_s\in C^\infty(X)\\
F_s=f& \mbox{ on }\p X\\
\end{array}
\right.
$$
where $r$ is some special defining function with respect to $h$ some representative of the conformal boundary. Thus the unique Poisson operators can be defined as ${\mathcal P}(s)f:= v_s$ provided  $s(n-s)\not\in \sigma_{pp}(-\triangle_{g^+})$ the essential spectrum of $-\triangle_{g^+}$. ${\mathcal P}(s)$ is meromorphic and could  extend holomorphically across $s\not\in n/2+\mathbb{N}/2$ provided  $s(n-s)\not\in \sigma_{pp}(-\triangle_{g^+})$. Here we are interested in $f\equiv 1$ on $\p X$. For $s=n+1$, the  solution $v_{n+1}$ satisfies (see \cite{Lee}):\\
1) $v=r^{-1}+R_h r/4n(n-1)+O(r^2)$ is positive;\\
2) $v_{n+1}^2-|\nabla_{g^+} v_{n+1}|^2=R_h /n(n-1)$ on $\p X$;\\
3) $-\triangle_{g^+} (v_{n+1}^2-|\nabla_{g^+} v_{n+1}|^2)=-2|\nabla^2_{g^+} v_{n+1}-v_{n+1} g^+|^2\le 0$;\\
where $R_h=\hat{R}$ is the scalar curvature of the metric $h$ on the boundary. We call the compactified metric $g_{n+1}^*:= v_{n+1}^{-2} g^+$ the Jack Lee's metric. Moreover, by the observation in \cite{CQY04}, we have the scalar curvature  $R_{g_{n+1}^*}=n(n+1)(v_{n+1}^2-|\nabla_{g^+} v_{n+1}|^2)$. Assume $R_h\ge 0$ on the boundary, using the Maximum principle, Jack Lee \cite{Lee} proved $v_{n+1}^2-|\nabla_{g^+} v_{n+1}|^2\ge 0$ and the first eigenvalue $\lambda_1( -\triangle_{g^+})\ge \frac{n^2}4$, that is, the spectrum of $ -\triangle_{g^+}$ is bounded below by $ \frac{n^2}4$. Together with the observation in  \cite{CQY04}, we have also $R_{g_{n+1}^*}\ge 0$. As a consequence, using the result due to Graham-Zworski \cite{GZ}, $v_s$ is holomorphic in $\mbox{Re}(s)>n/2$. Moreover, we have $v_s>0$ provided $s\in (n/2, +\infty)$. For any $s\in (n/2, n)\cup(n,n+1]$, we consider the compactified metrics $g_s^*:=y_s^2 g^+$ where $y_s:=(v_s)^{1/(n-s)}$. For $s=n+1$, $g_s^*$ is just Jack Lee's metric. When $s=n=3$, we define $g_s^*:=e^{2w}  g^+$ where $w:=-\frac{d}{ds}v_s|_{s=3}$. We note $v_3\equiv 1$ in $X$. It follows from the Poisson equations that 
$$
-\triangle_{g^+} w=3  \mbox{ in }X
$$
Thus, $g_s^*$ is an analytic family of metrics for $s\in (n/2, +\infty)$. Recall $J_{g_s^*}=\frac{R_{g_s^*}}{6}$ the trace of Schouten tensor in $X$ and $J_{h}=\frac{R_h}{4}$  the one on the boundary $\partial X$. Set $\gamma=s-n/2$. The direct calculations lead to for any $s\in (n/2,n) \cup (n,\infty)$
$$
J_{g_s^*}=\frac{1-2\gamma}{2} y_s^{-2}(|\nabla_{g_s^*} y_s|^2-1)\mbox{ and }J_{g_s^*}|_{\p X}=\frac{2\gamma-1}{2(\gamma-1)}J_h
$$
When $s=3$, $J_{g_s^*}=e^{-2w}(1-| \nabla_{g^+} w|^2)$ and $J_{g_s^*}|_{\p X}=2J_h$ (we set $y_s=e^w$). Moreover, we have the expansion near the boundary $\p X$, $w=\log r +A+Br^3$ where $A=-\frac12 J_h r^2+O(r^4)$ and $B=B_0+B_2 r^2+O(r^4)$ are even expansions. Here $B_0$ is $Q_3$ curvature on the boundary up to a multiple (see \cite{FG}). That is, $g_3^*$ is just Fefferman-Graham compactification.  We set $m=3-2\gamma$. By the formula (6.6) in \cite{casechang}, we have the following equations for $J_s:=J_{g_s^*}$ for $s>n/2$ and $s\neq 3$
\beq
\label{eqQ}
 y_s^{-m}\div_{g_s^*}( y_s^{-m}\nabla_{g_s^*} J_s)=-c_1(s)|E|_{g_s^*}^2+c_2(s)J_s^2
\eeq
where $E=Ric_{g_s^*}-\frac{R_{g_s^*}}{4}g_s^*$ is traceless Ricci tensor, $c_1(s):=\frac{2s-n-1}{(n-1)^2}$ and $c_2(s):=\frac{4s(n+1-s)}{(n+1)(2s-n-1)}$. When $s=3$, such equation can be read as
\beq
\label{eqQbis}
 \triangle_{g_s^*}J_s=-\frac12|E|_{g_s^*}^2+\frac32J_s^2
\eeq
Recall $Q_4(g_s^*)\equiv 0$ when $s=3$. By the continuous method, Case-Chang \cite{casechang} show $J_s> 0$ in $X$ provided $J_h\ge 0$ when $s>\frac52$. For this purpose, we define the set $I=\{s\in (\frac52, 4], J_t>0 \mbox{ in }X\; \;\forall t\in[s,4] \}$. Using the Lee's result (\cite{Lee}, see also Lemma 2.2\cite{GQ}) $J_4>0$ in $X$. Thus $4\in I$. As $g_s^*$ is an analytic family, we claim that $I$ is open. Indeed, assume $s\in I$, there exists some $\varepsilon>0$ such that for any $t\in (s-\varepsilon,s)$, $J_t>0 \mbox{ in }X$. Otherwise, there exists a increasing sequence $\{t_n\}\subset (\frac52, 4]$ such that $t_n\uparrow s$ and a sequence of points $\{x_n\}\subset X$ such that $J_{t_n}(x_n)\le 0$. Since $\bar X$ is compact, up to a subsequence, we have $x_n\to y\in \bar X$. As $g_s^*$ is an analytic family, we have $J_{s}(y)\le 0$. Hence $J_{s}(y)= 0$ and $y\in\p X$ since  $J_s>0 \mbox{ in }X$. Again from analyticity of $g_s^*$, we have some uniform Fermi coordinates $(z,r)\in\p X\times [0,\varepsilon)$ for any metric $g_{t_n}^*$ around the boundary $\p X$, that is $d_{g_{t_n}^*}((z,r),\p X)=r$.  Moreover, such geodesic tube $\p X\times [0,\varepsilon)$ contains a common neighborhood of $\p X$ for any $g_{t_n}^*$.  Set $V=\p_r$ the unitary vector field in $\p X\times [0,\varepsilon)$ whose restriction on the boundary $\p X$ is just insides normal vector field of the boundary.  We write $x_n=(z_n,r_n)$ so that $(z_n,0)$ is the orthogonal projection of $x_n$ on the boundary $\p X$. By the relation $J_s=\frac{2\gamma -1}{2(\gamma -1)}J_h$, we have $J_h(y)=0$. Note when $n=3$ and $s\in (\frac52, 4]$, we have $c_1(s)\ge 0$ and  $c_2(s)\ge 0$. Using the equations (\ref{eqQ}) and strong maximum principle, we have $V J_s(y)>0$. On the other hand, $J_{t_n}(x_n)<0$ and $J_{t_n}(z_n,0)\ge 0$ since $J_h\ge 0$ on the boundary. Thus, there exists $\bar r_n\in (0, r_n)$ such that 
$V J_{t_n}(z_n,\bar t_n)<0$. It is clear that $(z_n,\bar t_n)\to y$. By the analyticity of $J_s$ in $s$ and in space variables, we infer $VJ_s(y)\le 0$ which contradicts the fact $V J_s(y)>0$. Therefore, $I$ is open.

We use the idea in \cite{GQ} to prove $I$ is also closed. Let $s_n\in I\to s\in (\frac52, 4]$. Again from the fact $g_s^*$ is an analytic family, it follows that $J_s\ge 0$ in $X$. Using the equations (\ref{eqQ}) and strong maximum principle, we have $J_s>0$ in $X$ or $J_s\equiv 0$ in $X$. The idea due to \cite{GQ} can rule out the latter case. More precisely, the latter case implies 
$ |\nabla_{g_s^*} y_s|\equiv 1$ or equivalently, $ |\nabla_{g^+} y_s|\equiv y_s$. Hence the positive function $\phi:=( y_s)^s\in L^2$ satisfies
$$
-\triangle_{g^+ }\phi=s(n-s)\phi
$$
which contradicts $\lambda(-\triangle_{g^+ })\ge \frac{n^2}{4}$. As a consequence, $I=(\frac52, 4]$. Finally, we prove the desired result.\end{proof}

We call some {\bf basic facts} for the conformal metrics $g_i=e^{2w_i}g_i^+=v_i^{-2}g_i^+$.\\

\begin{coro}
\label{corollary4.2}
We have
\beq
\label{relation1add}
Ric_{g^+_i}=Ric_{g_i}+ 2v_i \nabla_{ g_i}^2 (v_i^{-1})+ (v \triangle_{ g_i} (v_i^{-1})-3v_i^2|\nabla_{ g_i}( v_i^{-1})|^2) g_i;
\eeq
\beq
\label{relation2add}
-12=R_{g^+_i}=v_i^{-2}(R_{ g_i}+6v_i \triangle_{ g_i} (v_i^{-1})-12v^2|\nabla_{ g_i}( v_i^{-1})|^2);
\eeq
\beq
\label{relation3}
R_{g_i}=6v_i^{2}(1-|\nabla_{ g_i}( v_i^{-1})|^2);
\eeq
\beq
\label{relation5}
Ric_{ g_i}=-2v_i \nabla^2_{g_i} (v_i^{-1}).
\eeq
\beq
\label{relation4}
-R_{ g_i}=2v_i \triangle_{ g_i} (v_i^{-1}).
\eeq
Moreover, there holds 
\beq
\label{primeestimate}
\|\nabla_{g_i}( v_i^{-1})\|\le 1.
\eeq
\end{coro}
\begin{proof}
The equalities (\ref{relation1add}) and (\ref{relation2add}) come from the conformal change. The equalities (\ref{relation3}) and (\ref{relation4}) are the results of (\ref{eq1.1}). In fact, (\ref{eq1.1}) is equivalent to 
\beq
\label{relation4bis1}
v_i{\triangle_{ g_i^+}v_i^{-1}}-|\nabla_{ g_i}v_i^{-1}|^2=-3.
\eeq
On the other hand, we have
$$
\triangle_{  g_i} v_i^{-1}=v_i^2{\triangle_{ g_i^+}v_i^{-1}}+2v_i^3 |\nabla_{ g_i^+} v_i^{-1}|^2=v_i^2{\triangle_{ g_i^+}v_i^{-1}}+2v_i |\nabla_{ g_i} v_i^{-1}|^2
$$
Thus, we infer
\beq
\label{relation4bis}
v_i^{-1}\triangle_{  g_i} v_i^{-1}=-3+3|\nabla_{  g_i} v_i^{-1}|^2
\eeq
Going back to  (\ref{relation2add}),  we get (\ref{relation3}) and (\ref{relation4}).\\
From Lemma \ref{lem4.1}, we know $R_{g_i}\ge 0$. Hence, by (\ref{relation3}), we get 
(\ref{primeestimate}) and finish the proof of Corollary \ref{corollary4.2}.
\end{proof}

We now outline the proof of Theorem \ref{Boundgeometry}. For the sequence of boundary
Yamabe metric $\hat{g_i}$ which is $C^{k+3}$ compact, we will first
establish the $C^{k-1, \alpha}$ compactness of the corresponding FG
metrics $g_i$ on the interior $X$; we will then apply a bootstrapping argument in section \ref{sb4.4} to establish the $C^{k+2,\alpha}$ compactness of the metrics
$g_i$.

We now notice that by the $\varepsilon$ regularity result 
established in Theorem  \ref{epsilonregularity1}, to establish the $C^{k-1, \alpha}$ compactness
of the metrics $g_i$, it suffices to prove the family is $C^1$ bounded. For this purpose, we see that the uniform boundness of  $C^1$ norm for the curvature $\|Rm_{g_i}\|_{C^1}$ induces the uniform  boundness of $L^2$ norm for the curvature $Rm_{g_i}$ on the ball $B(p,r)$ and also the uniform  boundness of $L^1$ norm for the $S$-tensor $S_i$ on the boundary $M\cap B(p,r)$.  On the other hand,  thanks to Bishop-Gromov volume comparison Theorem, we have the estimate $vol(B(p,r))\le Cr^4$ once the curvature is uniformly bounded for the metrics $g_i$. We will
now begin to establish this assertion by a contradiction argument. 
Assume the family $g_i$ is not $C^1$ bounded, or equivalently, the 
the $C^1$ norm of its curvature tends to infinity as $i$ tends to infinity, that is,
$$
\|Rm_{g_i}\|_{C^1}\to \infty\quad\quad\mbox{ as }i\to \infty
$$
We rescale the metric
$$
\bar g_i= K_i^2 g_i
$$
where there exists some point $p_i\in X$ such that
$$
K_i^2=\max\{\sup_X |Rm_{g_i}|, {\sup_X |\nabla Rm_{g_i}|^{2/3}}\}= |Rm_{g_i}|(p_i) (\mbox{ or } { |\nabla Rm_{g_i}|^{2/3}(p_i)})
$$
We mark the point $p_i$ as $0\in X$. Thus, we have 
$$
|Rm_{\bar g_i}|(0)=1 \mbox{ or }  |\nabla Rm_{\bar g_i}|(0)=1
$$
We denote the corresponding defining function $\bar v_i^{-1}= K_ie^{w_i}$, that is, $\bar g_i =\bar v_i^{-2}g^+_i$. \\
We observe  $C^{1}$ norm for the curvature $\|Rm_{\bar g_i}\|_{C^1}$ is uniformly bounded. Applying $\varepsilon$-regularity (Theorems \ref{epsilonregularity1} and \ref{epsilonregularity2}),  we obtain $C^{k-2}$ norm for the curvature $\|Rm_{\bar g_i}\|_{C^{k-2}}$ is also uniformly bounded.\\
On the other hand, we claim there is no collapse for the family of metrics $\{\bar g_i\}$.  To see this, we observe first the Sobolev inequality (\ref{Sobolev}) implies $B_{\bar g_i}(x,r)\ge cr^4$ for all $x\in \bar X$ and $r<1$. As the boundary $M$ is totally geodesic, we can use the doubling argument along the boundary $M$  to get a compact manifold without boundary $Y:=X\bigcup_M X_1$ where $X_1$ is a copy of $X$ with opposite orientation. On the closed manifold $Y$, we have a natural metric $(Y,\tilde g_i)$ extending $g_i$ on $X$. $(Y,\tilde g_i)$ is a $C^{2,\alpha}$ closed Riemannian manifold with any $\alpha\in (0,1)$.  Thanks to a result of Cheeger-Gromov-Taylor \cite{CGT}, we have the uniform lower bound for any closed simple geodesic on $(Y,\tilde g_i)$, which yields that  both the interior injectivity radius and the boundary injectivity radius on $X$ are uniformly bounded from below; thus we have proved the claim.  We now notice by a version of Cheeger-Gromov-Hausdorff's compactness theorem for the manifolds with boundary (for the convenience of readers, we will give more details in the Appendix in Lemma \ref{highorder}), modulo a subsequence and modulo diffeomorphism group, $\{\bar g_i\}$ converges in pointed Gromov-Hausdorff's sense for $C^{k-1,\alpha}$ norm to a non-flat  limit metric $\bar g_\infty$ with totally geodesic boundary whose doubling metric is complete.\\

We now have two types of blow-up. \\

Type (I) : Boundary blow-up 
\vskip .1in
If there exists two positive constants $C,\Lambda>0$ and a sequence of points $p_i\in X$ such that $d_{\bar g_i}(\p X, p_i)\le C$ and $\liminf_i |Rm_{\bar g_i}(p_i)|+|\nabla Rm_{\bar g_i}(p_i)| = \Lambda>0$. In such case, we translate $p_i$ to $0$ and do blow up analysis. Modulo a subsequence, we can assume
$$
\lim_i|Rm_{\bar g_i}(0)|+ |\nabla Rm_{\bar g_i}(0)|=\Lambda < \infty.
$$
 
In section \ref{sb4.2} below, we will do analysis on this type of blow-up and show
that it does not occur.\\

Type (II) : Interior blow up\\

If blow -up of type I does not occur; that is, $d_{\bar g_i}(\p X, p_i)$ tends
to infinity as $i$ tends to infinity. By applying some result in section \ref{sb4.2} ,
together with the topological assumptions we have made in the statements of
Theorems \ref{maintheorem} and \ref{maintheorem1}.
; we will show in section \ref{sb4.3}  below this type of blow up cannot occur either.

Finally in section \ref{sb4.4}, we will combine the results in section \ref{sb4.2}  and \ref{sb4.3} , together with some bootstrapping arguments to finish the proof of Theorems \ref{maintheorem} and \ref{maintheorem1}.\\

We remark the basic facts stated in Corollary \ref{corollary4.2} for the pair $ (v_i, g_i)$ continue to hold for the pair $(\bar v_i, \bar g_i)$ w.r.t to the same base metric $g_i^{+}$.

\subsection{blow-up analysis on the boundary}
\label{sb4.2}

\noindent\\

Under the assumptions of  Theorem \ref{maintheorem}, if we suppose  $K_i\to\infty$, we have asserted  in section \ref{sb4.1} that $(X,\bar g_i)$ converges to  $(X_\infty,g_\infty)$ in $C^{k-1,\alpha}$  norm in Gromov-Hausdorff sense  for some $k\ge 2$ and for all $\alpha\in (0,1)$, which is a manifold with totally geodesic boundary.\\

The following Lemma is the main part of this section, actually the key argument in this paper.

\begin{lemm} 
\label{lemma4.3}
Under the assumptions of  Theorem \ref{maintheorem} and assuming that $\bar g_i$ has the type I boundary blow up,  $g_\infty$ is conformal to hyperbolic space form.
\end{lemm}

\begin{proof} 
We divide the proof in several steps.\\

{\bf Step 1.} Claim: There exists some $C>0$ such that $\bar v_i^{-1}\ge C$ provided $d_{ \bar g_i}(x,\partial X)\ge A$ and $\bar v_i(x)d_{\bar g_i}(x,\partial X)\le C$ provided 
$0\le d_{ \bar g_i}(x,\partial X)\le A$ where $d_{ \bar g_i}$ is the Riemann distance function w.r.t. the metric $\bar g_i$. Here $A$ is some uniform constant smaller than the boundary injectivity radius. \\

Without loss of generality, we assume $A=1$. Let us denote by $r_i(x):=d_{ \bar g_i}(x,\partial X)$ the distance function to the boundary.  From the Chapter 3 in \cite{Chavel}, we have the elementary properties $|\nabla_{ \bar g_i} r_i|=1$ and $|\triangle_{ \bar g_i} r_i|$ is bounded in the tube neighborhood of the boundary $\{x, r_i(x)\le 1\}$ since the boundary is totally geodesic and $\bar g_i$ has the bounded curvature.  We could take a cut-off function $\eta\in C^2_0(B(x,r_i(x)))$ such that $\eta\equiv 1$ in $B(x,r_i(x)/2)$, $\eta\equiv 0$ off $B(x,3r_i(x)/4)$, $\|\nabla \eta\|\le \frac{4}{r_i(x)}$ and 
$\|\triangle  \eta\|\le \frac{C}{r_i^2(x)}$. In fact, we fix  a non-negative regular function $\xi$ such that $\xi(t)=1$ if $t\le \frac 12$ and $\xi(t)=0$ if $t\ge \frac 34$. Set $\eta(\cdot)=\xi(\frac{d_{ \bar g_i}(\cdot,x)}{r_i(x)})$. Let $z$ be a maximal point of $\bar v_ir_i\eta$. At that point we have $\nabla_{ \bar g_i} \bar v_ir_i\eta(z)=0$ and $0\ge \triangle_{ \bar g_i} (\bar v_ir_i\eta)(z)$. We recall the equation
$$
-\triangle_{\bar  g_i} \bar v_i+\frac{R_{\bar  g_i}}{6}\bar v_i=-2\bar v_i^3
$$
which implies at the point $z$
$$
\begin{array}{lll}
0&\ds\ge \triangle_{\bar  g_i} (\bar v_ir_i\eta)(z)=(2\bar v_i^3+\frac{R_{\bar  g_i}}{6}\bar v_i)r_i\eta-2\frac{|\nabla_{\bar  g_i} (r_i\eta)|^2 \bar v_i}{r_i\eta}+\bar v_i\triangle_{\bar  g_i} (r_i\eta)\\
&\ds\ge 2\bar v_i^3r_i\eta-2\frac{|\nabla_{\bar  g_i} (r_i\eta)|^2\bar v_i}{r_i\eta}+\bar v_i\triangle_{\bar  g_i} (r_i\eta)
\end{array}
$$
Here we use the fact $R_{\bar  g_j}\ge 0$.  On the other hand, $|\nabla_{\bar  g_i} (r_i\eta)| (z)\le 5$ and $|\triangle_{\bar  g_i} (r_i\eta)|(z)\le \frac{C}{r_i(z)}$. Hence at the point $z$, we infer
$$
\begin{array}{lll}
0&\ds\ge 2(\bar v_ir_i\eta)^3(z)-(50(\bar v_ir_i\eta)(z)(r_i\eta)(z)+C(\bar v_ir_i\eta)(z)\eta(z))\\
&\ds\ge 2(\bar v_ir_i\eta)^3(z)-(50(\bar v_ir_i\eta)(z)+C(\bar v_ir_i\eta)(z))
\end{array}
$$
Hence, $\bar v_ir_i\eta(z)\le C$ and as a consequence $\bar v_i(x)r_i(x)\le C$. \\
As a consequence, we claim there exists some $C>0$ such that $\bar v^{-1}_i\ge C$ provided $d_{\bar g_i}(x,\partial X)\ge 1$.\\
Recall $-R_{\bar g_j}\bar v^{-1}_i=2 \triangle_{\bar g_i} (\bar v^{-1}_i)\le 0$ since $R_{\bar g_j}\ge 0$. It follows from the maximal principle that $\bar v^{-1}_i$ atteint son minimum on the boundary in the set $\{x, r_i(x)\ge 1\}$.  Hence, the desired claim yields.\\

{\bf Step 2.} The limit metric $g_{\infty}$ is conformal to A.H. Einstein manifold. \\

 Assume $\bar g_i$ converge to some  complete non compact manifold $(X_\infty,g_\infty)$ with a boundary  in the pointed Gromov Hausdorff sense (cf. \cite{CLN} Theorem 6.35). Indeed, by the Sobolev inequality (the assumption (4) in Theorem \ref{maintheorem}), there exists some constant $c>0$ and $r_0>0$ such that for any point $p$, $vol(B(p,r))\ge c r^4$  for all $r\in (0,r_0)$. By the result due to Cheeger-Gromov-Taylor \cite{CGT} (see also \cite{CLN} Theorem 5.42), we have the lower bound estimate for the injectivity radius. Thus, the desired convergence follows from Gromov Hausdorff convergence.  For this limiting metric,  due
 to our assumption that $\hat g_i $ is a compact family, the boundary is $\R^3$ endowed with Euclidean metric. We denote by $x_1$ the geodesic defining function for the limiting AH metric $g_\infty^+$ with respect to the boundary metric $\R^3$. Indeed, $g_\infty^+$ is complete since $(X_\infty,g_\infty)$ is complete and the defining function $f=\lim_i\bar v_i^{-1}$ satisfying $\|\nabla_{g_\infty} f\|\le 1$ and $f(x)\ge C_1\min ( d(x,\p X), 1)$.  We now claim the metric $g^+_\infty=f^{-2}g_\infty$ is Einstein with negative scalar curvature. For this purpose, we take the limit in equations (\ref{relation5}) and (\ref{relation4bis}) and get that
$$
\begin{array}{lll}
f\triangle_{ g_\infty}f=-3+3|\nabla_{ g_\infty} f|^2\\
Ric_{g_\infty}=-2f^{-1}\nabla^2_{g_\infty} (f).
\end{array}
$$
Again from conformal change, it follows 
$$
Ric_{g^+_\infty}=Ric_{ g_\infty}+ 2f^{-1}\nabla^2_{g_\infty} (f)+ (f^{-1} \triangle_{g_\infty} (f)-3f^{-2}|\nabla_{g_\infty}( f)|^2)g_\infty
$$
Together with the two previous relations, we infer 
$$
Ric_{g^+_\infty}=-3g^+_\infty.
$$
Therefore, the desired claim follows. 

Moreover, we note that (\ref{relation4bis}), (\ref{relation4bis1}) and (\ref{eq1.1}) are equivalent between them.  We have seen (\ref{relation4bis}) is true for the limit metric $g_\infty$ by replacing $v^{-1}$ by $f$. Hence  (\ref{eq1.1})  is also true for the limit metric $g_\infty^+$, that is 
$$
-\triangle_{g_\infty^+} \log f=3
$$
Hence, $ g_\infty$ is Fefferman-Graham's compactification of $ g_\infty^+$.
On the other hand, recall the $S$-tensor is a pointwise conformal invariant. From our assumption we have
$$
\lim_{r\to 0}\sup_i\sup_x\oint_{B(x,r)} |S_i|=0, 
$$
we infer for the limiting metric and for any $r>0$
$$
\oint_{B(x,r)} |S_\infty|=0
$$

As a consequence, $S$-tensor for the limiting metric $S_\infty=0$ on the boundary.  Near the boundary, the limiting metric $g_\infty^+$ on $X_\infty$ is a locally hyperbolic space. This is a result due to Biquard \cite{Biquard}  and  Biquard-Herzlich \cite{BiquardHerzlich} since the boundary metric is flat and the $S$-tensor vanishes.\\

{\bf Step 3.}  $g_\infty^+$ is a locally hyperbolic space.\\
To see this, we work for the Einstein metric $g^+_\infty$. By (\ref{Bachflat2}),  the Weyl tensor satisfies
$$
\triangle_{g^+_\infty} W=Rm*W
$$
since the cotton tensor $C=0$. Therefore,
$$
|\triangle_{g^+_\infty} W|\le C(x)|W|
$$
where $C= C(x) $ is some regular function. Set 
$$
A:=\{x\in X_\infty|\;\exists r>0 \mbox{ such that }W\equiv 0  \mbox{ on } B(x,r)\}.
$$
It is clear that $A$ is an open set. From the step 1, we know the Weyl tensor $W$ vanishes in a neighborhood of the boundary so that  $A$ is not empty. As $C(x)$ is a regular function, we can always bound it locally by some positive constant from above. Applying the well known unique continuation principle for this strong elliptic system (see \cite{Mazzeo}), $A$ is also closed. As a consequence, $A= X_\infty$ since $ X_\infty$ is connected, that is,  $W$ is identically equal to $0$. Hence, $g^+_\infty$ is a hyperbolic space form.\\

Thus, we have finished the proof of Lemma \ref{lemma4.3}.
\end{proof}

We will now prove the analogue of Lemma \ref{lemma4.3} under the assumptions of Theorem \ref{maintheorem1} instead of Theorem \ref{maintheorem}, where the assumptions on $S$-tensor is replaced by the assumptions of  $T$-curvature.\\ To do so, we will first explore the relation between the $S$ tensor and
the $T$ curvature.

The non-local terms $S$-tensor and $T$-curvature come from different considerations: the former one from
boundary metric of Bach flat equations; the latter one from boundary term of the Gauss-Bonnet integrand. However, we will show they are linked in the sense that for the limiting metric $g_{\infty}$, $T_\infty\equiv 0$ is equivalent to $S_\infty\equiv 0$. For  conformally compact Einstein manifolds, 
when the boundary metric is Ricci flat, the fact that $T_\infty\equiv 0$ implies $S_\infty\equiv 0$, was proved  earlier in \cite{casechang, casechangbis}. We will apply the same strategy of proof there, but as our limiting metric 
$g_{\infty}$ is now defined on the non-compact manifold $X_{\infty}$,  we need to do some more careful analysis.

Recall under the assumptions of  Theorem \ref{maintheorem1}, if we suppose  $K_i\to\infty$, we have asserted  in section \ref{sb4.1} that $(X,\bar g_i)$ converges to  $(X_\infty,g_\infty)$ in $C^{k-1,\alpha}$  norm in Gromov-Hausdorff sense  for some $k\ge 5$ and for all $\alpha\in (0,1)$, which is a manifold with totally geodesic boundary (we have convergence in $C^{k+2,\alpha}$  norm in Gromov-Hausdorff sense by a bootstrapping argument in section \ref{sb4.4}). We now state our result. \\

\begin{lemm}
\label{T=S}
For the limiting metric $g_{\infty}$ is our setting, $T_\infty\equiv 0$ iff $S_\infty\equiv 0$.
\end{lemm}

\begin{proof}

We will first show that $T_{\infty} \equiv 0$ implies $S_\infty\equiv 0$. \\
Recall the blow up boundary metric $h=\hat {g}_\infty$ on the boundary is flat. Let $r$ be the special defining function related to $h$. Therefore, in a collar neighborhood $V=[0,\varepsilon)\times B$ where $B\subset M$ is some compact subset of conformal infinity, we  have the expansion (see \cite{G00}) for $g^+_\infty=r^{-2}(dr^2+ g_r)$
$$
g_r=h+\kappa r^3+o(r^3)
$$
where $g_r$ is a family of metrics on $M$ and $\kappa=-\frac23S$ by Remark \ref{remark2.2}.  We also recall $\mbox{tr}_h \kappa=0$. If there is no confusion, we drop the index $r$ for $g_r$. Moreover, we have (see \cite{GH})
$$
\begin{array}{ll}
0=rg^{ij}\p^2_rg_{ij}-g^{ij}\p_rg_{ij} -\vs\frac{r}{2}g^{li}g^{mj}\p_rg_{lm}\p_rg_{ij}\\
0=r\p^2_rg_{ij}-2\p_rg_{ij}-rg^{lm}\p_rg_{il}\p_rg_{jm}+\frac{r}{2}g^{lm}\p_rg_{lm}\p_rg_{ij}-(g^{lm}\p_rg_{lm})g_{ij}-2rRic(g_r)_{ij}
\end{array}
$$
As shown in \cite{casechangbis}, by differentiating up to 5 times and evaluating at $r=0$ for these relations, we infer
$$
g_r=h+\kappa r^3-\frac{\triangle_h \kappa}{10}r^5+\frac{1}{2}(k_1 +\frac{5}{24}|\kappa|^2h)r^6+O(r^7)
$$
where $k_1$ is the traceless part of the composition $(\kappa^2)_{ij}=\kappa_{im}\kappa^m_j$. The direct calculations lead to 
$$
\begin{array}{lll}
\ds \sqrt{\det g^+_\infty}= r^{-4}\sqrt{\det h}(1-\frac3{32} |\kappa|^2r^6+o(r^6))\\
\ds (g_r)^{ij}= h^{ij}-r^3\kappa^{ij}+O(r^5)\\
\end{array}
$$
so that 
\beq
\label{Laplaceinfty}
\triangle_{g^+_\infty}=(r\p_r)^2-3r\p_r-\frac{9}{16}r^6|\kappa|^2r\p_r+r^2\triangle_h-r^5\div(\kappa(\nabla_h\cdot, \cdot))+O(r^7)
\eeq
Here $O(r^7)$ means the bounded operators like $r\p_r$ or in $M$ with the coefficients in $O(r^7)$. Using the relation (\ref{relation5}), the set of functions $\bar v_i^{-1}$ is  compact in $C^{k+2, \alpha}$ space. Recall the set of metrics $\bar g_i$ is relatively compact in $C^{k+2, \alpha}$ space. Using the construction in \cite{Lee}, the special defining functions $r_i$ related to $g_i^+$ forment a relatively compact set locally in $C^{k+2, \alpha}(V)$ space provided $k\ge 1$. Here $\varepsilon$ is independent of $i$. We know (see\cite{FG})
$$
\bar v_i^{-1}=r_i\exp(A+Br_i^3)
$$
where $A=-\frac{R_{h_i}}{8}r_i^2+O(r_i^4)$ and $B=c_3T(\bar g_i)+ B_2r_i^2 +O(r_i^4)$ with $c_3$ an universal constant. We state both $\bar v_i^{-1}$ and $r_i$ are relatively compact in $C^{k-1, \alpha}(V)$. By taking the limit as $i\to\infty$, we infer
$$
f=r(1+O(r^4))
$$
since $R_{h_i}\to 0$ and under our assumption  $T(\bar g_i)\to T_\infty\equiv 0$. Recall $f$ solves (\ref{eq1.1}), that is,
\beq
\label{eq1.1bis}
-\triangle_{g^+_\infty}\log f=3
\eeq
We have the following expansion for $\log f$ 
$$
\log f=\log r+ f_4r^4+f_5 r^5+ f_6 r^6+  f_7 r^7
$$
where $f_4,f_5,f_6\in C^2(M\cap V)$ 	and $f_7\in C^2(V)$. Together with (\ref{Laplaceinfty}) and (\ref{eq1.1bis}), we deduce
$$
\triangle_{g^+_\infty}\log f=-3-\frac{9}{16}r^6|\kappa|^2+4f_4r^4+r^6\triangle_h f_4+10f_5r^5+18f_6r^6+O(r^7)=-3
$$
As a consequence, we get
$$
f_4=f_5=0, f_6=\frac 1{32}|\kappa|^2
$$
so that 
$$
f=r+ \frac 1{32}|\kappa|^2 r^6+O(r^7)
$$
Now we can calculate the scalar curvature for the metric $g_\infty=f^2 g^+_\infty$
$$
R_{g_\infty}=6(f^{-2}-|\nabla_{g^+_\infty} f^{-1}|^2)=6(f^{-2}-(r\p_rf^{-1})^2-r^2|\nabla_{g_r}f|^2)=-\frac{9}{4}|\kappa|^2r^4+o(r^4)
$$
that is, the scalar curvature is negative in some neighborhood of the boundary provided $\kappa\neq 0$. On the other hand 
$$
R_{g_\infty}=\lim R_{\bar g_i}\ge 0
$$
Thus, $\kappa=0$ in $V\cap M$. Now arguing as in the proof of the previous lemma, we prove $E$ is hyperbolic space. As we need to do the expansion of $g_r$ up to order 7,  $\bar g_i$ should be in $C^7$. We know $\bar g_i$ is compact in $C^{k+2,\alpha}_{loc}$ norm as in section \ref{sb4.4} below. Hence, we require $k+2\ge 7$, that is, $k\ge 5$. Thus  $T_\infty\equiv 0$ implies $S_\infty\equiv 0$.  Conversely, by lemma \ref{lemma4.3}, $S_\infty\equiv 0$ implies $g_\infty$ is conformal to hyperbolic space form. Using the proof in Proposition \ref{prop4.4} below,  $g_\infty$ is flat. Hence $T_\infty\equiv 0$.
Thus we have finished the proof of Lemma \ref{T=S}.
\end{proof}

\begin{lemm}
\label{lemma4.3bis}
Under the assumptions of  Theorem \ref{maintheorem1} and assuming that $\bar g_i$ has the type I boundary blow up,  $g_\infty$ is conformal to hyperbolic space form.
\end{lemm}

\begin{proof} 
We will show under the assumption on $T$ curvature on Theorem \ref{maintheorem1},  $T_{\infty} $ vanishes, hence by the Lemma \ref{T=S}, the $S_{\infty} $-tensor on the boundary also vanishes and thus the same proof as lemma \ref{lemma4.3} can be applied.\\

To see $T_{\infty} $ vanishes, we go through the similar proof as lemma \ref{lemma4.3}. Denote $0$ the marked point. Fixing $r>0$, we have
$$
\oint_{B_{\bar g_i}(0,r)} T(\bar g_i)=\oint_{B_{ g_i}(0,K_i^{-1}r)} T(g_i)
$$
By our assumption on $T$,  we infer
$$
\oint_{B_{ g_\infty}(0,r)} T_\infty=\lim_i \oint_{B_{\bar g_i}(0,r)}  T(\bar g_i)\ge 0
$$
since $K_i^{-1}r\to 0$.  On the other hand, the scalar curvature for the  limiting metric $ g_\infty$ is non-negative and vanishes on the boundary which implies that $T$-curvature  is non-positive, that is, $T_\infty=\frac{1}{12}\frac{\p R_\infty}{\p n}\le 0$. Thus, the integral of $T_{\infty} $-curvature over the geodesic ball $B(0,r)$ vanishes for any $r>0$. As a consequence, the   $T_\infty$ on the boundary for the limiting metric $ g_\infty$  is equal to zero. 
\end{proof}

\begin{rema}
As in Theorem \ref{maintheorem} and  Corollaries \ref{maincorollary0}-\ref{maincorollary01}, we have $S$-tensor $S_i$ is uniformly bounded in $L^1$. Thus, we can normalize in $C^0$ norm for the curvature and 
going through the blow up analysis. That is, we can assume $k\ge 1$ instead of $k\ge 2$ to reach the same results in  Theorem \ref{maintheorem} and Corollaries \ref{maincorollary0}-\ref{maincorollary01}. Also in Theorem \ref{maintheorem1} and   Corollaries \ref{maincorollary1.0}-\ref{maincorollary1.1}, if we assume in addition that $S$-tensor $S_i$ is uniformly bounded in $L^1$, we can assume $k\ge 4$ in the place of $k\ge 5$ to conclude the same results there.
\end{rema}

We now conclude this session by showing the boundary blow up can not occur by establishing a Liouville type theorem that, under the assumptions of theorems  \ref{maintheorem} or \ref{maintheorem1}, $g_{\infty}$ is the flat metric.

\begin{prop} 
\label{prop4.4}
Under the same assumptions as in Theorem \ref{maintheorem} or Theorem \ref{maintheorem1}, there is no blow up on the boundary.
\end{prop}

\begin{proof}
From Lemma \ref{lemma4.3}, we know the limiting manifold $(X_{\infty}, g_{\infty}^+)$ is locally hyperbolic space.  We now work on the limit metric $g_{\infty}$. For simplicity, we will omit the index $\infty$. We denote $\tilde g^+$ standard hyperbolic space with the upper half space model. As $\tilde g^+= g^+$ in a neighborhood of the boundary $\{x_1=0\}$, we can extend this local isometry to a covering map  $\pi: \tilde g^+\to g^+$. We write 
 $$
 g_1= e^{2w_1}\tilde g^+ \mbox{ and } g_2= e^{2\tilde w_2}g^+
 $$
where $g_1$ is the standard euclidean metric and $g_2= g_{\infty} $ the limit FG metric. With the help of the covering map $\pi$, we have $\pi^*g_2= e^{2w_2}g^+$ where $w_2=\tilde w_2\circ \pi$. 
We have for $i=1,2$
$$
-\triangle_{\tilde g^+} w_i=3
$$
and 
$$
w_1=\log x_1
$$
Remind $x_1$ is the geodesic defining function w.r.t. the flat boundary metric. We write $\pi^*g_2=e^{2w_2}\tilde g^+=e^{2w_2-2w_1}g_1:= e^{2u}g_1$ where $u=w_2-w_1$. The semi-compactified metric $g_2$ (or $\pi^*g_2$) has flat $Q_4$  and  the boundary metric of $g_2$ is the euclidean 3-space and totally geometric. We now claim $u$ satisfies
the following conditions:
\beq
\label {eq4.5bis}
\left\{
\begin{array}{lll}
\triangle^2 u =0 &\mbox{ in } \R^4_+, \\
-\triangle u-|\nabla u|^2 \ge 0&\mbox{ in } \R^4_+,\\
u=\nabla u =\triangle u=0&\mbox{ on } \p\R^4_+.
\end{array}
\right.
\eeq
The first equation comes from the flat $Q_4$ curvature and second one from the non-negative scalar curvature. To see the third assertion in (\ref{eq4.5bis}), we first observe
that as $g_2$ on the boundary is euclidean, $u \equiv 0 $ on the boundary.  On the other hand, we know by properties of $g_1$ and  $g_2$ that 
we have $\triangle_{\tilde g^+} u=0$, which in the coordinate of the hyperbolic metric $\tilde g^+$  is the same as  $x^2_1\sum_{i=1}^4\p_{x_i}(\frac{1}{x^2_1}\p_{x_i} u)=0$,
thus we have  
\beq
\label{eq4.6}
x_1\triangle u=2\p_1 u.
\eeq
From (\ref{eq4.6}) we conclude that on the boundary, $\p_1 u=0$, this plus the fact $ u \equiv 0$ on the boundary, we have $\nabla u=0$. On the other hand, it follows from Lemma 2.3 the restriction of the scalar curvature vanishes on the boundary so that $-\triangle u-|\nabla u|^2 = 0$. This yields $\triangle u=0$ on the boundary. We have thus established
(\ref{eq4.5bis}). \\  
As $-\triangle u$ is harmonic and non-negative, 
it follows from a result due to H.P.Boas and R.P. Boas \cite{BB}, 
$$
-\triangle u= ax_1
$$
with some $a>0$. Hence, by (\ref{eq4.6}),  $\p_1 u=-ax_1^2/2$ so that 
$|\nabla u|^2\ge |\p_1 u|^2=a^2x_1^4/4$. Hence $-\triangle u-|\nabla u|^2= ax_1-a^2x_1^4/4\ge 0$ if and only if $a=0$. 
From  (\ref{eq4.5bis}), $|\nabla u|^2\le -\triangle u=0$, $u$ is a constant. Together with the boundary condition, $u\equiv0$. That is, $g_2$ is a flat metric. However, by the blow up arguments, $g_2=g_\infty$ is not flat since either its curvature or the derivative of its curvature has been normalized to be one at a point in X,  which yields the desired contradiction.  Thus there is no boundary blow up and we have finished the proof of Proposition \ref{prop4.4}.
\end{proof}

\subsection{blow-up analysis in the interior}
\label{sb4.3}

Now we will do the blow-up analysis in the interior and want to prevent this to happen with the help of Proposition \ref{prop4.4}.\\

Recall for the interior blow up, we have
\begin{enumerate}
\item $\ds \max\{|Rm_{\bar g_i}|(0), |\nabla Rm_{\bar g_i}|(0)\}=1$;
\item $\ds \lim_i d_i(0,\p X)=\infty$;\\
We now claim we also have 
\item For any $C>0$,  $\ds\sup_{d_i(x,\p X)\le C}|Rm_{\bar g_i}(x)|+ |\nabla Rm_{\bar g_i}(x)|\to 0$.
\end{enumerate}
We now prove  the claim by a contradiction argument. Assume otherwise, then for some fixed $C>0$, modulo a subsequence, we have 
$$
\sup_{d_i(x,\p X)\le C}|Rm_{\bar g_i}(x)|+ |\nabla Rm_{\bar g_i}(x)|\to \alpha>0
$$
Let $p_i\in \{d_i(x,\p X)\le C\}$ such that 
$$|Rm_{\bar g_i}(p_i)|+ |\nabla Rm_{\bar g_i}(p_i)|=\sup_{d_i(x,\p X)\le C}|Rm_{\bar g_i}(x)|+ |\nabla Rm_{\bar g_i}(x)|
$$
We mark the point $p_i$ as the origin. By pointed Hausdorff-Gromov's convergence, $(X,\bar g_i, p_i)$ converges to a non-flat  manifold with the totally geodesic  boundary whose doubling manifold is complete.  This is a boundary blow up, which contradicts Proposition \ref{prop4.4}. Therefore, the desired property (3) follows.\\

\begin{lemm} 
\label{lemma4.7}
Under the assumptions as in either Theorem \ref{maintheorem} or Theorem \ref{maintheorem1}, we assume  $K_i\to\infty$ and  $\bar g_i$  has the type II interior blow up.   Then  the limiting metric $g_\infty$ is  Ricci free and non-flat.
\end{lemm}

\begin{proof}
{\it  We divide the proof into 3 steps.}\\

{\it Step 1.} Let $i_{\p}(\bar g_i)$ be the boundary injectivity radius. We claim  $i_{\p}(\bar g_i)\to \infty$. \\

We fix a large $B>1$ and consider the scaling metric $\tilde g_i= B^{-2} \bar g_i$. Thus,  $\ds\sup_{d_{\tilde g_i}(x,\p X)\le 1}|Rm_{\tilde g_i}(x)|\to 0$. We argue as before: we state  the Sobolev inequality (\ref{Sobolev}) implies $B(x,r)\ge cr^4$ for all $x\in \{d_{\tilde g_i}(x,\p X)\le 1\}$ and $r<1$. As the boundary $M$ is totally geodesic, we can use the doubling argument along the boundary $M$  to get a compact manifold without boundary $Y:=X\bigcup_M X_1$ where $X_1$ is a copy of $X$ with opposite orientation. On the closed manifold $Y$, thanks to a result of Cheeger-Gromov-Taylor \cite{CGT}, we have the uniform lower bound for any closed simple geodesic on $Y$, which yields $i_{\p}(\tilde g_i)\ge \alpha>0$ for some $\alpha >0$ independent of $i$. Hence $i_{\p}(\bar g_i)=Bi_{\p}(\tilde  g_i)\ge B\alpha$. This gives the desired claim.\\

{\it Step 2.}  We claim that there exists some positive constant $C>0$ such that for any $A>1$ there exists some positive entire number $N\in\mathbb{N}$ such that for all $j>N$ for all $x\in X$ with $r_i(x) = d_{\bar g_i}(x,\partial X) < A $,  we have
$$
\bar v_i^{-1}(x) \ge C{r_i(x)}
$$

The proof of the claim is as same as that one of the step 1 in Lemma \ref{lemma4.3}.\\

{\it Step 3.}   The limiting metric $g_\infty$ is of Ricci flat.\\

To see so, we first claim  $\lim \bar v_i^{-1}(0)=\infty$. For this purpose, we note
first by equation (\ref{relation5}), $ 2 \triangle_{\bar g_i} (v^{-1}_i) =-R_{\bar g_j}v^{-1}_i \le 0 $, thus by maximal principle,  for each fixed $A>1$ and for all sufficiently large $j$ one has $\bar v_i^{-1}(0)\ge \min_{d_{\bar g_i}(x,\p X)=A}\bar v_i^{-1}\ge CA$. 
 Thus, $\liminf \bar v_i^{-1}(0)\ge CA$.  As $A$ is arbitrary, we infer $\liminf\bar v_i^{-1}(0)=\infty$ which yields the desired result.  Thus for any compact set $B\subset X_\infty$, $ \bar v_i^{-1}$ is uniformly unbounded as $\|\nabla_{\bar g_i}( \bar v_i^{-1})\|\le 1$.  Apply  (\ref{relation3}), we have conclude the scalar curvature 
 of  $g_\infty$ is zero, since $g_{\infty} $ is also $Q$ flat (\ref{Qflat}), we conculde $g_\infty$ is also Ricci flat.  \\
Thus we have established Lemma \ref{lemma4.7}.

\end{proof}

We also now apply a recent result due to Cheeger-Naber \cite{CheegerNaber}, to show the Weyl tensor 
of the limiting metric $g_{\infty} $ also vanishes. First, we recall the result of Cheeger-Naber (Theorem 1.13 in \cite{CheegerNaber}). 

\begin{lemm} 
\label{lemma4.6bis1}
There exists $C=C(v)$ such that if $4$-dimensional Riemannian manifolds $X^4$ satisfies $|Ric_{X^4}|\le  3$ and $Vol(B_1(p))>v>0$ for some point $p\in X$, then 
$$
\int_{B_1(p)} |Rm|^2\le C(v)
$$
\end{lemm}

As a direct consequence, we have the following result.

\begin{lemm} 
\label{lemma4.6bis2}
Let  $E$ be a $4$-dimensional complete non-compact Ricci flat manifold. Assume 
$$
V(B(o,t))\ge Ct^4 ,\;\forall t>0
$$
for some positive constant $C>0$.  
Then  we have
$$
\int_E |W|^2<\infty.
$$
\end{lemm}

\begin{proof}
Let $g$ be the metric on $E$. We consider the scaling metric $g_i=i^{-2} g$ for all $i\in \mathbb N$. It is clear that $g_i$ is still a complete Ricci flat metric and $Vol(B_1(o))\ge C$ for each $g_i$. Using Cheeger-Naber's result, we deduce for the metric $g$
$$
\int_{B_1(o)} |W|^2_{g_i}= \int_{B_i(o)} |W|^2_g\le C
$$
where $C$ is a constant independent of $i$. Letting $i\to\infty$, the desired estimate follows. 
\end{proof}
We now recall a result due to Shen-Sormani \cite{Shen-Sormani}.

\begin{lemm} 
\label{lemma4.5}
Let  $E$ be a $4$-dimensional complete non-compact Ricci flat manifold. Assume $E$ is not flat and oriented. Then for any abelian group $G$, $H_3(E,G)$ is trivial. Moreover, for any abelian group $G$
$$
Tor(H_2(E,\mathbb{Z}),G)=0.
$$
In particular, $H_2(E,\mathbb{Z})$ has no elements of finite order.
\end{lemm}

We will also use the following result on the topology of $X$.
\begin{lemm} 
\label{lemma4.5bis}
If we glue $X$ and the unit 4-ball along the boundary ${\mathbb S}^3$  and denote $\tilde X:=X\bigcup_{{\mathbb S}^3} B^4$. Then $\tilde X$ is a homology 4-sphere. 
\end{lemm}

\begin{proof}
We note $X$ is a non-compact manifold. 
By Proposition 3.29 in \cite{Hatcher}, we infer that $H_4( X, \mathbb{Z})=0$. We use Mayer-Vietoris exact sequences for $B^4$ and $\overline{X}$ 
$$
\cdots\to H_i(\mathbb{S}^3, \mathbb{Z})\to H_i(B^4, \mathbb{Z})\oplus H_i(\overline{X}, \mathbb{Z})\to H_i(\tilde X, \mathbb{Z})\to \cdots
$$
We know $H_1(\mathbb{S}^3, \mathbb{Z})=H_2(\mathbb{S}^3, \mathbb{Z})=0$,$H_0(\mathbb{S}^3, \mathbb{Z})=H_3(\mathbb{S}^3, \mathbb{Z})=\mathbb{Z}$  and $H_i(B^4, \mathbb{Z})=0$ for $1\le i\le 4$  and $H_i( X, \mathbb{Z})=0$ for $i=1,2,4$ so that $H_i(\tilde X, \mathbb{Z})=0$ for $i=1,2$ and $H_4(\tilde X, \mathbb{Z})=\mathbb{Z}$. On the other hand, $\tilde X$ is a connected 4-manifold. Thus $H_0(\tilde X, \mathbb{Z})=\mathbb{Z}$. As $X$ is oriented, $\tilde X$ is closed and oriented. Using the universal coefficient theorem for cohomology (Theorem 3.2 in \cite{Hatcher}) , we infer
$$
H^1(\tilde X, \mathbb{Z})=0. 
$$
Applying the Poincar\'e duality theorem (Theorem 3.30  in \cite{Hatcher}), we get
$$
H_3(\tilde X, \mathbb{Z})=0. 
$$
Hence, it follows from  Theorem 3.26  in \cite{Hatcher}  that $H_0(\tilde X, \mathbb{Z})=H_4(\tilde X, \mathbb{Z})=\mathbb{Z}$ and $H_i(\tilde X, \mathbb{Z})=0$ for $i\neq 0,4$, that is,
$\tilde X$ is a homology 4-sphere.  Thus we have finished the proof of Lemma \ref{lemma4.5bis}.
\end{proof}

We recall a result due to Crisp-Hillman (\cite{CH} Theorem 2.2).

\begin{lemm} 
\label{lemma4.5bisbis}
Let  $X$ be a closed oriented homology 4-sphere and  $\mathbb{S}^3/\Gamma$ be a spherical 3-manifold with $\Gamma$ some finite group of $SO(4)$. Assume  $\mathbb{S}^3/\Gamma$ is embedded in $X$. Then $\Gamma=\{1\}$ or $\Gamma=Q_8$(quaternion
group) or $ \Gamma$ the perfect group (that is, $\mathbb{S}^3/\Gamma$ is a homology 3-sphere).
\end{lemm}

Now we will rule out the interior blow-up.\\

\begin{prop} 
\label{prop4.8}
Under the same assumptions as in  Theorem \ref{maintheorem},  there is no blow up in the interior.
\end{prop}

\begin{proof}
We argue by contradiction. We assume there is blow up in the interior.  Let $\bar g_i 
$ be renormalized metrics as before and we denote $X_i=(X,\bar g_i)$ for simplicity. We now choose $E=X_\infty$ the blow up metric in the interior. Hence, it is Ricci flat. Moreover, 

on each $X_i$, and for any Lipschitz function with compact support $f$
$$
 \|f\|_{L^4}^2\le C  (\|\nabla f\|_{L^2}^2+ \frac{1}{6}\int R f^2)
$$
As $C$ is uniformly bounded from below, we can pass in the limit, that is,
$$
\|f\|_{L^4(X_\infty)}^2\le C(X_\infty)  \|\nabla f\|_{L^2(X_\infty)}^2
$$
since  the limiting metric is Ricci flat. This means there is no collapse on $E$. 

We divide the proof in two cases.\\

{\bf Case 1:} $H_2( E,\mathbb{R})\neq 0$.\\
We use mayer-Vietoris exact sequences for $E$ and $\overline{X\setminus E}$ 
$$
\cdots\to H_2(\mathbb{S}^3/\Gamma, \mathbb{R})\to H_2(E, \mathbb{R})\oplus H_2(\overline{X\setminus E}, \mathbb{R})\to H_2(X, \mathbb{R})\to \cdots
$$
We know $H_2(\mathbb{S}^3/\Gamma, \mathbb{R})=0=H_2(X, \mathbb{R})$. This contradicts the fact that the above sequence  is exact. \\

{\bf Case 2:} $H_2( E,\mathbb{R})= 0$.\\
From the result in \cite{BKN}, $E$ is then ALE of order $3$, that is, at the infinity, $\p_\infty E=\mathbb{S}^3/\Gamma$ is spherical 3-manifold where $\Gamma$ is some finite group of $SO(4)$. By Lemma \ref{lemma4.5}, the third betti number $b_3=0$. Let $\tilde E$ be the universal cover of $E$. Thus $\tilde E$ is also Ricci flat. By Bishop-Gromov volume comparison theorem,  $V(B(o,r))\le C_1r^4$.  As  a consequence, the fundamental group of $E$ is finite. Since $H_1(E,\mathbb{Z})$ is the abelization of the fundamental group, $H_1(E,\mathbb{Z})$ is finite and the first betti number $b_1=0$.  \\

We note  $H_2(E,\mathbb{R})= 0$.  From Lemma \ref{lemma4.5}, $H_3(E,\mathbb{Z})=0$. On the other hand,  it is clear $H_4(E,\mathbb{Z})=0$ since $E$ is an open manifold (see \cite{Hatcher} Proposition 3.29). Thus,  the second, third and forth betti numbers vanish  $b_2=b_3=b_4=0$, and the Euler characteristic number  $\chi(E)=1$.

We know a spherical 3-manifold is Seifert fibred 3-manifold and $\tilde X$ is a homology 4-sphere. By the result due to Crisp-Hillman Lemma \ref{lemma4.5bisbis}, we know the boundary of $E$ at the infinity is $S^3/\Gamma $ with $\Gamma=\{1\}$ or $\Gamma=Q_8$ or $ \Gamma$ the perfect group.\\

From Gauss-Bonnet formula,
we have
$$
\frac{1}{8\pi^2}\int_E (|W_+|^2+|W_-|^2)=\chi(E)-\frac{1}{|\Gamma|}=1-\frac{1}{|\Gamma|}
$$
On the other hand, we have the signature of $E$ is trivial since $H_2(E, \mathbb{R})=0$. Using the signature formula
$$
0=-\tau(E)=\frac{1}{12\pi^2}\int_E (|W_+|^2-|W_-|^2)-\eta(\mathbb{S}^3/\Gamma)
$$
where $\eta(\mathbb{S}^3/\Gamma)$ is the eta invariant.\\
When $\Gamma=\{1\}$, it follows from the Gauss-Bonnet formula  that 
$$
W=0
$$
Hence $E$ is flat since $E$ is Ricci flat. This contradicts the non-flatness of  $E$. Or alternatively, Bishop-Gromov comparison theorem yields that a ALE Ricci flat manifold asymptotic to $\R^4$  is flat, which yields the desired result.\\
When $\Gamma=Q_8$, we know $\eta(\mathbb{S}^3/\Gamma)=\frac 34$. Hence, we have
$$
\frac{1}{8\pi^2}\int_E (|W_+|^2+|W_-|^2)=\frac 78
$$
and
$$
\frac{1}{8\pi^2}\int_E (|W_+|^2-|W_-|^2)=\frac 98
$$
The above two equalities leads to a contradiction.\\
When $\Gamma$ is the perfect group,  that is, the  binary icosohedral group of order 120,  then $\eta(\mathbb{S}^3/\Gamma)=-1+\frac{1}{|\Gamma|}+\frac{1079}{360}$ (\cite{GibbonsPope-Romer}). Similarly, from the above two formulas, we get
$$
\frac{1}{8\pi^2}\int_E (|W_+|^2+|W_-|^2)=\frac{119}{120}
$$
and
$$
\frac{1}{8\pi^2}\int_E (|W_+|^2-|W_-|^2)=\frac{361}{120}
$$
This is a contradiction.\\
Thus all cases can not happen, thus there is no interior blow up and we have finished the proof of Proposition \ref{prop4.8}.
\end{proof}
\begin{rema} The same method permits us to do blow up analysis on $4$-dimensional homological sphere manifolds without boundary. In particular, one could get the compactness result of Einstein metrics (or the other canonical metrics) on such manifolds under suitable assumptions.
\end{rema}

\subsection{Proof of  Theorem \ref{Boundgeometry}}
\label{sb4.4}
\begin{proof}[Proof of  Theorem \ref{Boundgeometry}]
Propositions \ref{prop4.4} and  \ref{prop4.8} yields the boundness of $\|Rm_{g_i}\|_{C^1}$. Together with Theorem \ref{epsilonregularity1}, we prove the $C^{k-2}$ norm for the curvature is uniformly bounded. To handle the $C^{k+1}$ norm of the curvature, we recall first the curvature tensor satisfies some elliptic PDE with the Dirichlet boundary conditions
\beq
\label{ellipticsystem}
\left\{
\begin{array}{llll}
\triangle R=R^2-3|Ric|^2&\mbox{ in }X\\
R=3\hat{R}&\mbox{ on }M\\
\triangle A-\frac16 \nabla^2 R=Rm*A&\mbox{ in }X\\
 A_{\alpha\beta}=\hat{A}_{\alpha\beta}, A_{\alpha n}=0, A_{n\alpha }=0, A_{n n}=\frac{\hat{R}}{4}&\mbox{ on }M\\
\triangle W=Rm*W+g*W*A&\mbox{ in }X\\ 
W=0 &\mbox{ on }M
\end{array}
\right.
\eeq
The desired result follows from the classical elliptic regularity theory and Cheeger-Gromov-Hausdorff convergence theory (Lemma \ref{highorder} in the Appendix). This concludes the proof of Theorem \ref{Boundgeometry}.
\end{proof}

\section{Proof of  the results in Section 1}

\begin{proof}[Proof of  Theorem \ref{maintheorem}]
We have already proved the family of metrics $ g_i$ has the bounded curvature in $C^{k+1}$ norm in Section 4.  The uniform Sobolev inequality holds for the family metric $g_i$ by the assumptions in Theorem \ref{maintheorem}, which implies, as in Section 4, for all $i$, for all $x\in \bar X$, we have $vol(B_{g_i}(x,1))\ge C>0$ for some constant $C>0$ independent of $i,x$, that is, there is non-collapse for the volume. Working on the doubling manifold, it follows from a result of Cheeger-Gromov-Taylor \cite{CGT}, we have the uniform lower bound for any closed simple geodesic on the doubling manifolds, which yields that both the interior injectivity radius and the boundary injectivity radius  uniformly lower bound on $X$. By Cheeger-Gromov-Hausdorff compactness theory,  to prove the compactness of metrics $g_i$, it suffices to prove their diameters are uniformly bounded from above. We will prove this fact by contradiction. We divide the proof in 4 steps.\\

{\it Step 1.} Without loss of generality, we suppose the boundary injectivity radius is bigger than $1$. There exists some $C>0$ such that $v^{-1}_i\ge C$ provided $d_{ g_i}(x,\partial X)\ge 1$ and $v_i(x)d_{ g_i}(x,\partial X)\le C$ provided $0\le d_{ g_i}(x,\partial X)\le 1$.Thus the limit metric is conformal to an asymptotic hyperbolic Einstein manifold. The claim can be proved in the same way as in the proof of Lemma \ref{lemma4.3}.\\

Assume the diameter of $g_i$ is unbounded, then $g_i$ converges to some  non-compact metric $g_\infty$ on manifold $X_\infty$ with totally geodesic boundary  in the Cheeger-Gromov-Hausdorff sense, whose doubling is complete. Note that by our assumption of Theorem \ref{maintheorem} that  the boundary metric $\{\hat{g}_i \}$ of this family $\{g_i\}$ is a compact family.  \\

{\it Step 2.}  There exists some constant $C>0$ independent of $i$ such that $\ds\int |Rm_{g_i}|^2\le C$.\\
From the relation (\ref{relation4}), we infer
$$
\int_X v_i^{-1}R_{g_i}=-2\int_X \triangle_{ g_i} (v_i^{-1})=-2\oint_{\p X}\frac{\p v_i^{-1}}{\p n}=2vol(\p X, \hat{g}_i )
$$
since $\frac{\p _iv_i^{-1}}{\p n}=-1$ on the boundary. Thus, with the help of step 1, we get for some given constant $C>0$
\beq
\label{bound1}
\int_{d_{g_i}(x,\p X)\ge 1} R_{g_i}\le C
\eeq
On the other hand, the boundary metric $(\p X, h_i)$ is a compact family and also $g_i$ has uniform bound for the curvature tensor. Hence
\beq
\label{bound2}
\int_{d_{g_i}(x,\p X)\le 1} R_{g_i}\le C
\eeq
Combining (\ref{bound1}) and (\ref{bound2}), we obtain
$$
\int_{X} R_{g_i}\le C.
$$
which implies
\beq
\label{bound3}
\int_{X} R_{g_i}^2\le C
\eeq
since the scalar curvature $R_{g_i}$ is uniformly nounded. 
Using the free $Q$-curvature condition (\ref{Qflat}), we deduce
$$
\int_{X} |E_{g_i}|^2=\frac{1}{12}\int_{X} R_{g_i}^2-\frac13\oint_{\p X} \frac{\p R_{g_i}}{\p n}
$$
As we have the compactness of boundary metric and the curvature tensor is bounded in $C^k$ norm, we get the uniform bound for $L^2$ norm with the Ricci tensor. Now, thanks of the Gauss-Bonnet-Chern formula, we have
$$
\chi(X)=\frac{1}{32\pi^2}\int _{X} \left(|W_{g_i}|^2+\frac16 R_{g_i}^2 -2|E_{g_i}|^2\right)
$$
since the boundary is totally geodesic. This yields the $L^2$ bound for the  Weyl tensor. Therefore,  the desired result yields.\\

{\it Step 3.}  We denote $f=\lim  v_i^{-1}$, assuming $X_\infty$ is non-compact, we claim 
\beq
\label{conformalfactor}
\lim_{x\to\infty}\frac{f (x)}{d_{g_\infty}(x,\p X_\infty)}=1
\eeq
To prove the claim, we fix a point $0\in \p X_\infty$ and use the distance function $r=d_{g_\infty}(x,0)$.  Note $\p X_\infty$ is a compact set so that $\ds \lim_{x\to \infty}\frac{r(x)}{d_{g_\infty}(x,\p X_\infty)}=1$. 
We remark on the limiting metric one has always the Sobolev inequality, that is, for any compactly supported Lipschitz function $U$
$$
(\int_{X_\infty} U^4)^{1/2}\le C\int_{X_\infty} (|\nabla_{ g_\infty} U|^2 +\frac16 R_\infty U^2)
$$ 
From step 1, $\ds \int R_\infty^2< \infty$. Thus, we can find $r_0$ such that 
$$
\int_{\{x,r(x)\ge r_0\}} R_\infty^2\le \varepsilon,
$$
where $\varepsilon$ is a small constant appeared in Theorem \ref{epsilonregularity2}.   Recall from Theorem \ref{epsilonregularity2} and from the non-negative scalar curvature and also the relations (\ref{relation1}) and (\ref{relation3})
\beq
\label{gradient}
\|\nabla_{ g_\infty} f\|\le 1
\eeq
\beq
\label{gradient1}
|\nabla_{ g_\infty}(f)|^2=1-o(1)
\eeq
\beq
\label{gradient2}
{(\nabla^2_{ g_\infty}( f))_i}^j=o(r^{-1})
\eeq
In fact, we have from (\ref{gradient}) that $f(x)\le r(x)$. On the other hand, it follows from Theorem \ref{epsilonregularity2} 
$$
Rm=o(r^{-2})
$$
Together with the relations (\ref{relation3}) and (\ref{relation5}), we obtain (\ref{gradient1}) and (\ref{gradient2}).\\

Now we consider vector field $-\nabla_{g_\infty} f (x)$. For any $\varepsilon>0$, there exists some $A>0$ such that for all $r(x)>A$
$$
1-\varepsilon \le |\nabla_{ g_\infty}( f)|\le 1
$$
Set $S_t:=\{x;r(x)=t\}$. Define $m(t):=\inf_{S_t} f$. We remark $m(t)\ge 0$ since $f\ge 0$. 
We consider along the flow $F_t$ of  $-\nabla_{ g_\infty} f (x)$. 
We have $F_r(S_{r+A})\subset B(0, 2r+A)\setminus B(0,A)$ since $|\nabla_{ g_\infty}( f)|\le 1$. On the other hand $$m(2r+A)\le \inf_{S_{r+A}} f(F_r(S_{r+A}))\le \inf_{S_{r+A}} f(S_{r+A})-r(1-\varepsilon),$$
which  implies
$$
m(r+A)-r(1-\varepsilon)\ge m(2r+A)\ge 0
$$
However, $f(x)\le r$ for all $x\in S_r$ since $|\nabla_{\bar g_\infty}( f)|\le 1$. Gathering the above facts, we get the desired claim.\\

{\it Step 4.} A contradiction.\\

We know the Sobolev inequality is still true for the limiting metric, that is, for any Lipschitz function $U$ compactly supported in $\{x,r(x)\ge r_0\}$, we have 
$$
(\int_{X_\infty} U^4)^{1/2}\le C\int_{X_\infty} |\nabla_{ g_\infty} U|^2
$$
For large $s$, we fix a point $x$ with $r(x)=r_0+s$ and consider the function $U(y)=s-d(y,x)$ if $d(y,x)\le s$, otherwise $U(y)=0$. From the above inequality, we get
$$
\frac s2 vol(B(x,\frac s2))^{\frac14}\le vol(B(x,s))^{\frac12}
$$
so that 
$$
vol(B(0,2(r_0+s)))\ge vol(B(x,\frac s2))\ge C(2(r_0+s)))^4
$$
We denote $s_1=2(r_0+s)$. By Proposition 3.4 in  \cite{Chavel}, Proposition 4.5 in \cite{TV05} and the Courant-Lebesgue Lemma, there exists some $s_2\in (s_1/2,s_1)$ such that 
$$
vol(\p B(0, s_2))\ge cs_1^3
$$
and
$$
vol(\p B(0,s_2))=vol(\p B(0,s_2)\setminus N),
$$
where $N$ is the set of cut-locus w.r.t. $0$. Set $w=s_2$. Integrating on $B(0,w)$, we get 
$$
0\ge \int_{B(0,w)} \triangle_{ g_\infty}( f)=\oint_{\p B(0,w)} \langle\nabla_{ g_\infty} ( f), \nabla_{ g_\infty} r\rangle=\oint_{S_w} \langle\nabla_{ g_\infty} ( f), \nabla_{ g_\infty} r\rangle-vol(\p X_\infty)
$$
Hence we could find  a regular point $x$ on the sphere $\p S(0,w)$ such that 
$$
\langle \nabla_{ g_\infty} f,\nabla_{ g_\infty} r\rangle \le \varepsilon
$$
for some small $\varepsilon$ since $vol(\p B(0,w))\ge c w^3$.
Taking normalized radial geodesic connecting this point $\gamma(t)$ to the boundary $\p X$ , we set $l(t)=f(\gamma(t))$. Then 
$$
l'(s)=\langle \nabla_{ g_\infty} f,\nabla_{ g_\infty} r\rangle, \mbox{ and, }l''(s)= \nabla^2_{ g_\infty} f(\nabla_{ g_\infty} r,\nabla_{ g_\infty} r)=o(t^{-1})
$$
provided $s\in [t/2,t]$ with the large $t$. Thus, for any small $\varepsilon >0$ and for all $s\in [t/2,t]$ ($t$ depending on $\varepsilon$) one has
$$
l'(s)\le 2\varepsilon
$$limit
so that
$$
l(t)-l(t/2)\le t\varepsilon
$$
This contradicts of the claim (\ref{conformalfactor}) of the step 3 for the large $t$. Hence, we have finished  the proof of Theorem \ref{maintheorem}.
\end{proof}

\begin{rema} 
We can get more informations at the infinity by the strategy as in \cite{TV05}.
\end{rema}

\begin{proof}[Proof of   Corollary \ref{maincorollary0} ]
By the assumption (4), $\hat{g_i}$ is a compact family so that 
$$
\lim_{r\to 0}\sup_i\sup_x Vol(B(x,r)) =0
$$
Now the desired result follows from Theorem \ref{maintheorem} and the Dunford-Pettis' Theorem on the weak compactness in $L^1$. In fact, by Dunford-Pettis' Theorem, a subset in $L^1$ is relatively weakly compact if and only if it is uniformly integrable. Hence, relatively weak compactness of the family $\{S_i\}$ implies the condition (2) in Theorem \ref{maintheorem}.
\end{proof}

\begin{proof}[Proof of   Corollary \ref{maincorollary0bis}  ]
For $1<q<\infty$, we know the bounded set in $L^q$ is weakly compact in $L^q$ since $L^q$ is reflexive. Thus such set  is also weakly compact in $L^1$ and the  desired result follows from  Corollary \ref{maincorollary0}. When $q=\infty$, a bounded set in $L^q$ is also bounded in $L^2$ since $\{(M, \hat{g_i})\}$ is compact. Therefore, we prove the result.
\end{proof}

\begin{proof}[Proof of   Corollary  \ref{maincorollary01} ]
We argue by contradiction. Otherwise, we could find a sequence of conformally compact oriented Einstein metrics $(X, g_i^+)$ which satisfies the assumptions (1) and  (3-5) as in Theorem \ref{maintheorem}, 
$\oint_X|S_i|\to 0$ 
and whose compacitified metrics $(X,g_i)$ would blow up. Now $S$-tensor converge in $L^1$ and thus it is strongly compact in $L^1$. Therefore, it is weakly compact in $L^1$. It follows from Corollary \ref{maincorollary0} that it is a compact family of the compacitified metrics $(X,g_i)$. This contradiction gives the desired result.
\end{proof}

\begin{proof}[Proof of  Theorem \ref{maintheorem1}]
The proof is almost same as the one of Theorem \ref{maintheorem}. The only difference is to replace Lemma \ref{lemma4.3} by Lemma \ref{lemma4.3bis} for the boundary blow-up setting.
\end{proof}

\begin{proof}[Proof of   Corollary \ref{maincorollary1.0} ]
By the assumption (4), $\hat{g_i}$ is a compact family so that 
$$
\lim_{r\to 0}\sup_i\sup_x Vol(B(x,r)) =0
$$
Now the desired result follows from Theorem \ref{maintheorem1} and the Dunford-Pettis' Theorem on the weak compactness in $L^1$.
\end{proof}

\begin{proof}[Proof of   Corollary \ref{maincorollary1} ]
By the assumption $T_i$ is uniformly bounded from below, 
By the H\"older's inequality, we infer for any $x\in M$, $r>0$ and $i$
$$
\oint_{B(x,r)} (T_i)_-\le Vol(B(x,r))^{\frac{q}{q-1}}(\oint_{B(x,r)} ((T_i)_-)^q)^{1/q}\le C_6^{1/q}Vol(B(x,r))^{\frac{q}{q-1}}
$$
Here $(T_i)_-$ is the negative part of the $Q_3$ curvature. From the compactness of $\hat{g_i}$, we deduce that
$$
\liminf_{r\to 0}\inf_i\inf_{x\in M}\int_{B(x,r)}T_i\ge 0. 
$$
Finally, the desired result follows from Theorem \ref{maintheorem1}. Therefore, we prove the result.
\end{proof}

\begin{proof}[Proof of   Corollary \ref{maincorollary1.1} ]
This is a direct result of Theorem \ref{maintheorem1}. The proof is similar to the one of Corollary \ref{maincorollary01}.
\end{proof}

\appendix
\section{Cheeger-Gromov-Hausdorff Theory for manifolds with  boundary}
There is Cheeger-Gromov-Hausdorff theory for manifolds with boundary in the literature cf\cite{AKKLT, Kodani, Knox} etc. For the convenience of readers, we give a description.\\
Let $(X,g)$ be a Riemannian manifold with boundary, for $p\in X\setminus \p X$ define the interior injectivity radius of $p$, $i_{int}(p)$, to be the supremum over all $r > 0$ such that all unitary geodesics $\gamma : [0,t_\gamma] \to  X$ that start at $\gamma(0) = p$ are minimizing from $0$ to $\min\{t_\gamma,r\}$, where $t_\gamma$ is the first time the geodesic $\gamma$ intersects $\p X$. The interior injectivity radius of $M$ is defined as
$$
i_{int}(g) = \inf\{i_{int}(p)| p \in X\setminus \p X\}.
$$
The boundary injectivity radius of $p \in  \p X$ is defined by
$$
 i_\p(p) = \inf\{t | \gamma_p \mbox{ stops minimizing at } t\},
$$
where $\gamma_p$ is the geodesic in $M$ such that $\gamma_p'(0)$ is the inward unitary normal
tangent vector at $p$. The boundary injectivity radius of $X$ is defined by
$$
i_\p(g) = \inf\{ i_\p(p) | p\in\p X\}.
$$

Let us recall some definition about the harmonic radius for the manifolds with the boundary. Assume $X$ is a complete $4$-dimensional manifold with the boundary $\p X$. The local coordinates $(x_0,x_1,x_2,x_3)$ in $D$ around some interior point $p\in X\setminus \p X$ is called harmonic if $\triangle x^i=0$ for all $0\le i\le 3$. When $p\in \p X$, we need the coordinates $x_i$ harmonic and also $x_0|_{D\cap \p X}\equiv 0$ and $(x_1,x_2,x_3)|_{D\cap \p X}$ are also harmonic coordinates on the boundary. Given $\alpha\in (0,1)$ and $Q\in (1,2)$, we define  the harmonic radius $r^{1,\alpha}(Q)$ to be the biggest number $r$ satisfying the following properties:\\
1)if $dist(p, \p X)>r$, there is a neighborhood $D$ of $p$ in $X\setminus\p X$ and a coordinate chart $\varphi :B_{r/2}(0)\to D$ such that, in these coordinates for any tangent vector $\eta\in T_x X$
\beq
\label{relation1}
Q^{-2}|\eta|^2 \le g_{jk}(x)\eta^j\eta^k \le Q^{2}|\eta|^2
\eeq
and
\beq
\label{relation2}
r^{1+\alpha}\sup|x-y|^{-\alpha}|\p g_{jk}(x)-\p g_{jk}(y)|\le Q-1
\eeq
2)if $dist(p, \p X)\le r$, there is a neighborhood $D$ of $p$ in $\bar X$  and a coordinate chart $\varphi :B_{4r}^+(0)\to D$ such that $\{x^0=0\}$ maps to $\p X$ and relations (\ref{relation1}) and (\ref{relation2}) hold in these coordinates. \\

We recall a compactness result due to Anderson-Katsuda-Kurylev-Lassas-Taylor. Given $R_0,i_0,S_0, d_0\in (0,\infty)$, let us denote ${\mathcal M}(R_0,i_0,S_0, d_0)$ the class of compact, connected, 4-dimensional Riemannian manifolds with the boundary $(\bar X, g)$, with smooth metric tensor satisfying the following conditions:\\
\beq
\|Ric_X\|_{L^\infty(X)} \|\le  R_0, \|Ric_{\p X}\|_{L^\infty(\p X)} \le R_0
\eeq
where $Ric$ is the Ricci tensor;
\beq
\label{injec}
i_{int}(X)\ge i_0,  i(\p X)\ge i_0, i_\p(X)\ge 2i_0
\eeq
where $i(\p X)$ is the injectivity radius of the boundary manifolds $\p X$ with the induced metric;
\beq
\| H\|_{Lip(\p X)}\le S_0
\eeq
where $H$ is the mean curvature on the boundary;
\beq
diam(\bar X)\le d_0
\eeq

Let us denote ${\mathcal N}_*(\alpha, \rho,Q)$ the class of pointed manifolds $(\bar X,g,p)$ with $p\in \bar X$ such that the harmonic radius $r^{1,\alpha}(Q)\ge \rho$.

\begin{lemm}(Anderson-Katsuda-Kurylev-Lassas-Taylor)\\   
\label{AKKLT}
${\mathcal M}(R_0,i_0,S_0, d_0)$ is precompact in the $C^{1,\alpha}$  topology for any $\alpha\in (0,1)$. Moreover, there exists $\rho>0$ depending on $\alpha$ such that 
$$
{\mathcal M}(R_0,i_0,S_0, d_0) \subset {\mathcal N}_*(\alpha, \rho,Q).
$$
\end{lemm}
The second assertion in above Lemma still holds without the diameter assumption if we work with the pointed complete connected manifolds.\\

We now give the high order regularity of metric tensors for the manifold with the boundary, provided the boundary metric and curvature tensor are regular, that is,

\begin{lemm}\label{highorder}
Given $\alpha\in (0,1)$,  let $(X,g)$ be a complete regular $n$-dimensional Riemmannian metric with $C^{k+2,\alpha}$ totally geodesic boundary $\p X$ for some $k\ge 1$. Assume (\ref{injec}) holds and there exists some positive constant $R_1$ such that 
\beq
\label{curv}
\|  Ric\|_{C^{k,\alpha}(X)}\le R_1,\; \| \hat{g}\|_{C^{k+2,\alpha}(\p X)}\le R_1
\eeq
Then there exists some positive constant $C=C(i_0,R_1,n)$ such that 
$$
\|  g\|_{C^{k+2,\alpha}(X)}\le C
$$
\end{lemm}
For simplycity, we just prove the result for $4$-dimensional case. This part is some adaption of the result in \cite{AKKLT}. For the convenience of readers, we give the proof in details.
\begin{proof} 
Let $(x_0, x_1,x_2,x_3)$ be some harmonic coordinates on a neighborhood $D$ of some point $p\in \bar X$. From Lemma \ref{AKKLT}, we know the harmonic radius is bounded from below since we have lower bound for the injectivity radius and boundness of Ricci curvature. We work in such coordinates. Let $(g_{ij})_{0\le i,j\le 3}$ be metric matrix and its inverse matrix  $(g^{ij})_{0\le i,j\le 3}$. Denote $g=\det{g_{ij}}$ and $A_{ij}$ the determinant of $3\times 3$ matrix formed by omitting column $i$ and row $j$ from the matrix $(g_{ij})$. We have some elliptic PDE for  metric tensor
$$
\triangle g_{ij}=-2Ric_{ij}+{\mathcal P}_{ij}(g,\p g)
$$
where ${\mathcal P}(g,\p g)$ is a quadratic form in $\p g$  with coefficients that are rational functions of $g_{ij}$. Recall in harmonic coordinates, we could write Laplace-Betrami operator as follows
$$
\triangle u= g^{ij}\p_i\p_j u
$$
Again by Lemma \ref{AKKLT}, metric matrix $(g_{ij})$ and its invers $(g^{ij})$  are bounded in  the H\"older space $C^{1,\alpha}$ so that  the terms on the right hand side are bounded in the H\"older space $C^{0,\alpha}$. Let $p$ be some interior point on $X$. By the classical interior estimates \` a priori (see \cite{GT} Theorem 6.2), we get the boundness of the metric matrix $(g_{ij})$ in the H\"older space $C^{2,\alpha}$. Iterating the above procedure, we get the interior $C^{k+2,\alpha}$ estimates \` a priori. Now we treat the boundary case. Assume $p\in \p X$.  We use the above elliptic equations for the indexes $1\le i=\gamma,j=\beta\le 3$.  We note the boundary is totally geodesic so that $g_{\gamma\beta}=\hat{g}_{\gamma\beta}$ on $D\cap \p X$. Moreover, $g_{\gamma\beta}\in C^{k+2,\alpha}(D\cap \p X)$. Thus, we get the elliptic system with Dirichlet boundary conditions
$$
\left\{
\begin{array}{rllllll}
\triangle g_{\gamma\beta}&=&-2Ric_{\gamma\beta}+{\mathcal P}_{\gamma\beta}(g,\p g)&\mbox{in }D\\
g_{\gamma\beta}&=&\hat{g}_{\gamma\beta}&\mbox{on }D\cap \p X
\end{array}
\right.
$$
By Theorem 6.6 \cite{GT} , we infer the boundness of  $(g_{\gamma\beta})$ in the H\"older space $C^{2,\alpha}$.  Now we write elliptic PDE for $g^{0i}$ with Neumann boundary conditions. Let us denote $N=\frac{\nabla x_0}{|\nabla x_0|}$ the unit normal vector on the boundary $\p X$. In local coordinates, we can write on the boundary $\p X$
$$
N(u)=(g^{00})^{-1/2}g^{0j}\p_j u
$$
Note the boundary $\p X$ is totally geodesic so that on the boundary $D\cap \p X$ 
$$
N(g^{00})=-2H g^{00}=0
$$
and
$$
N(g^{0\gamma})=-H g^{0\gamma}+\frac12  (g^{00})^{-1/2}g^{\gamma j}\p_j g^{00}= \frac12  (g^{00})^{-1/2}g^{\gamma j}\p_j g^{00}
$$
We can write Ricci equation for the components $g^{0i}$ with Neumann boundary conditions
$$
\left\{
\begin{array}{llllll}
\triangle g^{0i}&=&-2Ric^{0i}+{\mathcal P}^{0i}(g,\p g)&\mbox{in }D\\
N(g^{00})&=&0&\mbox{on }D\cap \p X\\
N(g^{0\gamma})&=& \frac12  (g^{00})^{-1/2}g^{\gamma j}\p_j g^{00}&\mbox{on }D\cap \p X
\end{array}
\right.
$$
Using Theorem 6.30 \cite{GT}, we deduce the boundness of  $(g_{00})$ in the H\"older space $C^{2,\alpha}$. Going back the equation, we have $N(g^{0\gamma})\in C^{1,\alpha}(D\cap \p X)$. Again from Theorem 6.30 \cite{GT}, we deduce the boundness of  $(g^{0\gamma})$ in the H\"older space $C^{2,\alpha}$.  To see this, recall in harmonic coordinates, we have
$$
\frac 12(\delta_{ij})\le (g_{ij})\le 2 (\delta_{ij}),\; \frac 12(\delta_{ij})\le (g^{ij})\le 2 (\delta_{ij})
$$
Thus  for any $i,j\in\{0,1,2,3\}$, we have $\frac 12\le g_{ii},  g^{ii}\le 2$ and $|g_{ij}|\le 2$. Hence the coefficients in the boundary derivative $N$ are all in   $ C^{1,\alpha}(D\cap \p X)$ and $(g^{00})^{-1/2}g^{00} \ge \sqrt{1/2}$. On the other hand, the coefficients $g^{ij}$ in the Laplace-Betrami operator $\triangle =g^{ij}\p_i\p_j$ are in $ C^{0,\alpha}(D)$.  Thus the desired uniform estimates follows. 
From the fact $g=\det(g_{\gamma\beta})/g^{00}$, we obtain $g\in  C^{2,\alpha}(D)$. Therefore
$$
A_{0 i}=g g^{0i}\in  C^{2,\alpha}(D)
$$
Now we denote $(h_{\gamma\beta})=(g_{\gamma\beta})$ the $3\times 3$ matrix  and $(h^{\gamma\beta})$ the inverse matrix of $(h_{\gamma\beta})$, and $h=\det(h_{\gamma\beta})$. We remark for any $1\le \gamma\le 3$
$$
g_{\gamma 0}=(-1)^{3+\beta}\frac{ A_{0 \beta}}{hh^{\beta\gamma}}\in  C^{2,\alpha}(D)
$$
Finally, we have
$$
g_{0 0}=\frac{1-g_{ \gamma 0}g^{ \gamma 0}}{g^{0 0}}\in  C^{2,\alpha}(D)
$$
Now, iterating the above procedure, we prove the desired result $ g\in {C^{k+2,\alpha}(D)}$ since  the coefficients in the boundary derivative $N$ are all in   $ C^{k+1,\alpha}(D\cap \p X)$ and  the coefficients $g^{ij}$ in the Laplace-Betrami operator $\triangle =g^{ij}\p_i\p_j$ and the terms on the right hand side in Ricci equations are in $ C^{k,\alpha}(D)$. Thus, we finish the proof.
\end{proof}

\end{document}